\documentclass[preprint,12pt]{elsarticle}

\usepackage[utf8]{inputenc}   
\usepackage[T1]{fontenc} 
\usepackage{url}
\usepackage{booktabs,enumitem,appendix}
\usepackage{amsmath,amsfonts,amsthm,amssymb}
\usepackage{mathtools,leftidx}
\usepackage{csquotes}
\usepackage{graphicx}

\graphicspath{{.}{./figures/}}

\renewcommand{\leq}{\leqslant} 
\renewcommand{\geq}{\geqslant}
\renewcommand{\epsilon}{\varepsilon} 
\renewcommand{\hat}{\widehat}
\renewcommand{\bar}{\overline}
\renewcommand{\tilde}{\widetilde}
\def\1{\mbox{1\hspace{-.35em}1}}
\def\R{\mathbb{R}}

\def\N{\mathbb{N}}
\def\P{\mathbb{P}}
\def\E{\mathbb{E}}
\def\L{\mathbb{L}}

\def\C{\mathbb{C}}
\def\Y{\mathbb{Y}}
\def\Z{\mathbb{Z}}

\def\H{\mathbb{H}}
\def\rmd{\mathrm d}
\def\rme{\mathrm e}
\def\rmi{\mathrm i}

\def\GIG{\bG\centerdot\tilde \bI\centerdot \bG}%\leftidx{^G}{\bI}{^G}}
\def\GIGd{\bG\centerdot\mathcal I^{\bdexp}\centerdot \bG}%\leftidx{^G}{\boldsymbol{\mathcal I}}{^G}}
\def\GIGG{\bG\centerdot\mathcal I^{\bGexp}\centerdot \bG}%\leftidx{^G}{\boldsymbol{\mathcal I^{G}}}{^G}}

\def\bdexp{\mbox{\footnotesize $\bd$}}
\def\bDexp{\mbox{\footnotesize $\bD$}}
\def\bGexp{\mbox{\footnotesize $\bG$}}

\newcommand{\mV}{\mbox{${\mathcal V}$}}
\newcommand{\mW}{\mbox{${\mathcal W}$}}
\newcommand{\mY}{\mbox{${\mathcal Y}$}}
\newcommand{\mZ}{\mbox{${\mathcal Z}$}}
\newcommand{\mS}{\mbox{${\mathcal S}$}}
\newcommand{\mR}{\mbox{${\mathcal R}$}}

\newcommand{\bmV}{\mbox{$\boldsymbol{\mathcal V}$}}
\newcommand{\bmW}{\mbox{$\boldsymbol{\mathcal W}$}}
\newcommand{\bmY}{\mbox{$\boldsymbol{\mathcal Y}$}}
\newcommand{\bmZ}{\mbox{$\boldsymbol{\mathcal Z}$}}
\newcommand{\bmS}{\mbox{$\boldsymbol{\mathcal S}$}}
\newcommand{\bmR}{\mbox{$\boldsymbol{\mathcal R}$}}

\newcommand{\bff}{\mbox{${\boldsymbol{f}}$}}

\newcommand{\bd}{\mbox{${\mathbf d}$}}

\newcommand{\bV}{\mbox{${\mathbf V}$}}
\newcommand{\bv}{\mbox{$\boldsymbol{\upsilon}$}}

\newcommand{\bT}{\mbox{${\mathbf T}$}}
\newcommand{\bS}{\mbox{${\mathbf S}$}}

\newcommand{\bg}{\mbox{${\mathbf g}$}}

\newcommand{\bi}{\mbox{${\mathbf i}$}}
\newcommand{\bI}{\mbox{${\mathbf I}$}}

\newcommand{\bA}{\mbox{${\mathbf A}$}}
\newcommand{\bB}{\mbox{${\mathbf B}$}}
\newcommand{\bD}{\mbox{${\mathbf D}$}}
\newcommand{\bQ}{\mbox{${\mathbf Q}$}}
\newcommand{\bM}{\mbox{${\mathbf M}$}}

\newcommand{\bG}{\mbox{${\mathbf G}$}}
\newcommand{\bY}{\mbox{${\mathbf Y}$}}
\newcommand{\bX}{\mbox{${\mathbf X}$}}
\newcommand{\bZ}{\mbox{${\mathbf Z}$}}
\newcommand{\bW}{\mbox{${\mathbf W}$}}

\newcommand{\bomega}{\mbox{\boldmath$\omega$}}
\newcommand{\bOmega}{\mbox{\boldmath$\Omega$}}
\newcommand{\bLambda}{\mbox{\boldmath$\Lambda$}}

\newcommand{\bGamma}{\mbox{\boldmath$\Gamma$}}

\newcommand{\bSigma}{\mbox{\boldmath$\Sigma$}}
\newcommand{\bUpsilon}{\mbox{\boldmath$\Upsilon$}}

\newcommand{\bnu}{\mbox{\boldmath$\nu$}}

\newcommand{\bepsilon}{\mbox{\boldmath$\epsilon$}}
\newcommand{\bTheta}{\mbox{\boldmath$\Theta$}}
\newcommand{\bPhi}{\mbox{\boldmath$\Phi$}}

\newcommand{\barG}{\mbox{$\tilde{G}$}}
\newcommand{\bbarG}{\mbox{\boldmath$\tilde{G}$}}
\newcommand{\barR}{\mbox{$\tilde{R}$}}
\newcommand{\mj}{\mbox{$<\mathcal J>$}}

\newcommand{\argmin}{\displaystyle \mathop{\mbox{argmin}}}

\newcommand{\tend}{\displaystyle \mathop{\;\longrightarrow\;}}

\newcommand{\BigO}{O}
\newcommand{\abs}[1]{\left\lvert #1 \right\rvert}
\newcommand{\norm}[1]{\left\lVert #1 \right\rVert}

\DeclareMathOperator{\var}{Var}
\DeclareMathOperator{\cov}{Cov}

\DeclareMathOperator{\diag}{diag}
\DeclareMathOperator{\sign}{sgn}
\DeclareMathOperator{\trace}{trace}
\DeclareMathOperator{\Vect}{vec}
\newcommand{\vect}[1]{\Vect\!\left( #1 \right)}

\makeatletter
\newcommand{\pushright}[1]{\ifmeasuring@#1\else\omit\hfill$\displaystyle#1$\fi\ignorespaces}
\newcommand{\pushleft}[1]{\ifmeasuring@#1\else\omit$\displaystyle#1$\hfill\fi\ignorespaces}
\makeatother

\theoremstyle{plain}% default
\newtheorem{thm}{Theorem}%[section]
\newtheorem{lem}[thm]{Lemma}
\newtheorem{prop}[thm]{Proposition}
\newtheorem{cor}[thm]{Corollary}
\theoremstyle{definition}
\theoremstyle{remark}
\newtheorem{rem}{Remark}

\journal{Stochastic Processes and Applications}

\begin{document}

\begin{frontmatter}

\title{Asymptotic normality of wavelet covariances and multivariate wavelet Whittle estimators}

\author{Ir\`ene Gannaz}
\address{Univ Lyon, INSA Lyon, UJM, UCBL, ECL, ICJ, UMR5208, 69621 Villeurbanne, France}
\ead{irene.gannaz@insa-lyon.fr}

\begin{abstract}
Multivariate processes with long-range dependence properties can be encountered in many fields of application. Two fundamental characteristics in such frameworks are long-range dependence parameters and correlations between component time series. We consider multivariate long-range dependent linear processes, not necessarily Gaussian. We show that the covariances between the wavelet coefficients in this setting are asymptotically Gaussian. We also study the asymptotic distributions of the estimators of the long-range dependence parameter and the long-run covariance by a wavelet-based Whittle procedure. We prove the asymptotic normality of the estimators, and we provide an explicit expression for the asymptotic covariances. An empirical illustration of this result is proposed on a real dataset of rat brain connectivity.
\end{abstract}


\end{frontmatter}

\mathtoolsset{showonlyrefs}

\setlength{\parindent}{0pt}
\setlength{\parskip}{\baselineskip}

\section{Introduction}

Univariate long-range dependent processes are processes with an autocovariance function with a power-law decay or equivalently a spectral density diverging at the zero frequency with a power-law rate. Univariate long-range dependence (LRD) has encountered much interest and is used widely in applications. See, for example, \cite{PercivalWalden,Beran2016,Pipiras2017} and references therein.

Data is often recorded by multiple sensors where multivariate modeling brings better representation and can increase the consistency of inference.  Multivariate processes with LRD properties are found in a wide range of applications, such as geoscience \citep{WhitcherJensen00}, finance \citep{Gencay}, or neuroscience \citep{AchardGannaz}. Extensions of univariate LRD models to multivariate frameworks were initiated by \cite{Robinson95b}, and this topic has met great interest over the last decades. Several models have been proposed, such as multivariate autoregressive
fractionally integrated moving average (ARFIMA) models \citep{Lobato97,SelaHurvich2008,KechagiasPipiras}. In \cite{KechagiasPipiras}, Kechagias and Pipiras provide properties in the time and spectral domains of linear representations of multivariate long-range dependent processes. A nonlinear example of multivariate long-range dependent processes was also proposed by Didier and Pipiras in \cite{DidierPipiras}, where a multivariate Brownian motion was defined.

The specificity of the multivariate setting is that, in addition to LRD properties, it helps identify the correlation structure between the processes. The coupling between each component is characterized by the long-run covariance matrix \cite{Robinson95b}. 
 A key point for real data application is the development of statistical tests on LRD parameters and on long-run covariance. For example, as illustrated here on a real data example, these characteristics are intrinsically related to the brain activity recordings in neuroscience. Some work has shown that their distributions can be modified by  pathologies (see {\it e.g.} \cite{maxim.2005.1} and \cite{Achard_coma}). A statistical test may be useful to rigorously assess such observations.
%In real-data application, two main quantities are interesting: the LRD parameters, measuring the strength of long-range dependence for each time series, and the covariance structure between time series. This is captured by the so-called long-run covariance or fractal connectivity. 

We focus on semiparametric estimators, which are more robust to model misspecification~\citep{Robinson95b}. A common estimation procedure in this framework is Whittle estimation, which is based on a Fourier decomposition of the processes \citep{Lobato99, Shimotsu07, Nielsen11}. The authors prove the consistency and the asymptotic distribution of their estimators. More recently, the asymptotic normality of estimators has been provided by Baek et al.~\cite{Baek2020}, in prolongation of \cite{Robinson08}, in a multivariate framework where components can be co-integrated. An estimation with a Lasso penalty is also proposed in this setting by Pipiras et al. \cite{pipiras_lasso}, and D\"uker and Pipiras \cite{DukerPipiras} establish the asymptotic normality of this procedure.

As an alternative to Fourier, wavelet-based estimators can be used. Wavelet transforms are interesting especially because wavelet analysis performs an implicit differentiation, which offers the possibility to consider non-stationary processes. Wavelet-based Whittle estimation was introduced by Moulines et al. in \cite{Moulines08Whittle} for univariate long-range dependent time series. It was generalized to the multivariate setting by Achard and Gannaz \cite{AchardGannaz}. Estimators are consistent and have theoretical rates comparable to Fourier-based estimators. The numerical performances of wavelet-based and Fourier-based estimators are also similar, as illustrated in \cite{AchardGannaz_code}.

This paper considers linear processes, not necessarily Gaussian, with long-range dependence. The two main characteristics we are interested in are the long-range dependence parameters, which measure the LRD behavior of the processes, and long-run covariance, which captures the dependence structure between the components. Long-range dependence parameters and long-run covariance are estimated jointly with the procedure described in \cite{AchardGannaz}. The aim of this paper is to establish the asymptotic normality of these estimators. Roueff and Taqqu \cite{Roueff09asymptotic} prove that this asymptotic normality is acquired in the univariate setting, but no result exists for wavelet-based estimation in the multivariate setting. 

We first state that the sample covariances between the wavelet coefficients at a given scale are asymptotically Gaussian. We recover the univariate results of \cite{Roueff09asymptotic} but also provide the behavior of sample wavelet covariance between two processes, with possible different LRD parameters.

Asymptotic distributions of the estimators of long-range dependence parameters and long-run covariance of \cite{AchardGannaz} are then obtained. The results highlight that multivariate estimation of LRD parameters decreases variance with respect to an estimation of LRD parameters component by component. Long-run covariance can be estimated by the sample wavelet correlation at a unique scale or by wavelet-based Whittle procedure, which aggregates the sample wavelet covariance at numerous scales. We highlight that, not surprisingly, Whittle estimation converges at a better rate. We also prove the asymptotic normality of the Whittle estimator. Moreover, test procedures can be built from the asymptotic normality theorems, for LRD parameters, and for long-run covariance.
 
The paper is organized as follows. Section \ref{sec:model} introduces the specific framework of our study. The LRD properties of the processes are described, and assumptions on a linear representation of the time series are given. The properties of the wavelet representation of the processes are also synthesized. The asymptotic behavior of the covariance between wavelet coefficients is provided in Section~\ref{sec:cov}. Wavelet-based Whittle estimation is considered in Section~\ref{sec:d}. The asymptotic normality of the estimators is established. Section~\ref{sec:real} illustrates the asymptotic normality of the estimators on a real data example, with the study of functional magnetic resonance images (fMRI) of a dead rat and a live rat.

\section{The semiparametric multivariate long-range dependence framework} 
\label{sec:model}

Let $\bX=\{X_{a}(k),k\in\Z, a=1,\dots,p\}$ be a multivariate stochastic process. We consider a process $\bX$ with long-range dependence parameters $\bd=(d_1,d_2,\ldots,d_p)$. The stationary framework corresponds to LRD parameters $d_i\in(-1.2,1/2)$. In this case, following \cite{Lobato99, Shimotsu07}, we suppose that the cross-spectral density satisfies: for all $\lambda\in[-\pi,\pi]$,
$$\text{for all } (a,b)~,~f_{a,b}(\lambda)=\frac{1}{2\pi}\Omega_{a,b}(1-e^{-i\lambda})^{-d_a}(1-e^{i\lambda})^{-d_b}f_{a,b}^S(\lambda).$$
The functions $f_{a,b}^S(\cdot)$ correspond to the short-range dependence behavior of the process. This modelling is semiparametric since, if it imposes the LRD behavior, short-range dependence is left nonparametric through functions $f^S(\cdot)$. Some assumptions on $f^S(\cdot)$ are needed, which will be detailed below.

The LRD parameters, $\bd$, model the long-run dynamics of
the process. This model is a multivariate extension of a scalar
fractionally integrated process (the so-called I(d) process), and for any $a\in\{1,\dots,p\}$, the time series $X_a$ exhibits long-range dependence whenever $0<d_a<1/2$. The case $-1/2<d_a<0$ corresponds to antipersistence, where the spectral density $f_a(\cdot)$ tends toward 0 at the origin. The case $d_a=0$ is the weak-dependence case, where the spectral density $f_a(\cdot)$ tends toward a positive constant at the origin. See \cite{Lobato99,Shimotsu07}. For simplicity, the term LRD is used throughout the paper, regardless of the values of $\bd$.  

Wavelet analysis performs an implicit differentiation, which offers the possibility to consider non stationary processes, that is, LRD parameters $d_i$ possibly higher than $1/2$.
Let $\mathbb{L}$ denote the difference operator, $\mathbb{L} \bX(t)=\bX(t+1)-\bX(t)$. The $k$th difference operator, $\mathbb{L}^k$, $k\in\N$, is defined by $k$ recursive applications of $\mathbb{L}$. Introduce $\bD\in\N^p$. We suppose that the multivariate process $\bZ=\left\{\mathbb{L}^{D_{a}}X_a(k),\; k\in\Z,\; a=1,\dots,p\right\}$ is covariance stationary with a spectral density matrix given by, for all $\lambda\in[-\pi,\pi]$:
\begin{equation}
\text{for all } (a,b)\in\{1,\dots,p\}^2~,~f_{a,b}^{(D_a,D_b)}(\lambda)=\frac{1}{2\pi}\Omega_{a,b}(1-e^{-i\lambda})^{-d_a^*}(1-e^{i\lambda})^{-d_b^*}f_{a,b}^S(\lambda),
\end{equation}
where the long-range dependence parameters of $\bZ$ are given by $d_a^*\in(-1/2,1/2)$ for all $a=1,\dots, p$. 

Let the overline be the conjugate operator and $\circ$ be the Hadamard product. For any vector $\bv\in\R^p$, $\diag(\bv)$ stands for the $p\times p$ matrix with entries $\bv$ in the diagonal and 0 elsewhere. 

The LRD assumption can be expressed as follows:

\begin{enumerate}[label={(M1)}]
\item \label{ass:LRD}
The generalized spectral density of the multivariate process $\bX$ is, for all $\lambda\in[-\pi,\pi]$, 
\begin{equation}\label{eqn:density}
\bff(\lambda)=\bOmega\circ( \bLambda^0(\bd) \bff^S(\lambda) \overline{\bLambda^0(\bd)} ) ,\;  \text{~~with~~} \bLambda^0(\bd)=\diag((1-e^{-i\lambda})^{-\bdexp}),
\end{equation} 
where $\bd=\bD+\bd^*$, $\bD\in\N^p$, $\bd^*\in(-1/2, 1/2)^p$. 
\end{enumerate}

The matrix $\bOmega$ is called \textit{fractal connectivity} by~\cite{Achard08} or \textit{long-run covariance} matrix by \cite{RobinsonDiscussion}. Similar to \cite{Moulines07SpectralDensity, AchardGannaz} we introduce some regularity assumptions on the short-range dependence, modeled by function $\bff^S(\cdot)$.

The space $\mathcal H_p(\beta,L)$ is defined as the class of non-negative symmetric functions $\bg(\cdot): [\pi,\pi]\to \C^{p\times p}$ such that $\bg(0)=\mathbf{1}_{p\times p}$ and such that $$\sup_{\lambda\in(-\pi,\pi)}|\bg(\lambda)-\mathbf{1}_{p\times p}|\leq L |\lambda|^\beta,$$
with $\mathbf{1}_{p\times p}$ the $p\times p$ matrix with all entries equal to 1. We suppose that the following assumption is fulfilled:
\begin{enumerate}[label={(M2)}]
\item \label{ass:beta}
$\bff^S(\cdot)\in\mathcal H_p(\beta,L)$ with $0<\beta\leq 2$ and $0<L$. 
\end{enumerate}
Assumption \ref{ass:beta} imposes that $\bff^S(0)$ has constant entries equal to 1. This assumption is necessary to make $\bOmega$  identifiable in \ref{ass:LRD}.

When $\lambda$ tends toward 0, the spectral density matrix can be approximated at the first order by
\begin{equation}\label{eqn:approx}\bff(\lambda)\sim  \tilde\bLambda(\bd) \bOmega  \overline{\tilde\bLambda(\bd)}, \; \text{~~with~~} \tilde \bLambda(\bd)=\diag(|\lambda|^{-\bdexp}e^{-i\pi \bdexp/2}),
\end{equation} 
where $\sim$ means that the ratio of the left- and right-hand sides converges to one.

Lobato \cite{Lobato99} uses $\tilde \bLambda(\bd)=\diag(\lambda^{-\bdexp})$ as an approximation of $\bff(\cdot)$ whereas Shimotsu \cite{Shimotsu07} chooses to approximate $\bff(\cdot)$ using $\tilde \bLambda(\bd)=\diag(\lambda^{-\bdexp}e^{-i(\pi-\lambda) \bdexp/2})$, which corresponds to a second-order approximation due to the remaining term $\lambda$ in the exponential. We refer to \cite[Section 2.1]{AchardGannaz} and references therein for examples of processes satisfying approximation \eqref{eqn:approx}.

\subsection{Linear decomposition}

We suppose hereafter that the multivariate process admits a linear representation.

\begin{enumerate}[label=(M3)]
\item \label{ass:linear}
There exists a sequence $\{\bA^{(\bDexp)}(u)\}_{u\in\Z}$ in $\R^{p\times p}$ such that $\sum_{u\in\Z} \max_{a,b=1,\dots,p}|A_{a,b}^{(\bDexp)}(u)|^2<\infty$ and 
\[
\forall t\in\Z,\,~\bigl(\L^{D_{a}}X_a(t)\bigr)_{a=1,\dots,p}=\sum_{u\in\Z} \bA^{(\bDexp)}({t+u}) \bepsilon(u)
\]
with $\bepsilon(t)$ weak white noise process, in $\R^p$. Let $\mathcal F_{t-1}$ denote the $\sigma$-field of events generated by $\{\bepsilon(s), \,s\leq t-1\}$. Assume that $\bepsilon$ satisfies $\E[\bepsilon(t)|\mathcal F_{t-1}]=0$, $\E[\epsilon_a(t)\epsilon_b(t)|\mathcal F_{t-1}]=\1_{a=b}$ and $\E[\epsilon_a(t)\epsilon_b(t)\epsilon_c(t)\epsilon_d(t)|\mathcal F_{t-1}]=\mu_{a,b,c,d}$ with $|\mu_{a,b,c,d}|\leq \mu_\infty<\infty$, for all $a,b,c,d=1,\ldots,p$.
\end{enumerate}

Define for all $\lambda\in\R$, $\bA^{(\bDexp)\ast}(\lambda)=\sum_{t\in\Z} \bA^{(\bDexp)}(t)\rme^{\rmi\lambda\,t}$ the Fourier series associated to $\{\bA^{(\bDexp)}(u)\}_{u\in\Z}$. That is, $\bA^{(\bD)\ast}(\lambda)=\bigl(A^{(\bDexp)\ast}_{a,b}(\lambda)\bigr)_{a,b=1,\dots,p}$ with
\begin{equation}\label{eqn:Aast}
A^{(\bDexp)\ast}_{a,b}(\lambda)= (2\pi)^{-1/2}\sum_{t\in\Z}A^{(\bDexp)}_{a,b}(t)e^{-i\lambda t}\,,\quad \lambda\in\R\,.
\end{equation}
We add the following assumption:
\begin{enumerate}[label=(M4),topsep=-10pt]
\item \label{ass:Aast} For all $(a,b)\in\{1,\dots,p\}^2$, for all $\lambda\in\R$, the sequence $(2^{-j\,d_a}\lvert A^{(\bDexp)\ast}_{a,b}(2^{-j}\lambda)\rvert)_{j\geq 0}$ is convergent as $j$ goes to infinity.
\end{enumerate}
This assumption is necessary for technical reasons. It does not seem restrictive.

An example of a process that satisfies these assumptions is the causal multivariate linear representations with trigonometric power law coefficients proposed in \cite{KechagiasPipiras}.

\subsection{Wavelet representation}

\label{sec:W}

We introduce a discrete wavelet transform. Write $L^2(\R)$ the set of square-integrable functions with respect to the Lebesgue measure. Let $\phi(\cdot)$ and $\psi(\cdot)$ be two functions of $L^2(\R)$. Their Fourier transforms are given by $\hat \phi(\lambda)=\int_{-\infty}^\infty \phi(t)e^{-i\lambda t}dt$ and $\hat \psi(\lambda)=\int_{-\infty}^\infty \psi(t)e^{-i\lambda t}dt$, for all $\lambda\in\R$. We suppose that  $\phi(\cdot)$ and $\psi(\cdot)$ satisfy the following assumptions:
\begin{enumerate}[label=(W\arabic*),topsep=-10pt]
\item \label{ass:Wcompact}The functions $\phi(\cdot)$ and $\psi(\cdot)$ are integrable,  have compact supports, $\int_\R \phi(t)dt=1$ and $\int_\R \psi^2(t)dt=1$.
\item \label{ass:Wregularity}There exists $\alpha>1$ such that $\sup_{\lambda\in\R} |\hat \psi(\lambda)|(1+|\lambda|)^\alpha\,<\,\infty$.
\item \label{ass:Wmoments} The mother wavelet $\psi(\cdot)$ has $M>1$ vanishing moments.
\item \label{ass:Wpoly} The function $\sum_{k\in\Z}k^\ell\phi(\cdot -k)$ is polynomial with degree $\ell$ for all $\ell=1,\ldots,M-1$.
\item \label{ass:Wparameters}For all $i=1,\dots,p$, $(1+\beta)/2-\alpha\,<\,d_i\,\leq\, M$.
\end{enumerate}
Recall that $\beta$ in \ref{ass:Wparameters} is the regularity of the short-range dependence behavior introduced in \ref{ass:beta}.

These assumptions are the same as the ones considered in \cite{Moulines07SpectralDensity, Moulines08Whittle, AchardGannaz}. Assumptions \ref{ass:Wcompact}--\ref{ass:Wpoly} are usual when considering that $\phi(\cdot)$ and $\psi(\cdot)$ are respectively the scaling-function and the wavelet-function associated with a multiresolution analysis~\citep{Cohen}. They are satisfied, for example, by Daubechies wavelets. These wavelets are parametrized by the number of vanishing moments $M$. Assumption \ref{ass:Wregularity} holds with $\alpha$ an increasing function of $M$ going to infinity (see \cite{Daubechies}). Assumptions \ref{ass:Wcompact}--\ref{ass:Wparameters} are fulfilled by Daubechies wavelet basis with sufficiently large $M$. 

Assumption \ref{ass:Wmoments} implies that the wavelet transform performs an implicit differentiation of order $M$ and makes it possible to consider nonstationary processes. In Fourier analysis, tapering procedures are necessary to consider nonstationary frameworks, see {\it e.g.}~\cite{VelascoRobinson00, HurvichChen00}, and references therein.

At a given resolution $j\geq 0$, for $k\in\Z$, we define the dilated and translated functions $\phi_{j,k}(\cdot)=2^{-j/2}\phi(2^{-j}\cdot -k)$ and $\psi_{j,k}(\cdot)=2^{-j/2}\psi(2^{-j}\cdot -k)$. The wavelet coefficients of the process $\bX$ are defined by \[
\bW({j,k})=\int_\R \tilde{\bX}(t)\psi_{j,k}(t)dt\quad j\geq 0, k\in\Z,
\]
where $\tilde{\bX}(t)=\sum_{k\in\Z}\bX(k)\phi(t-k).$
For given $j\geq 0$ and $k\in\Z$, $\bW({j,k})$ is a $p$-dimensional vector $\bW({j,k})=\begin{pmatrix}
W_1({j,k}) & W_2({j,k}) & \dots & W_p({j,k})^T \end{pmatrix}$ where $W_a({j,k})= \int_\R \tilde{X_a}(t)\psi_{j,k}(t)dt$, $a=1,\dots,p$.
Throughout the paper, we adopt the same convention as in \cite{Moulines07SpectralDensity} and \cite{Moulines08Whittle}; that is, large values of the scale index $j$ correspond to coarse scales (low frequencies). The index $k$ is a location parameter, and $\bW(j,k)$ captures information at scale $j$ and location $k$ on the behavior of the process $\bX$.

In practice, let $\bX(1),\ldots \bX(N_X)$ denote the observations of the process $\bX$.  
Since the wavelets have a compact support, only a finite number $n_j$ of coefficients are non-null at each scale $j$. Suppose without loss of generality that the support of $\psi(\cdot)$ is included in $[0,T_{\psi}]$ with $T_{\psi}\geq 1$. For every $j\geq 0$, define \begin{equation}\label{eqn:nj} 
n_j:=  \max{(0,2^{-j}(N_X-T_{\psi}+1)-T_\psi+1)}.
\end{equation}
At each scale $j$, the non-zero coefficients belong to $\{\bW({j,k}), \, k=0,\dots,n_j\}$. 

Let $j_0$ be the minimal scale and $j_1 = j_0+\Delta $ the maximal scale which are considered in the estimation procedure. Following \cite{Moulines08Whittle, AchardGannaz}, the asymptotic behavior is given for $N_X$ and $j_0$ going to infinity. Results obtained in \cite{Moulines08Whittle, AchardGannaz,AchardGannaz_code} state that optimal rates in estimation are obtained when $j_0$ is high enough to remove the scales affected by low-range dependence. In practice, the number of scales $\Delta$ is finite. Yet, considering the asymptotic behavior, two cases may be distinguished: either the number of scales $\Delta$ is finite and fixed when $j_0$ goes to infinity, or $\Delta=j_1-j_0$ goes to infinity. The latter case seems natural, for example, when one takes all available scales above $j_0$ in estimation.

In the following, $n$ will denote the number of wavelet coefficients used for estimation and $<\mathcal J>$ the mean of scales, that is, \[
n=\sum_{j=j_0}^{j_1} n_j \text{ ~~and~~ } <\mathcal J>=\frac{1}{n}\sum_{j=j_0}^{j_1}n_j\,j\,.
\]
Define also
\begin{equation}
\label{eqn:eta}
\eta_\Delta := \sum_{u=0}^\Delta u \frac{2^{-u}}{2-2^{-\Delta}} \text{  ~~and ~~ }
\kappa_\Delta := \sum_{u=0}^\Delta (u-\eta_\Delta)^2 \frac{2^{-u}}{2-2^{-\Delta}}\,.
\end{equation}
These sequences converge respectively to 1 and to 2 when $\Delta$ goes to infinity \citep[Lemma 13]{Moulines08Whittle}.

%Let $j_0$ be the minimal scale and $j_1 = j_0+\Delta $ the maximal scale which are considered in the estimation procedure. Following \cite{Moulines08Whittle, AchardGannaz}, the asymptotic behavior is given for $N_X$ and $j_0$ going to infinity.
Moulines et al. state that under
assumptions \ref{ass:Wcompact}--\ref{ass:Wparameters}, the wavelet coefficient process $\{\bW(j,k),~ k\in\Z\}$ is covariance stationary for any given $j\geq 0$ \citep{Moulines07SpectralDensity}. Let 
\begin{align}
 D_{u,\tau}(\lambda;\delta)&= \sum_{t\in\Z}\lvert\lambda+2t\pi\rvert^{-\delta} \overline{\hat\psi(\lambda+2t\pi)}\,2^{u/2}\hat\psi(2^u(\lambda+2t\pi))\,\rme^{-\rmi 2^{u}\tau(\lambda+2t\pi)}\,,\\
\label{eqn:Dinf}\tilde D_{u,\infty}(\lambda;\delta) &= \sum_{\tau=0}^{2^{-u}-1} D_{u,\tau}(\lambda;\delta)\,.
\end{align}
Moulines, Roueff and Taqqu \cite{Moulines07SpectralDensity} establish that  $D_{u,\tau}(\lambda;\delta)$ is an approximation of the cross-spectral density between wavelet coefficients $\{\bW({j,k}), \,k\in\Z\}$ and  $\{\bW({j+u,2^uk+\tau}),\; \tau=0,\dots, 2^u-1,\, k\in\Z\}.$ The parameter $\delta$ captures the long-range dependence of the processes. Indeed, the cross-spectral density  of $(W_a({j,k}),W_b({j+u,2^jk'+\tau}))$ is approximated by $D_{u,\tau}(\lambda;d_a+d_b)$. Function $\tilde D_{u,\infty}(\lambda;\delta)$ allows us to consider between-scales dependence.

For $u\geq 0$, $(\delta_1,\delta_2)\in(-\alpha,M)^2$, define 
\begin{equation}
\label{eqn:Iu}
I_u(\delta_1,\delta_2)=   \int_{-\pi}^\pi\overline{\tilde D_{u,\infty}(\lambda;\delta_1)}\tilde D_{u,\infty}(\lambda;\delta_2)\, \rmd\lambda\,,
\end{equation} where $\tilde D_{u,\infty}(\lambda;\delta_2)$ is defined in \eqref{eqn:Dinf}. $I_u(\delta_1,\delta_2)$ will naturally appear when studying the covariance between sample wavelet covariances.

\section{Asymptotic normality of sample wavelet covariances and correlations}
\label{sec:cov}

Define $\hat \sigma_{a,b}(j)$ as the empirical covariance of the wavelet coefficients at a given scale $j\geq 0$, between components $a$ and $b$, and let  $\sigma_{a,b}(j)$ denote the theoretical covariance,
\begin{align*}
\hat \sigma_{a,b}(j)&=\frac{1}{n_j} \sum_{k=0}^{n_j-1}W_a({j,k})W_b({j,k}),\\
\sigma_{a,b}(j)&=\E[W_a({j,k})W_b({j,k})].
\end{align*}
Let $\hat \bSigma(j)=(\hat \sigma_{a,b}(j))_{a,b=1,\dots,p}$ and $\bSigma(j)=(\sigma_{a,b}(j))_{a,b=1,\dots,p}$ be the two associated matrices in $\R^{p\times p}$. In the following, for any matrix $\bM\in\C^{p\times p}$, the maximal entry will be denoted by $\norm{\bM}_\infty=\max_{a,b=1,\dots,p} \lvert M_{a,b}\rvert$.

Proposition 2 in \cite{AchardGannaz} proposes an approximation of the wavelet covariance at a given scale. It is recalled below.

\begin{prop}[\cite{AchardGannaz}]
\label{prop:approx}
Suppose assumptions \ref{ass:LRD}--\ref{ass:beta} and \ref{ass:Wcompact}--\ref{ass:Wparameters} hold.
For all $j\geq 0$, for all $\lambda\in(-\pi,\pi)$,
 \begin{equation}
 \label{eqn:theta}
 \norm{\bLambda(j)(\bd)^{-1}\bSigma(j)\bLambda(j)(\bd)^{-1} - \bG}_\infty\leq C L 2^{-\beta\,j}.
 \end{equation}
with constant $C$ depending on $\beta$, $\min_\ell d_\ell$, $\max_\ell d_\ell$, $\max_{\ell,m}\abs{\Omega_{\ell,m}}$, $\phi$ and $\psi$ and\begin{align}
\label{eqn:lambda}
\bLambda_j(\bd)&=\diag\bigl(2^{j\bdexp}\bigr),\\
\label{eqn:G0}
G_{a,b}&=\Omega_{a,b} \cos(\pi(d_a-d_b)/2)K(d_a+d_b),\quad a,b=1,\dots,p,\\
K(\delta)&=\int_{-\infty}^\infty \lvert\lambda\rvert^{-\delta}\lvert\hat\psi(\lambda)\rvert^2\rmd\lambda, \quad\delta\in(-\alpha,M).
\end{align}
\end{prop}

For $u\geq 0$, $(\delta_1,\delta_2)\in(-\alpha,M)^2$, let us introduce $\tilde I_u(\delta_1,\delta_2)$ as \begin{equation}
\label{eqn:Itilde}
\tilde I_u(\delta_1,\delta_2)=2\pi\,\frac{I_{u}(\delta_1,\delta_2)}{K(\delta_1)K(\delta_2)}\,,
\end{equation}
with $I_u(\delta_1,\delta_2)$ defined in \eqref{eqn:Iu}. We also define ${\GIG(u)}\in\R^{p^2\times p^2}$ as:
\begin{multline}
\label{eqn:Itilde_mat}
\GIG(u,u')=\diag\Bigl(\vect{\bLambda_{u\wedge u'}(\bd)^{-1}\bG\bLambda_{u\wedge u'}(\bd)^{-1}}\Bigr)
\\
(\tilde \bI_{\lvert u-u'\rvert} (d_a+d_b, d_{a'}+d_{b'})_{(a,b),(a',b'))\in\{1,\dots,p^2\}}\diag\Bigl(\vect{\bLambda_{u\wedge u'}(\bd)^{-1}\bG\bLambda_{u\wedge u'}(\bd)^{-1}}\Bigr)\,.
\end{multline}

\begin{rem} Observe that 
\[\tilde I_0(\delta_1,\delta_2)=\frac{2\pi\,\int_{-\pi}^\pi g_\psi(\lambda;\delta_1)g_\psi(\lambda;\delta_2)\rmd\lambda}{\left(\int_{-\pi}^\pi g_\psi(\lambda;\delta_1)\rmd\lambda\right)\left(\int_{-\pi}^\pi g_\psi(\lambda;\delta_2)\rmd\lambda\right)}
\]
where $g_\psi(\lambda;\delta)=\sum_{t\in\Z}\lvert \lambda+2t\pi\rvert^{-\delta} \lvert\hat\psi(\lambda+2t\pi)\rvert^2$. It is straightforward that  $\tilde I_0(\delta_1,\delta_2)\leq 2\pi$. Cauchy-Schwarz's inequality on the denominator also provides $\tilde I_0(\delta_1,\delta_1)\geq 1$. %Using Jensen's inequality, 
%\[
%\tilde I_0(\delta_1,\delta_1)=\frac{\frac{1}{2\pi}\,\int_{-\pi}^\pi g_\psi(\lambda;\delta_1)^2\rmd\lambda}{\left(\frac{1}{2\pi}\int_{-\pi}^\pi g_\psi(\lambda;\delta_1)\right)^2}\geq 1.
%\]
\end{rem}

 Here and subsequently, $\tend^{\mathcal L}$ denotes a convergence in distribution. The asymptotic distribution of the sample wavelet covariance process is given in the following theorem.

\begin{thm}
\label{thm:gauss}
For all $j_0\geq 0$, $u\geq 0$, define \begin{align*}
{\hat{\bT}(j_0+u)}&=\vect{2^{-(j_0+u)(d_a+d_b)} \hat \sigma_{a,b}(j_0+u),\; a,b=1,\dots,p}\\
\vec{\bG}&=\vect{G_{a,b},\; a,b=1,\dots,p}
\end{align*}
where $\vect{\bM}$ denotes the operation which transforms a matrix $\bM\in\R^{p_1\times p_2}$ in a vector of $\R^{p_1 p_2}$.
Suppose assumptions {\ref{ass:LRD}--\ref{ass:Aast}} and { \ref{ass:Wcompact}--\ref{ass:Wparameters} } hold. Let $2^{-j_0\beta}\to 0$ and $N_X^{-1}2^{j_0}\to 0$.
Then for all $\Delta\in\N$,
\begin{equation}
\left\{\sqrt{n_{j_0+u}}\left(\hat \bT(j_0+u)- \vec{\bG}\right), ~ u=0,\dots,\Delta\right\}\tend^{\mathcal L}_{j_0\to\infty} \{\bQ(u),\,u=0,\dots,\Delta\},
\end{equation}
where $\bQ(\cdot)$ is a centered Gaussian process with covariance function $\cov{(Q_{a,b}(u),\,Q_{a',b'}(u'))}=
V_{(a,b),(a',b')}(u,u')$ where
\begin{multline}
\label{eqn:variance}
V_{(a,b),(a',b')}(u,u')= 2^{-|u-u'|/2}\,\Bigl(\GIG_{(a,a'),(b,b')}(u,u')+\GIG_{(a,b'),(a',b)}(u,u')\Bigr), 
\end{multline}
and $\GIG(u,u')$ is defined in \eqref{eqn:Itilde_mat}.
\end{thm}
The proof is given in \ref{proof:thm:gauss}. It is similar to the one of the univariate setting given in \cite[Theorem 2]{Roueff09asymptotic}. It relies on decimated processes and limit theorems developed in \cite{Roueff09central}.

\begin{rem}
In the univariate setting, we obtain the same result as \cite[Theorem 2]{Roueff09asymptotic}. The authors use a different normalization, by $\sqrt{N_X2^{-j_0}}$ rather than by $\sqrt{n_{j_0+u}}$. The correspondence between the results follows first from the equivalence $\sqrt{n_{j_0+u}}\sim \sqrt{N_X2^{-j_0-u}}$ and second from approximation \eqref{eqn:theta} for the expectancy term. 

The main difference in the multivariate case is that the LRD properties in two processes can be different. It introduces a bias term through the presence of the cosinus term in $\bG$ \eqref{eqn:G0} and slightly modifies the variance through terms $\tilde \bI_u(\cdot,\cdot)$ \eqref{eqn:Itilde_mat}.
\end{rem}

\begin{rem}
In \cite{WhitcherGuttorpPercival}, Whitcher et al. establish the asymptotic normality for wavelet correlations of bivariate multivariate time series with long-range dependence. The advantage of Theorem~\ref{thm:gauss} is to provide an explicit form of the asymptotic variance.
\end{rem}

\begin{rem}
As already pointed out by Roueff and Taqqu \cite{Roueff09asymptotic}, the covariance of the wavelet coefficients involves between-scales correlations which do not vanish when the sample size goes to infinity. This fact contrasts with the behavior of Fourier periodogram or Fourier-based Whittle estimation. In the variance formulation \eqref{eqn:variance}, these correlations appear through quantities $\{\tilde I_u(\delta_1,\delta_2),\; u\geq 0,\; (a,b)\in\{1,\dots,p\}^2\}$.
\end{rem}

We can deduce the asymptotic normality for sample wavelet correlations by means of delta method. We do not present here the multivariate result for the sake of brevity, except when the matrix $\bG$ is diagonal since formulas are more simple. We focus on the pointwise result to highlight the specificity of our setting. 

\begin{cor}
\label{cor:cor}
Let $(a,b)\in\{1,\dots,p\}^2$, $a\neq b$, and $j\geq j_0\geq 0$. Define \[
\hat \rho_{a,b}(j)=\frac{\hat \sigma_{a,b}(j)}{\sqrt{\hat \sigma_{a,a}(j)\hat \sigma_{b,b}(j)}}\text{  and  }r_{a,b}=\frac{G_{a,b}}{\sqrt{G_{a,a} G_{b,b}}}.
\]
Then, under conditions of Theorem~\ref{thm:gauss},
\[
\sqrt{n_j}\,\left(\hat \rho_{a,b}(j)-r_{a,b}\right)\tend^{\mathcal L}_{j\to\infty} \mathcal{N}\left(0,V^{(\mathbf{\rho})}_{a,b}\right)
\]
with
\begin{multline}
\label{eqn:var_rho}
V^{(\mathbf{\rho})}_{a,b}=\left(\tilde I_0(2d_a,2d_b)+\tilde I_0(d_a+d_b,d_a+d_b)(r_{a,b}^2+r_{a,b}^4)\right.\\\left.-( \tilde I_0(2d_a,2d_b) + \tilde I_0(2d_b,d_a+d_b))\,2\, r_{a,b}^2\right. \\\left.- (\tilde I_0(2d_a,2d_a)+ \tilde I_0(2d_b,2d_b))\,r_{a,b}^2/2\right).
\end{multline}
When all off-diagonal entries of $\bG$ are equal to 0, \begin{multline*}
\sqrt{n_j}\,\vect{\hat \rho_{a,b}(j), ~1\leq a<b\leq p}\\
\tend^{\mathcal L}_{j\to\infty} \mathcal{N}_{p(p-1)/2}\left(0,\diag\Bigl(\vect{2^{-j(2d_a+2d_b)}\tilde I_0(2d_a,2d_b), ~1\leq a<b\leq p}\Bigr)\right)\,.
\end{multline*}
\end{cor}

\begin{rem} Let $\hat \rho$ denote the sample correlation of a bivariate-Gaussian-distributed $n$-sample with correlation $\rho$. Then  $\sqrt{n}(\hat \rho-\rho)\tend^{\mathcal L}_{n\to\infty} \mathcal{N}\left(0,(1-\rho^2)^2\right)$; see \emph{e.g.} \cite[Theorem 4.2.4]{Anderson}. When parameters $d_a$ and $d_b$ are equal, Corollary \ref{cor:cor} entails that
\[
\sqrt{n_j}\,(\hat \rho_{a,b}(j)-r_{a,b})\tend^{\mathcal L}_{j\to\infty} \mathcal{N}\left(0,\tilde I_0(2d_a,2d_a)(1-r_{a,b}^2)^2\right)\,.
\] 
We recover a similar form of the asymptotic distribution, up to a normalization constant.
\end{rem}

\begin{rem}
In \cite[Section 4.2.]{WhitcherGuttorpPercival} Whitcher et al. use the convergence of the Fisher transform of $\hat\rho_{a,b}(j)$ to a standard Gaussian distribution at a rate $\sqrt{n_j}$, when the correlation $r_{a,b}(j)$ is equal to zero.
 The result is true if we suppose that between-scale wavelet coefficients are independent, which is asymptotically satisfied when the regularity of the wavelet goes to infinity \citep{decorrelate}. Corollary~\ref{cor:cor} illustrates that an additional normalization by $\tilde I_0(2d_a,2d_b)^{-1/2}$ of $\hat\rho_{a,b}(j)$ is necessary. 

Some computed values of $\tilde I_0(2d_a,2d_b)$ are displayed in Table~\ref{tab:I}. It shows that $\tilde I_0(2d_a,2d_b)$ indeed decreases when the regularity increases. But between-scale wavelet coefficients dependence may not be negligible if the regularity is not high enough.
For example, in the absence of long-range dependence, when $d_a=d_b=0$, $\tilde I_0(0,0)=1.62$ for Daubechies wavelets with $M=4$ vanishing moments.
 Hence, in real data application, the approximation in \cite{WhitcherGuttorpPercival} may lead to false positives.
\end{rem}

\begin{table}[!ht]
%\begin{tabular}{lcccccccccc}
%$\bd=(d_1,d_2)$ & (0,0)&(0,0.05)&(0,0.10)&(0,0.15)&(0,0.20)&(0,0.25)&(0,0.30)&(0,0.35)&(0,0.40)& (0,0.45)\\
%%$\tilde I_0(d_1,d_2)$& 1.43 & 1.68 & 1.96 & 2.29 & 2.66 & 3.08 & 3.55 & 4.08  & 4.66 & 5.31\\
%$\tilde I_0(2d_1,2d_2)$&  1.43 &1.55& 1.68 &1.81& 1.95& 2.09 &2.24& 2.38& 2.53 &2.67
%\end{tabular}
\centering
\begin{tabular}{llccccc}
%\toprule
&&\multicolumn{5}{c}{$\bd=(d_1,d_2)$} \\ 
& & (0, 0)&(0, 0.1)&(0, 0.2)&(0, 0.3)&(0, 0.4)\\
  \midrule
M= 1  &	 $\alpha$= 1.00 	 &  5.43 &  5.56 &  5.67 &  5.77 &  5.86\\ 
M= 2  &	 $\alpha$= 1.34 	 &  2.65 &  2.66 &  2.67 &  2.68 &  2.69\\ 
M= 3  &	 $\alpha$= 1.64 	 &  1.85 &  1.86 &  1.86 &  1.86 &  1.87\\ 
M= 4  &	 $\alpha$= 1.91 	 &  1.62 &  1.62 &  1.61 &  1.61 &  1.61\\ 
M= 5  &	 $\alpha$= 2.18 	 &  2.05 &  2.04 &  2.04 &  2.03 &  2.02\\ 
M= 6  &	 $\alpha$= 2.43 	 &  1.90 &  1.91 &  1.92 &  1.93 &  1.94\\ 
M= 7  &	 $\alpha$= 2.68 	 &  1.22 &  1.23 &  1.24 &  1.25 &  1.26\\ 
M= 8  &	 $\alpha$= 2.93 	 &  1.01 &  1.01 &  1.01 &  1.01 &  1.01\\ 
\bottomrule
\end{tabular}
\caption{Values of $\tilde I_0(2d_1,2d_2)$ with respect to $\bd=(d_1,d_2)$  for Daubechies's wavelets with different values of vanishing moments $M$ in \ref{ass:Wmoments}. Parameter $\alpha$ characterizes the regularity of the wavelets in \ref{ass:Wregularity}.}
\label{tab:I}
\end{table}

\begin{rem}
The wavelet correlation at a given scale is also known as \textit{wavelet coherence}. It is used in some applications, as in environmental studies by \cite{WhitcherGuttorpPercival}, or in neurosciences in \cite{spie_test}. In such real data applications, the crucial point is the use of test procedures. In particular, the test of the nullity of the correlations is essential. Corollary~\ref{cor:cor} shows that the asymptotic distribution depends on the long-range dependence parameters $\bd$. Plugging in a consistent estimator of parameter $\bd$ in \eqref{eqn:var_rho} allows for a test procedure to be built. For instance,   one can use the wavelet Whitlle estimator described in Section~\ref{sec:d} below. 
\end{rem}

\section{Asymptotic normality of the parameters estimates}
\label{sec:d}

For clarity, the true parameters are denoted with an exponent 0 in this part.

The wavelet-based local Whittle procedure proposes to estimate the parameters by maximizing a pseudo-likelihood given by a Gaussian approximation of the wavelet coefficients $\{\bW(j,k), j\geq 0, k=0,\dots,n_j\}$. Moulines et al. \citep{Moulines08Whittle} and Achard and Gannaz \citep{AchardGannaz} prove
that the wavelet-based Whittle approximation provides consistent estimators even for non-Gaussian processes. The Whittle procedure can also be applied in multivariate cases, which is not possible for example with the regression of the wavelet log-scalogram \cite{AbryVeitch98,Achard08}.

Let $\hat \bd$ and $\hat \bOmega$ be the wavelet Whittle estimators as defined in \cite[Section 3.3]{AchardGannaz}. They maximize the objective function
\begin{multline} \label{eqn:whittle}
\mathcal L(\bG(\bd),\bd) = \frac{1}{n}\sum_{j=j_0}^{j_1} \Bigl[ n_j \log\det\left(\bLambda_j(\bd)\bG(\bd)\bLambda_j(\bd)\right)\\+\sum_{k=0}^{n_j-1} \bW_{j,k}^T\left(\bLambda_j(\bd)\bG(\bd)\bLambda_j(\bd)\right)^{-1}\bW_{j,k}\Bigr],
\end{multline}
where the superscript $T$ denotes the transpose operator and $\bLambda_j(\bd)$ and the matrix $\bG(\bd)$ are defined respectively in \eqref{eqn:lambda} and \eqref{eqn:G0}.

The function $\mathcal L(\cdot,\cdot)$ corresponds to the negative log-likelihood of $\{\bW(j,k), j\geq 0, k=0,\dots,n_j\}$ under a Gaussian assumption, where Proposition~\ref{prop:approx} is used for a parametrization of the variance at each scale.
The estimation of the vector of long-range dependence parameters $\bd$ satisfies $\hat \bd=\argmin_{\bdexp\in\R^p} R(\bd)$, with
\begin{equation}
\label{eqn:R}
R(\bd)= \log\det(\hat \bG(\bd)) + 2\log(2)\left(\frac{1}{n}\sum_{j=j_0}^{j_1} j\, n_j\right)\left(\sum_{\ell=1}^p d_\ell\right).
\end{equation}
The covariance matrix $\bOmega$ is estimated by \begin{gather}\label{eqn:Ohat}\hat\Omega_{a,b}=\hat G_{a,b}(\hat \bd)/(\cos(\pi(\hat d_a-\hat d_b)/2) K(\hat d_a+\hat d_b)),\; a,b=1,\dots,p,\\
\label{eqn:G}\text{where  }
\hat \bG(\bd) =\frac{1}{n} \sum_{j=j_0}^{j_1} n_j \bLambda_j(\bd)^{-1}
\hat \bSigma(j)\bLambda_j(\bd)^{-1}.
\end{gather}

We introduce 
\begin{align}
\label{eqn:Idelta}
&\mathcal I^d_\Delta(\delta_1,\delta_2)= \frac{2}{\kappa_\Delta}\tilde I_0(\delta_1,\delta_2)\\
\nonumber & \pushright{+\frac{2}{\kappa_\Delta^2}\sum_{u=1}^{\Delta}(2^{u\delta_1}+2^{u\delta_2})\,2^{-u} \frac{2-2^{-\Delta+u}}{2-2^{-\Delta}}((u+\eta_{\Delta-u}-\eta_\Delta)(\eta_{\Delta-u}-\eta_\Delta)+\kappa_{\Delta-u})\,\tilde I_u(\delta_1,\delta_2)}\\
 &\pushright{\text{ if } \Delta<\infty,} \end{align}
 \begin{equation}
\label{eqn:Iinf}
\mathcal I^d_\infty(\delta_1,\delta_2)= \tilde I_0(\delta_1,\delta_2)+ \sum_{u=1}^{\infty}(2^{u\delta_1}+2^{u\delta_2})\,2^{-u}\,\tilde I_u(\delta_1,\delta_2)\;,\quad {\text{if } \Delta=\infty.}
\end{equation}
 Define also
\begin{multline}
\label{eqn:Itilded_mat}
\GIGd(\Delta)=\diag\Bigl(\vect{\bG^0}\Bigr)
\bigl(\mathcal I^d_\Delta(d_a^0+d_b^0, d_{a'}^0+d_{b'}^0)_{(a,b),(a',b')\in\{1,\dots,p^2\}}\bigr)\diag\Bigl(\vect{\bG^0}\Bigr)\,.
\end{multline}

The asymptotic normality of the estimator of the long-range dependence parameters is established by our next theorem.

\begin{thm}
\label{thm:d}
Suppose assumptions {\ref{ass:LRD}--\ref{ass:Aast}} and { \ref{ass:Wcompact}--\ref{ass:Wparameters}} hold.
Let $j_0 < j_1 \leq j_N$ with $j_N=\max\{j,n_j\geq 1\}$ such that \[
j_1-j_0\to \Delta\in\{1,\dots,\infty\},\; \log(N_X)^2(N_X 2^{-j_0(1+2\beta)}+ N_X^{-1/2} 2^{j_0/2})\to 0.
\] 
Then $\sqrt{n}(\hat\bd -\bd^0)$ converges in distribution to a centered Gaussian distribution with a variance equal to \begin{equation}
\label{eqn:vard}
\bV^{(\bdexp)}(\Delta)=\frac{1}{2\,\log(2)^{2}}(\bG^{0-1}\circ \bG^0+\bI_p)^{-1}\, \bUpsilon(\Delta)\, (\bG^{0-1}\circ \bG^0+\bI_p)^{-1},
\end{equation}
where $\bI_p$ is the identity matrix in $\R^{p\times p}$ and with entry $(a,a')$ of $\bUpsilon^{(\Delta)}$, for $(a,a')\in\{1,\dots,p\}^2$, given by
\begin{equation}
\label{eqn:W}
\Upsilon_{a,a'}{(\Delta)}=\sum_{b,b'=1,\dots,p} (G^{0-1})_{a,b}(G^{0-1})_{a',b'} \bigl(\GIGd_{(a,a'),(b,b')}(\Delta)+\GIGd_{(a,b'),(a',b)}(\Delta)\bigr)
\end{equation}
where quantities $\GIGd(\Delta)$ are defined by \eqref{eqn:Itilded_mat}.
\end{thm}
The proof is given in \ref{proof:thmd}.

\begin{rem}
In the univariate setting, we recover \cite[Theorem 5]{Roueff09asymptotic}, using the equality \[
\sum_{v=0}^{\Delta-u} \frac{2^{-v}}{2-2^{-\Delta}}(v-\eta_\Delta)(u+v-\eta_\Delta)= \frac{2-2^{-\Delta+u}}{2-2^{-\Delta}}((u+\eta_{\Delta-u}-\eta_\Delta)(\eta_{\Delta-u}-\eta_\Delta)+\kappa_{\Delta-u}),
\]
in \eqref{eqn:Idelta}.
Observe that the result is also normalized by $\sqrt{n}$ rather than $\sqrt{N_X2^{-j_0}}$.
\end{rem}

\begin{rem}
The condition on $j_0$ and $j_1$, that is, $\log(N_X)^2(N_X2^{-j_0(1+2\beta)}+ N_X^{-1/2} 2^{j_0/2})\to 0$ is more restrictive than the condition required for the consistency of the estimators given in \cite[Theorem 6]{AchardGannaz}. Roueff and Taqqu \cite[Theorem 5]{Roueff09asymptotic} obtain a similar result in the univariate setting.  As illustrated in \cite{AchardGannaz_code}, the condition $\log(N_X)^2\,N_X^{-1/2} 2^{j_0/2}\to 0$ means that the highest frequencies, which are affected by the short-range dependence, should be removed from the estimation. The additional condition $\log(N_X)^2\,N_X2^{-j_0(1+2\beta)}\to 0$ prevents us from choosing the scale $j_0=N_X^{1/(1+2\beta)}$ giving the minimax rate \citep[Corollary 7]{AchardGannaz}. Yet a near minimax rate is possible, with only a logarithmic lost, choosing, for example, $j_0=\log(N_X)^3\,N_X^{1/(1+2\beta)}$.
\end{rem}

\begin{rem}
If the vector $\bd^0$ has all entries equal to $d^0$, the resulting covariance is \[
\frac{1}{4\log(2)^{2}}\mathcal I_\Delta{(2d^0,2d^0)}(\bG^{0-1}\circ \bG^0+\bI_p)^{-1}\,.
\]
We recognize a form of asymptotic variance similar to the ones given by \cite{Lobato99}, \cite{Shimotsu07} and \cite{DukerPipiras}  with Fourier-based Whittle estimators. Note that they use a different approximation of spectral density at zero frequency. Lobato \cite{Lobato99} and Shimotsu \cite{Shimotsu07} consider respectively $G_{a,b}^0=\Omega_{a,b}\rme^{\rmi\pi(d_a-d_b)/2}$ and $\bG^0=\bOmega$.
D\"uker and Pipiras's modelling in \cite{DukerPipiras} is more general and does not suppose a linear representation of the time series. Their result is valid for a general form of matrix $\bOmega$. Additionally, Baek et al. \cite{Baek2020} establish asymptotic normality of estimators in a bivariate model with possible co-integration.
\end{rem}

\begin{rem}Consider the bivariate setting with $\bOmega=\begin{pmatrix} 1 & \rho \\ \rho & 1 \end{pmatrix}$ and $d_1=d_2=d$. Let $\hat d_1^U$ and $\hat d_2^U$ be the wavelet Whittle estimators obtained by separately considering the components $\{X_1(k), k=1,\dots,N_X\}$ and $\{X_2(k), k=1,\dots,N_X\}$ in. That is, for $i=1,2$,
\[
\hat d_i^U=\displaystyle \argmin_{d_i\in\R} R_i(d_i) \text{  with  } R_i(d_i)=\log\Bigl(\frac{1}{n}\sum_{j=j_0}^{j_1} n_j2^{-2jd_i}\hat \sigma_{ii}(j)\Bigr) + 2\log(2)\Bigl(\frac{1}{n}\sum_{j=j_0}^{j_1} j\,n_j\Bigr)\,d_i.
\]
According to Theorem \ref{thm:d} $\hat d_i^U$, $i=1,2$, are asymptotically normal, with the same asymptotic variance $\sigma^2(d,j_1-j_0)=V^{(d)}(j_1-j_0)$, given by \eqref{eqn:vard}.\\
Let now $\hat \bd$ be the bivariate wavelet Whittle estimator defined in \eqref{eqn:R}. Theorem \ref{thm:d} provides the asymptotic normality of $\hat d$ with the asymptotic variance given by \eqref{eqn:vard}, which is equal to \[
\bV^{(\bdexp)}(j_1-j_0)=(\bG^{0-1}\circ \bG^0+\bI_p)^{-1} \,2\,\sigma^2(d,j_1-j_0)=\begin{pmatrix}
1-\rho^2/2 & \rho^2/2 \\ \rho^2/2 & 1-\rho^2/2
\end{pmatrix}\sigma^2(d,j_1-j_0).
\]
This result proves that we reduce the entrywise variance when we perform multivariate estimation instead of univariate estimation. A similar conclusion was obtained for Fourier-based estimation by \cite{Lobato99} and \cite{Nielsen11}. Achard and Gannaz \cite{AchardGannaz_code} support this assertion on simulated data. In real data application, \cite{avignon} also establishes that the multivariate approach performs better than the univariate one, comparing  their application on fMRI data where subjects were scanned twice.
\end{rem}

\begin{rem}
Quantities $\mathcal I_\Delta(\delta_1,\delta_2)$ are computable for given $\delta_1, \delta_2, \Delta$. Hence, plugging in \eqref{eqn:vard}--\eqref{eqn:W} consistent estimators of $\bd$ and $\bG$, for example $\hat\bd$ and $\hat\bG(\hat\bd)$ \citep[Theorem 6]{AchardGannaz}, a test procedure on parameters $\bd$ can be built. 
\end{rem}

We now study the asymptotic behavior of the estimation of long-run covariance. We show the asymptotic normality of $\hat\bG(\hat\bd)$, defined in \eqref{eqn:G}.

Write
\begin{align*}
\mathcal I_\Delta^G(\delta_1,\delta_2) & = \tilde I_{0}(\delta_1,\delta_2)+\sum_{u=1}^\Delta (2^{u\delta_1}+2^{u\delta_2})2^{-u}\,\frac{2-2^{-\Delta+u}}{2-2^{-\Delta}}\,\tilde I_{u}(\delta_1,\delta_2)
 &\text{if } \Delta<\infty, \\
\mathcal I_\infty^G(\delta_1,\delta_2) &=\tilde I_{0}(\delta_1,\delta_2)+\sum_{u=1}^\infty (2^{u\delta_1}+2^{u\delta_2})2^{-u}\,\tilde I_{u}(\delta_1,\delta_2) &\text{if } \Delta=\infty. 
\end{align*}
Let us also define
\begin{multline}
\label{eqn:ItildeG_mat}
\GIGG(\Delta)=\diag\Bigl(\vect{\bG^0}\Bigr)
\bigl(\mathcal I_\Delta^G(d_a^0+d_b^0, d_{a'}^0+d_{b'}^0)_{(a,b),(a',b')\in\{1,\dots,p^2\}}\bigr)\diag\Bigl(\vect{\bG^0}\Bigr).
\end{multline}

We are now in a position to formulate the asymptotic distribution of $\hat G(\hat d)$.

\begin{thm}
\label{thm:G}
Suppose Assumptions {\ref{ass:LRD}--\ref{ass:Aast}} and {\ref{ass:Wcompact}--\ref{ass:Wparameters}} hold.
Let \[
j_1-j_0\to \Delta\in\{1,\dots,\infty\},\; \log(N_X)^2(N_X 2^{-j_0(1+2\beta)}+ N_X^{-1/2} 2^{j_0/2})\to 0.
\] 
Then $\vect{\sqrt{n}\left(\hat \bG(\hat\bd)-\bG^0\right)}$ converges in distribution to a centered Gaussian distribution with a variance equal to $\bV^{\bGexp(\Delta)}$, with
\begin{equation}
\label{eqn:WG}
V_{(a,b),(a',b')}^{(\bGexp)}(\Delta)=\GIGG_{(a,a'),(b,b')}(\Delta)+\GIGG_{(a,b'),(a',b)}(\Delta)
\end{equation}
where quantities $\GIGG(\Delta)$ are defined by \eqref{eqn:ItildeG_mat}.
\end{thm}
The proof is given in \ref{proof:thmG}.

We can deduce a convergence result for correlations.

\begin{cor}
\label{cor:corG}
Let $(a,b)\in\{1,\dots,p\}^2$, $a\neq b$. Define \[
\hat r_{a,b}=\frac{\hat G_{a,b}(\hat\bd)}{\sqrt{\hat G_{a,a}(\hat\bd)\hat G_{b,b}(\hat\bd)}}\text{  and  }r_{a,b}=\frac{G_{a,b}^0}{\sqrt{G_{a,a}^0 G_{b,b}^0}}.
\]
Then, under conditions of Theorem~\ref{thm:G},
\[
\sqrt{n}\,\left(\hat r_{a,b}-r_{a,b}\right)\tend^{\mathcal L}_{j\to\infty} \mathcal{N}\left(0,V^{(\mathbf{r})}_{a,b}(\Delta)\right)
\]
with
\begin{multline}
\label{eqn:Vr}
V^{(\mathbf{r})}_{a,b}(\Delta)=\mathcal I_\Delta^{G}(2d_a,2d_b)+\mathcal I_\Delta^{G}(d_a+d_b,d_a+d_b)(r_{a,b}^2+r_{a,b}^4)\\-( \mathcal I_\Delta^{G}(2d_a,2d_b) + \mathcal I_\Delta^{G}(2d_b,d_a+d_b))\,2\, r_{a,b}^2 - ( I_\Delta^{G}(2d_a,2d_a)+ \mathcal I_\Delta^{G}(2d_b,2d_b))\,r_{a,b}^2/2.
\end{multline}
When all off-diagonal entries of $\bG$ are equal to 0, \begin{equation}
\label{eqn:r0}
\sqrt{n}\,\vect{\hat r_{a,b},\; 1\leq a<b\leq p}\tend^{\mathcal L}_{j\to\infty} \mathcal{N}_{p(p-1)/2}\Bigl(0,\diag\bigl(\mbox{vec}\bigl({\tilde I_\Delta^G(2d_a,2d_b), \; 1\leq a<b\leq p}\bigr)\bigr)\Bigr).
\end{equation}
\end{cor}
The proof is based on delta method, and it is similar to the proof of Corollary~\ref{cor:cor}. It is thus omitted. The covariance structure of $\vect{\hat r_{a,b},\; a,b=1,\dots,p}$ can also be deduced from Theorem~\ref{thm:G}, but it is not displayed here.  

\begin{rem}
The result is very similar to the one presented in Corollary~\ref{cor:cor}. For all $(a,b)\in\{1,\dots,p\}^2$, the sequence $(\sqrt{n_j}(\rho_{a,b}(j)-r_{a,b}))_{j\geq 0}$ converges in distribution as $j$ goes to infinity. The strength of Corollary~\ref{cor:corG} is that all the scales are used to estimate $r_{a,b}$, which reduces the variance. Indeed, $(\sqrt{n}(\hat r_{a,b}-r_{a,b}))_{j\geq 0}$ converges in distribution as $j$ goes to infinity, with $n=\sum_{j=j_0}^{j_1} n_j$.
\end{rem}

\begin{rem}
When the LRD parameters are equal, \emph{i.e.} $d_a=d_b$, Corollary \ref{cor:corG} provides a more simple form, which is
\[
\sqrt{n}(\hat r_{a,b}-r_{a,b})\tend^{\mathcal L}_{j\to\infty} \mathcal{N}\left(0,\mathcal I_\Delta^{G}(2d_a,2d_a)(1-r_{a,b}^2)^2\right)\,.
\]
\end{rem}
% $I_\Delta(d)$, $\Delta=\infty$
%M= 1  &	 $\alpha$= 1 	     &  5.68 &  5.81 &  5.91 &  6.00 &  6.08\\ 
%M= 2  &	 $\alpha$= 1.34 	 &  2.65 &  2.66 &  2.67 &  2.68 &  2.69\\ 
%M= 3  &	 $\alpha$= 1.64 	 &  1.85 &  1.86 &  1.86 &  1.86 &  1.87\\ 
%M= 4  &	 $\alpha$= 1.91 	 &  1.62 &  1.62 &  1.61 &  1.61 &  1.61\\ 
%M= 5  &	 $\alpha$= 2.18 	 &  2.05 &  2.04 &  2.04 &  2.03 &  2.02\\ 
%M= 6  &	 $\alpha$= 2.43 	 &  1.90 &  1.91 &  1.92 &  1.93 &  1.94\\ 
%M= 7  &	 $\alpha$= 2.68 	 &  1.22 &  1.23 &  1.24 &  1.25 &  1.26\\ 
%M= 8  &	 $\alpha$= 2.93 	 &  1.01 &  1.01 &  1.01 &  1.01 &  1.01\\ 

\begin{rem}
The asymptotic variances $\bV^{(\mathbf{G})}{(\Delta)}$ and $\bV^{(\mathbf{r})}{(\Delta)}$, given respectively in Theorem~\ref{thm:G} and Corollary~\ref{cor:corG}, depend on parameters $\bd$. Test procedures can be built by plugging in $\bV^{(\bGexp)}(\Delta)$ and $\bV^{(\mathbf{r})}(\Delta)$ the estimator $\hat\bd$, which is consistent \citep[Theorem 6]{AchardGannaz}. 
\end{rem}

\begin{rem} In \cite{DukerPipiras}, D\"uker and Pipiras propose a global test for non-connectivity. That is, a test of $(H_0)\;\forall a\neq b,\; r_{a,b}=0 $ against  $(H_1)\; \exists a\neq b,\;r_{a,b}\neq0$.  A similar test can be developed in our setting, based on \eqref{eqn:r0}. Another possibility is to perform the $p(p-1)/2$ tests of $(H_{0\,a,b})\ r_{a,b}=0 $ against  $(H_{1\,a,b})\; \;r_{a,b}\neq0$, for $1\leq a<b\leq p$ and to apply a multiple testing correction on the p-values, for instance, Bonferroni's or Sidak's \citep{multiple}. This approach may be less powerful than the previous one if we are interested in the global test, but it provides information on which correlations are significant.
\end{rem}

We can go further than Theorem~\ref{thm:d} and Theorem~\ref{thm:G} by giving the joint distribution of estimators $\hat\bd$ and $\hat \bG(\hat\bd)$. 

\begin{prop}
\label{prop:dG}
Suppose assumptions of Theorem \ref{thm:G} hold.

Let $\mathbf T = \begin{pmatrix} \hat\bd-\bd^0, & \vect{\hat \bG(\hat\bd)-G^0}\end{pmatrix}$.

Then $\sqrt{n}\,\mathbf T$ converges in distribution to a centered Gaussian distribution.
\end{prop}
A proof is given in \ref{proof:thmdG}. An explicit form of the asymptotic covariance term is given in \eqref{eqn:covdG2}-\eqref{eqn:covdG3}. It is not displayed here to gain in clarity.

\begin{rem}
Baek et al. \citep{Baek2020} and D\"uker and Pipiras \citep{DukerPipiras} also find that the estimates of long-range dependence parameters and long-run covariance converge jointly to a Gaussian distribution in a Fourier-based Whittle estimation framework. As stated before, they consider a more general model, allowing for a complex-valued matrix $\bOmega$.
\end{rem}
  
\section{Illustration on real data}
\label{sec:real}

We illustrate here the asymptotically Gaussian behavior on real data rather than on simulations. We consider fMRI recordings on dead and live rats. The dataset is freely available at \url{https://zenodo.org/record/2452871} \citep{guillaume2020functional,becq_10.1088/1741-2552/ab9fec}. The duration of scanning is 30 minutes with a time repetition of 0.5 second so that $N_X=3,600$ time points are available at the end of experience. After preprocessing as described in~\cite{Pawela2008}, we extracted $p=51$ time series, each one being associated with a brain region of the rat. fMRI recordings of brain activity are based on the hemodynamic response to a magnetic field, which may create some temporal and spatial dependence. They suffer from different sources of noise, including system-related instabilities, subject motion, or physiological fluctuations \citep{fMRInoise}. Additionally, during the preprocessing step, we aggregate the time series of each voxel to obtain a unique time series for each brain region.  This aggregation step may create LRD properties \citep{aggregation}. Our claim is that long-range dependence and long-run covariance are closely related to brain activity and not to recording artifacts or preprocessing. We would like to check this assertion on the dataset. This means that we expect $\bd^0=0$ and a diagonal matrix $\bG^0$ for a dead rat but not for a live one.

%\subsection{Tests of significance}

We estimate $\bd$ and $\bG$ by wavelet-based Whittle estimation, using \emph{multiwave} package~\citep{AchardGannaz_code}. We follow the procedure described in \cite[Section 5.2]{AchardGannaz_code} to choose the scales. Estimation is performed taking $j_0=4$ and $j_1=9$, which is the maximal scale; that is, we remove the frequencies above 0.12 Hz.%j=4 : interval [0.002 ; 0.12] Hz.

Based on Theorem~\ref{thm:d}, for each rat, we can test if the LRD parameters are significant for each brain region. That is, for all $a=1,\dots,p$, we test \[
(H_{0\,a}^{(\bdexp)})~d_a=0 \text{~~against~~}(H_{1\,a}^{(\bdexp)})~d_a\neq 0\,,
\]
replacing $\bd$ and $\bG$ respectively by $\hat d$ and $\hat \bG(\hat\bd)$ in $\bV^{(\bdexp)}$. We consider a level $\alpha'=5\%$ and apply Bonferroni's multiple testing correction, \emph{i.e.} each test is applied with a level $\alpha'/p$ to ensure that the probability to have a false positive on the $p$ tests is equal to $\alpha'$. 

Next Corollary~\ref{cor:corG} allows to test the significance of the long-run correlation between each pair of brain regions. For all $1\leq a <b\leq p$, we test \[
(H_{0\,a,b}^{(\mathbf{r})})~r_{a,b}=0 \text{~~against~~}(H_{1\,a,b}^{(\mathbf{r})})~r_{a,b}\neq 0\,.
\]
Similarly, we apply Bonferroni's multiple testing correction and we consider a level $\alpha'/(p(p-1)/2)$ for each test.

The tests have been applied on one dead rat and one live rat. The results   are displayed in Figure~\ref{fig:graphs} as graphs. %Each vertex of the graph corresponds to a brain region. Colored vertices are regions where the LRD parameter $d_a$ is significant, \emph{i.e.} where the null hypothesis $(H_{0\,a}^{(\mathbf{r})})$ is rejected. Next, there is an edge between two vertices $a,b$ if the long-run correlation is significant, \emph{i.e.} if the null hypothesis $(H_{0\,a,b}^{(\mathbf{r})})$ is rejected. 
 Figure~\ref{fig:graphs} shows that, indeed, we can conclude that $\bd^0=0$ and that off-diagonal entries of $\bG^0$ are equal to zero for the dead rat. For the live rat, six brain regions (over 51) have a significant LRD parameter, and 483 correlations (over 1275 of $\{r_{a,b},~1\leq a <b\leq p\}$) are significant. These observations tend to confirm that long-range dependence and long-run covariance result from brain activity.

\begin{figure}[!ht]
\centering
\includegraphics[height=8cm]{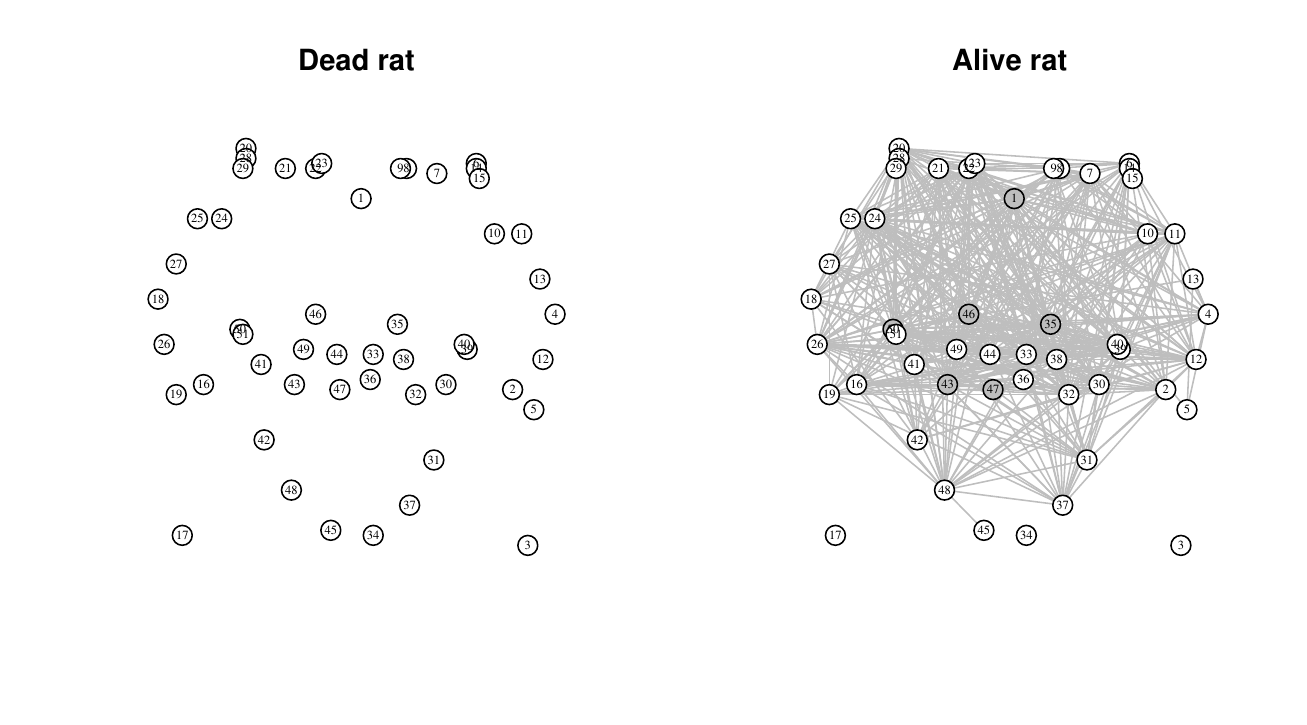}

\vspace{-\baselineskip}

\caption{Inferred graphs of cerebral activity for a dead rat (left) and a live rat (right). Each vertex of the graph corresponds to a brain region. Colored vertices are regions where the LRD parameter $d_a$ is significant, \emph{i.e.} where the null hypothesis $(H_{0\,a}^{(\mathbf{r})})$ is rejected. Two vertices $a, b$ are connected by an edge if the long-run correlation is significant, \emph{i.e.} if the null hypothesis $(H_{0\,a,b}^{(\mathbf{r})})$ is rejected. }
\label{fig:graphs}
\end{figure}

\section*{Conclusion}

In this paper, we consider a multivariate process with long-range dependence properties, with a linear representation. We first establish that the covariance between wavelet coefficients is asymptotically Gaussian. The variance is explicitely given, and the convergence is established under mild assumptions on the wavelet transform and on the process. The asymptotic normality for the wavelet-based Whittle estimators defined in~\cite{AchardGannaz} is also established. 

These results allow to perform statistical tests on the LRD parameters and on the long-run covariance. We propose an application on fMRI data, where we have recordings on a dead rat and alive one. The tests of significance on the LRD parameters and on the long-run correlations highlight that these characteristics are intrinsically linked to brain activity. %Moreover we obtain as expected that no long-run correlation is significant for the dead rat, on the contrary of what is observed for the alive rat. 
%An essential point for real data application is the development of statistical tests, on the long-range dependence parameters and on the long-run covariance. For example, as illustrated here on a real data example, these characteristics are intrinsically related to the brain activity recordings in neuroscience. Some works highlighted that their distributions can be modified by  pathologies (see {\it e.g.} \cite{maxim.2005.1} and \cite{Achard_coma}). Statistical test may be useful to assess rigorously such observations.

 \appendix
%\renewcommand{\thesection}{Appendix \thesection}

%\section*{Appendix}
%
%The construction of the proof of Theorem~\ref{thm:gauss} is adapted from that of \cite[Theorem 2]{Roueff09asymptotic}. The proof has been structured as follows. \ref{proof:coeff} proposes a writing of wavelet coefficients as decimated linear processes, and provides an approximation by a $m$-dependent decimated linear processes, $m\geq 0$. Notations and technical results on the decimated decompositions are stated in \ref{proof:v}. They are useful for applying the propositions of \cite{Roueff09central}, which lead to the asymptotic normality described in Theorem~\ref{thm:gauss}. This is explained in \ref{proof:thm:gauss}. \ref{proof:cor:gauss} deals with the proof of Corollary~\ref{cor:cor}.
%
%Next, \ref{sec:cvgceS} establishes the asymptotic normality of some linear combinations of sample wavelet covariances. This result is useful to prove the remaining theorems.
%
%\ref{proof:thmd}, \ref{proof:thmG} and \ref{proof:thmdG} are devoted to the proofs of Section \ref{sec:d}. They contain the proofs respectively of Theorem~\ref{thm:d}, Theorem~\ref{thm:G}, and Proposition~\ref{prop:dG}. Finally, \ref{proof:technical} presents three useful lemmas, which do not depend on the context of the manuscript.

\renewcommand{\appendixname}{Appendix }
\section{Expression of wavelet coefficients}

\label{proof:coeff}

Let $(a,b)\in\{1,\dots,p\}^2$. The objective is to study the asymptotic normality of the sample wavelet covariance $\{\hat \sigma_{a,b}(j_0+u),\,u=1,\dots,\Delta\}$ when $j_0$ goes to infinity. To this end, we introduce a new indexing of wavelet coefficients similar to that in \cite[pages 543 and 544]{Roueff09asymptotic}. This new indexing enables to approximate the sequence of wavelet coefficients with $m$-dependent variables and to use the results on linear decimated processes of \cite{Roueff09central}.
 The section is structured as follows. In \ref{sec:coeff}, we give a linear representation of wavelet coefficients, the new indexing of the coefficients is defined in \ref{sec:indexing}. \ref{sec:mdep} finally introduces the approximation by a $m$-dependent process.

\subsection{Linear representation of wavelet coefficients}

\label{sec:coeff}

Consider a scale $j\geq 0$ and $k\in\Z$.
Define, for all $l\in \Z$, $h_{j,l} = \int_\R \phi(t + l)2^{-j/2}\psi(2^{-j} t) \rmd t$, the discrete wavelet filter associated to $(\phi(\cdot),\psi(\cdot))$. Then under \ref{ass:Wcompact}, the vector of wavelet coefficients $\bW(j,k)$ defined in Section~\ref{sec:W} can be written as \[
\bW(j,k)=\sum_{l\in\Z} h_{j,2^jk-l} \bX(l),
\]
with $\bW(j,k)\in\R^p$.
For all $\lambda\in\R$ let us denote \[\mathbb{H}_j(\lambda) = \sum_{l\in\Z} h_{j,l} \rme^{-\rmi\lambda\,l}=\int_\R \sum_{l\in\Z}\phi(t + l)2^{-j/2}\psi(2^{-j} t) \rmd t,\] the discrete Fourier transform of $\{h_{j,l}, l\in\Z\}$.

Suppose that the multivariate process $\bX=\{X_{a}(k),k\in\Z, a=1,\dots,p\}$ satisfies Assumption \ref{ass:linear}. To express wavelet coefficients, we introduce, for all $\lambda\in(-\pi,\pi)$, 
\begin{equation}
\label{eqn:Aastdef}
\bA^\ast(\lambda) = \diag\bigl((1-\rme^{\rmi\lambda})^{-\bDexp}\bigr) \bA^{(\bDexp)\ast}(\lambda),
\end{equation}
 and $\{\bA(t), t\in\Z\}\in\ell^2(\Z)$ such that
 \begin{equation}
 \label{eqn:Adef}
\bA^\ast(\lambda) =  (2\pi)^{-1/2}\sum_{t\in\Z} \bA(t)\rme^{\rmi\lambda t}.
\end{equation} 
The function $\bA^\ast(\lambda)$ satisfies
\begin{equation}
\label{eqn:A=f}
\bA^\ast(\lambda)\overline{\bA^\ast(\lambda)}^T=\bff(\lambda),
\end{equation}
where $\bff(\cdot)$ is defined in \ref{ass:LRD}.

For all $\lambda\in(-\pi,\pi)$, let 
\begin{equation} 
\label{eqn:Aj}
\bA^\ast(j;\lambda) = \mathbb{H}_j(\lambda)\bA^{\ast}(\lambda),
\end{equation}
with $\bA^\ast(\lambda)$ defined in \eqref{eqn:Aastdef}. Let us also define $\{\bA(j;t), t\in\Z\}\in\ell^2(\Z)$ such that
 \begin{equation}
\bA^\ast(j;\lambda) = (2\pi)^{-1/2} \sum_{t\in\Z} \bA(j;t)\rme^{\rmi\lambda t}.
\end{equation}
Then, the wavelet coefficients can be written as 
\[
\bW({j,k})=\sum_{l\in\Z} \bA(j;2^jk-l){\bepsilon(l)}.
\]

\paragraph*{Implicit differentiation by wavelet representation}
As the wavelet $\psi$ admits $M$ vanishing moments under \ref{ass:Wmoments}, $\mathbb{H}_j$ can be factorized as $\mathbb{H}_j(\lambda)=(1-e^{i\lambda})^M \tilde{\mathbb{H}}_j(\lambda)$, with $\tilde{\mathbb{H}}_j$ trigonometric polynomial, $\tilde{\mathbb{H}}_j(\lambda)=\sum_{t\in\Z}\tilde h_{j,t}e^{it\lambda}$. %See \emph{e.g.} \cite[page 4]{Roueff09asymptotic}.
 It results that \[W_a({j,k})=\sum_{l\in\Z} \tilde h_{j,2^jk-l}(\mathbb{L}^{M} X_a)(l).\]
%where $\mathbb{L}^M$ is the $M$-iterated lag-operator.

\subsection{New indexing of wavelet coefficients}

\label{sec:indexing}

Let $j\geq 0$ and $k\in\{0,\dots, n_j-1\}$. We introduce the new indexing proposed by \cite{Roueff09asymptotic}. Let $u=j-j_0$, $u\in\{0,\dots,\Delta\}$, and define $(i, s)$ such that $k=2^{\Delta-u}(s-1)+i$, with $i\in\{2^{\Delta-u},\dots, 2^{\Delta-u+1}-1\}$ and $s\in\Z$. We have $2^{j}k=2^{j_1}(s-1+2^{u-\Delta}i)$. Index $i$ varies from $1$ to $N=2^{\Delta+1}-1$ and each couple $(j,k)$ corresponds to a unique couple $(i,s)$. We can rewrite wavelet coefficients as
\[
\bLambda_{j_0}(\bd)^{-1}\bW({j,k})=\sum_{t\in\Z} \bmV^{(i,j_0)}(2^{j_1}s-t)\bepsilon(t)
\]
{ with } $\bLambda_{j_0}(\bd)$ is defined in \eqref{eqn:density} and
\[
\bmV^{(i,j_0)}(t)=\bLambda_{j_0}(\bd)^{-1}\bA(j;2^{j}(i-2^{j_1-j})+t),~ j=j_1-\lfloor\log_2(i)\rfloor,
\]
{ where } $\lfloor\log_2(i)\rfloor=\Delta-u$ is the integer part of $\log_2(i)$. 
Write \[ \bmZ^{(i,s,j_0)}= \sum_{t\in\Z} \bmV^{(i,j_0)}(2^{j_1} s-t)\bepsilon(t)=\bLambda_{j_0}(\bd)^{-1}\bW(j,2^{-[\log_2(i)]}(s-1)+i)\,.\]
$\bW(j,k)$, $\bmV^{(i,j_0)}$, $\bmZ^{(i,s,j_0)}$ belong respectively to $\R^p$, $\R^{p\times p}$ and $\R^p$.

For all $u=0,\dots, \Delta$, denoting $j=j_0+u$, the empirical variance satisfies
\begin{align}
\nonumber \bLambda_{j_0}(\bd)^{-1}\hat \bSigma(j)\bLambda_{j_0}(\bd)^{-1}&=\frac{1}{n_{j}} \sum_{k=0}^{n_{j}-1}\bLambda_{j_0}(\bd)^{-1}\bW({j,k})\,\bW({j,k})^T\bLambda_{j_0}(\bd)^{-1}\\
\label{eqn:defSR}& = \frac{\sqrt{n_{j_1}}}{n_{j}}\sum_{i=2^{\Delta-u}}^{2^{\Delta-u+1}-1}  \underbrace{n_{j_1}^{-1/2}\sum_{s=0}^{n_{j_1}-1} \bmY^{({i,s,j_0})}}_{\bmS^{({i,j_0})}}{+\underbrace{n_{j}^{-1} \sum_{i=2^{\Delta-u}}^{T_\psi\,(2^{\Delta-u}-1)} \bmY^{({i,n_{j_1},j_0})}}_{\bmR(j)}\,,}
\end{align}
with
\[
\bmY^{({i,s,j_0})} = \bmZ^{({i,s,j_0})} \bmZ^{({i,s,j_0})T} =(\mZ_{a}^{(i,s,j_0)}\mZ_{b}^{(i,s,j_0)})_{a,b=1,\dots,p}\,.
\]
Indeed, when $s\in\{0,\dots, n_{j_1}-1\}$ and $i\in\{2^{\Delta-u},\dots, 2^{\Delta-u+1}-1\}$, index $k=2^{\Delta-u}(s-1)+i$ varies in $\{0,\dots, 2^{\Delta-u}n_{j_1}-1\}$, and when 
$s=n_{j_1}$ and $i\in\{2^{\Delta-u},\dots, T_\psi\,(2^{\Delta-u}-1)\}$, index $k$ varies from $2^{\Delta-u}n_{j_1}$ to $2^{\Delta-u}(n_{j_1}-1)+T_\psi\,(2^{\Delta-u}-1)=2^{-j}(N_X-T_\psi+1)-T_\psi=n_{j}-1$. That is,
\begin{multline*}
\{k=0,\dots, n_{j}-1\}=\{k=2^{\Delta-u}(s-1)+i,\;s=0,\dots, n_{j_1}-1, \;i=2^{\Delta-u},\dots, 2^{\Delta-u+1}-1\}\\\cup\{k=2^{\Delta-u}(s-1)+i,\;s=n_{j_1},\; i=2^{\Delta-u}, \dots, T_\psi\,(2^{\Delta-u}-1)\}.
\end{multline*}

The proof of Theorem~\ref{thm:gauss} consists in establishing first the asymptotic normality of $\{\bmS^{(i,j_0)}, i=1,\dots,N\}_{j_0\geq 0}$ when $j_0$ goes to infinity, and second that $(\bmR(j))_{j\geq 0}$ is negligible. To prove the asymptotic normality of $\{\bmS^{(i,j_0)}, i=1,\dots,N\}_{j_0\geq 0}$, we will need to approximate the variables $\{\bmY^{(i,j_0)}, i=1,\dots,N\}_{j_0\geq 0}$ by $m$-dependent variables.

\subsection{Approximation by a $m$-dependent process}
\label{sec:mdep}

Following \cite{Roueff09central}, we introduce a non-negative infinitely differentiable function $H(\cdot)$ defined on $\R$ such that $H(0)=1$ and $H(t)=0$ if $\abs{t}>1/2$. Write $\hat H(\cdot)$ its Fourier transform, $\hat H(\lambda)=\int_{-\infty}^\infty H(t)\rme^{-\rmi\lambda\,t}\rmd t$. Since $H$ is supposed infinitely derivable, when $\abs{\lambda}$ tends to infinity, $\hat H(\lambda)$ decreases to 0 faster than any polynomial. Hence, there exists $c_H>0$ such that $\lvert \hat H(\lambda) \rvert \leq c_H \lvert \lambda \rvert^{-\delta_v-1}$ for all $\abs{\lambda}\geq 1$, with $\delta_v$ defined in Lemma~\ref{lem:v1}. Additionally, $\frac{1}{2\pi}\int_{-\infty}^{\infty}\hat H(\lambda)\,\rmd\lambda=H(0)=1$.

Let us define, for all $t\in\R$, for all $\lambda\in\R$,
\begin{align}
\bmV^{(i,j_0)(m)}(t)&= H(2^{-j_1} t/m)\,\bmV^{(i,j_0)}(t)\,,\\
\bmV^{(i,j_0)(m)\ast}(\lambda)& = (2\pi)^{-1/2}\sum_{t\in\Z}\bmV^{(i,j_0)(m)}(t)e^{-i\lambda t}=\frac{m}{2\pi}\int_{-\infty}^\infty \hat H (m\xi)\bmV^{(i,j)\ast}(\lambda-2^{-j_1}\xi)\,\rmd\xi \,,\\
\bmZ^{(i,s,j_0)(m)}&= \sum_{t\in\Z} \bmV^{(i,j)(m)}(2^{j_1} s -t)\bepsilon(t)\,.%\text{for }a=1,\dots,p.
 \end{align}
Then for all $a,b=1,\dots,p$, vectors $\Z_{a,b}^{(s,j_0)(m)}=\begin{pmatrix}
\mZ_{a}^{(1,s,j_0)(m)}, \dots, \mZ_{a}^{(N,s,j_0)(m)}, \mZ_{b}^{(1,s,j_0)(m)}, \dots, \mZ_{b}^{(N,s,j_0)(m)}
\end{pmatrix}^T\in\R^N$ are $m$-dependent relatively to index $s$. That is, for all $q\geq 1$, for all $(s_1, \dots, s_q)$ such that $s_{r+1}\geq s_{r}+m$ for $r=1,\dots,q$,
vectors $\Z_{a,b}^{(s_1,j_0)(m)}$, \ldots, $\Z_{a,b}^{(s_q,j_0)(m)}$ are independent. 

We will next study sequences $\{\bmS^{(i,j_0)(m)}, i=1,\dots,N\}_{j_0\geq 0}$ which are defined as follows:
\begin{align}
\label{eqn:Ym}\bmY^{(i,s,j_0)(m)}&=\bmZ^{(i,s,j_0)(m)}\bmZ^{(i,s,j_0)(m)T},\\
\label{eqn:Sm}\bmS^{(i,j_0)(m)}&=n_{j_1}^{-1/2}\sum_{s=0}^{n_{j_1}-1}\bmY^{(i,s,j_0)(m)}.
\end{align}
The outline of the proof of Theorem~\ref{thm:gauss} is first to prove the asymptotic normality of $\{\bmS^{(i,j_0)(m)}, i=1,\dots,N\}_{j_0\geq 0}$ with the use of the results on decimated $m$-dependent processes of \cite{Roueff09asymptotic}. Next a similar result for $\{\bmS^{(i,j_0)}, i=1,\dots,N\}_{j_0\geq 0}$ is deduced by letting $m$ go to infinity.

\section{Notations and technical lemmas}

\label{proof:v}

This section provides some technical results on the quantities introduced in the wavelet representation and in the approximation by a $m$-dependent process, respectively in \ref{sec:indexing} and \ref{sec:mdep}. These results will be used for the proof of Theorem~\ref{thm:gauss}.

\subsection{Useful inequalities concerning the linear wavelet representation}

We first give two lemmas, respectively on the behavior of the spectral density $f(\cdot)$ and of the function $A^\ast(\cdot)$.

\begin{lem} \label{lem:f}
Suppose \ref{ass:LRD}--\ref{ass:beta} hold. Then there exists $C_f>0$ depending on $L$, $\beta$ and $\bOmega$ such that for all $a,b=1,\dots, p$, for all $\lambda\in(-\pi,\pi)$,
\begin{align*}
\abs{f_{a,b}(\lambda)}\leq C_f\abs{\lambda}^{-d_a-d_b}\,.
\end{align*}
\end{lem}
\begin{proof} Let $(a,b)\in\{1,\dots,p\}^2$ and $\lambda\in(-\pi,\pi)$. By \ref{ass:LRD}, \[ \abs{f_{a,b}(\lambda)}\leq \max_{\ell,m}\abs{\Omega_{\ell,m}}\abs{1-\rme^{\rmi\lambda}}^{-d_a-d_b}\,\abs{f_{a,b}^S(\lambda)}.\]
From Assumption \ref{ass:beta}, $\abs{f^S_{a,b}(\lambda)}\leq L(1+\pi^\beta)$. Additionally, $\abs{1-\rme^{\rmi\lambda}}=\abs{2\sin\bigl(\lambda/2\bigr)}\leq \abs{\lambda}$.
Lemma \ref{lem:f} follows with $C_f=L(1+\pi^\beta)\max_{a,b=1,\dots,p}\abs{\Omega_{a,b}}$.
\end{proof}

\begin{lem}\label{lem:sum_A}
Suppose \ref{ass:LRD}--\ref{ass:linear} hold. Then there exists $C_A>0$ depending on $L$, $\beta$ and $\bOmega$ such that for all $(a,b)\in\{1,\dots,p\}^2$,
\begin{equation}
\label{eqn:sum_A}
\abs{\Bigl({\bA^\ast(\lambda)}\overline{\bA^\ast(\lambda)}^T\Bigr)_{a,b}}\leq C_A\,\abs{\lambda}^{-d_a-d_b}\,.
\end{equation}
\end{lem}
\begin{proof}
The lemma is straightforward combining \eqref{eqn:A=f} and Lemma \ref{lem:f}.
\end{proof}

\subsection{Preliminary results on $(\bmV^{(i,j_0)})$} 
\label{sec:v}

Define
\[
\bmV^{(i,j_0)\ast}(\lambda)=(2\pi)^{-1/2} \sum_{t\in\Z}\bmV^{(i,j_0)}(t)\rme^{-\rmi\lambda t}\,,\quad \lambda\in\R\,.
\]
Observe that $[\bmV^{(i,j_0)}(2^{j_1}s-t)]^\ast(\lambda)=\overline{\bmV^{(i,j_0)\ast}(\lambda)}\rme^{-\rmi\,2^{j_1}s\lambda}$.

For all $i, i'= 1,\dots, N$, for all $\lambda\in\R$, let us define also $\bmW^{(i,i')\ast}(\lambda)=(\mW_{a,b}^{(i,i')\ast}(\lambda))_{a,b=1,\dots,p}$ with
\begin{multline}
\label{eqn:wdef}
 \bmW^{(i,i')\ast}(\lambda)= 2^{(u-\Delta)/2+(u'-\Delta)/2}\,\overline{\hat\psi(2^{u-\Delta}\lambda)}\,{\hat\psi(2^{u'-\Delta}\lambda)}\,\rme^{\rmi(\,2^{u-\Delta}(i-1)-\,2^{u'-\Delta}(i'-1))\lambda}\\
 \diag\bigl(\lvert\lambda\rvert^{-\bdexp}\rme^{-\rmi\,\sign(\lambda)\,\pi\bdexp/2}\bigr)\bLambda_\Delta(\bd)\,\bOmega\,\bLambda_\Delta(\bd) \diag\bigl(\lvert\lambda\rvert^{-\bdexp}\rme^{\rmi\,\sign(\lambda)\,\pi\bdexp/2}\bigr),
\end{multline}
where $u=\Delta-\lfloor \log_2(i)\rfloor$, $u'=\Delta-\lfloor \log_2(i')\rfloor$.

We begin by providing some results on the behavior of $(\bmV^{(i,j_0)\ast}(\cdot))$ in a sequence of lemmas.
\begin{lem} \label{lem:v1}
Suppose assumptions of Theorem~\ref{thm:gauss1} hold. Suppose $\Delta <\infty$. Then there exists $\delta_v>1/2$ such that for all $j\geq j_0$, $j-j_0\leq \Delta$,  we have
\begin{equation}
\label{eqn:sup_v_ast}
\sup_{|\lambda|<\pi}\norm{\bmV^{(i,j_0)\ast}(\lambda)\overline{\bmV^{(i,j_0)\ast}(\lambda)}^T}_\infty\leq C_v\,  2^{j}(1+2^j\lvert\lambda\rvert)^{-2\delta_v},
\end{equation}
with $j=j_0+\Delta-\lfloor\log_2(i)\rfloor$ and $C_v=C_A\,C_{H1}^2\,2^{\Delta(d_a+d_b)} < \infty$, depending on $L$, $\beta$, $\bOmega$, $\bd$, $\Delta$, $\phi(\cdot)$ and $\psi(\cdot)$.
\end{lem}
\begin{proof} 
Recall that the Fourier transform $\bA^\ast(\cdot)$ was defined in \eqref{eqn:Aastdef}.
Observe that
\[
\bmV^{(i,j_0)\ast}(\lambda)\overline{\bmV^{(i,j_0)\ast}(\lambda)}^T=  (2\pi)^{-1}\abs{\H_{j}(\lambda)}^2\,\bLambda_{j_0}(\bd)^{-1}{\bA^\ast(\lambda)}\overline{\bA^\ast(\lambda)}^T\bLambda_{j_0}(\bd)^{-1}\,.
\]
Lemma~\ref{lem:sum_A} yields
\[
\abs{\bigl(\bmV^{(i,j_0)\ast}(\lambda)\overline{\bmV^{(i,j_0)\ast}(\lambda)}^T\bigr)_{a,b}}\leq C_A\, (2\pi)^{-1}\abs{\H_{j}(\lambda)}^2\, 2^{({j-j_0})(d_a+d_b)} \abs{2^j\lambda}^{-d_a-d_b}\,.
\]
From \eqref{eqn:Hj} we get, for all $a,b=1,\dots,p$,
\begin{multline*}
\abs{\bigl(\bmV^{(i,j_0)\ast}(\lambda)\overline{\bmV^{(i,j_0)\ast}(\lambda)}^T\bigr)_{a,b}}\\\leq C_AC_{H1}^2\,  2^{(j-j_0)(d_a+d_b)}\,2^{j}\left(\frac{2^j\abs{\lambda}}{1+2^j\abs{\lambda}}\right)^{2M-d_a-d_b}(1+2^j\abs{\lambda})^{-2\alpha-d_a-d_b}.
\end{multline*}
Therefore, since $d_a+d_b<2\,M$ and $0\leq\frac{2^j\abs{\lambda}}{1+2^j\abs{\lambda}}\leq 1$,
\begin{equation}
\abs{\bigl(\bmV^{(i,j_0)\ast}(\lambda)\overline{\bmV^{(i,j_0)\ast}(\lambda)}^T\bigr)_{a,b}}\leq C_AC_{H1}^2\,  2^{(j-j_0)(d_a+d_b)}\,2^{j}(1+2^j\lvert\lambda\rvert)^{-2\alpha-d_a-d_b}.
\end{equation}
Lemma~\ref{lem:v1}, hence, holds with $C_v=C_AC_{H1}^2\,  2^{(j-j_0)M}$ and $\delta_v=\alpha+\max_{a=1,\dots,p} d_a$. Assumption \ref{ass:Wparameters} ensures that $\delta_v>1/2$.
\end{proof}

The following lemma provides some convergence results on $(\bmV^{(i,j_0)\ast}(\cdot) )$.
\begin{lem} 
\label{lem:v}
For all $i=1,\dots, N$, for all $\lambda\in\R$, there exist
$\bPhi^{(i,j_0)}(\lambda)\in(-\pi,\pi)^{p\times p}$ and $\bmV^{(i,\infty)\ast}(\lambda)\in\C^{p\times p},$ such that 
\begin{gather}
\label{eqn1ii} 
2^{-j_1/2}{\bmV^{(i,j_0)\ast}}(2^{-j_1}\lambda)\rme^{-\rmi\bPhi^{(i,j_0)}(2^{-j_1}\lambda)} \tend_{j_0\to \infty} \bmV^{(i,\infty)\ast}(\lambda)\,,\\
\label{eqn1iii} 
2^{-j_1}\bmV^{(i,j_0)\ast}(2^{-j_1}\lambda)\overline{\bmV^{(i',j_0)\ast}(2^{-j_1}\lambda)}^T \tend_{j_0\to\infty} \bmW^{(i,i')\ast}(\lambda)\,,
\end{gather}
where $\bmW^{(i,i')\ast}(\lambda)$ is defined in \eqref{eqn:wdef}.
\end{lem}
\begin{proof}
\item 
\paragraph[Proof of convergence]{Proof of \eqref{eqn1ii}}
Let $\bPhi^{(i,j_0)}(\lambda)$ be the arguments of $\bmV^{(i,j_0)\ast}(\lambda)$. Let $(a,b)\in\{1,\dots,p\}^2$. From \eqref{eqn:Hjapprox}, we have
\begin{multline}
\label{eqn:vj}
\left\lvert2^{-j_1/2}\mV_{a,b}^{(i,j_0)\ast}(2^{-j_1}\lambda)\rme^{-\rmi\Phi_{a,b}^{(i,j_0)}(2^{-j_1}\lambda)}-2^{-\Delta/2+\Delta d_a}\abs{\hat\phi(2^{-j_1}\lambda)\hat\psi(2^{u-\Delta}\lambda)2^{-j_1d_a}A^\ast_{a,b}(2^{-j_1}\lambda)}\right\rvert\\
\leq C_{H2} 2^{j(1/2-\alpha-M)+(u-\Delta)M}\lvert\lambda\rvert^M 2^{-\Delta/2+\Delta d_a}\abs{2^{-j_1d_a}A^\ast_{a,b}(2^{-j_1}\lambda)}.
\end{multline}
Lemma \ref{lem:sum_A} gives the inequality \[
\abs{2^{-j_1d_a}A^\ast_{a,b}(2^{-j_1}\lambda)}^2\leq 2^{-2\,j_1d_a}\sum_{a'=1}^p\abs{A^\ast_{a,a'}(2^{-j_1}\lambda)}^2 \leq C_A\lvert\lambda\rvert^{-2d_a} \,,
\]
for all $a,b=1,\dots,p$.
Hence, \begin{multline*}
%\norm{2^{-j_1/2}\bmV({i,j_0})^{\ast}(2^{-j_1}\lambda)\rme^{-\rmi\bPhi^{(i,j_0)}(2^{-j_1}\lambda)}-2^{-\Delta/2+\Delta \bdexp}\abs{\hat\phi(2^{-j_1}\lambda)\hat\psi(2^{u-\Delta}\lambda)}\abs{\bLambda_{j_1}(\bd)^{-1}\bA^\ast(2^{-j_1}\lambda)}}_\infty\\
%\leq C_{H2}C_A \lvert\lambda\rvert^{M-2d_a} 2^{-\Delta/2+\Delta d_a+(u-\Delta)M}\, 2^{j(1/2-\alpha-M)}.
\left\lvert2^{-j_1/2}\mV_{a,b}^{(i,j_0)\ast}(2^{-j_1}\lambda)\rme^{-\rmi\Phi_{a,b}^{(i,j_0)}(2^{-j_1}\lambda)}-2^{-\Delta/2+\Delta d_a}\abs{\hat\phi(2^{-j_1}\lambda)\hat\psi(2^{u-\Delta}\lambda)2^{-j_1d_a}A^\ast_{a,b}(2^{-j_1}\lambda)}\right\rvert\\
\leq C_{H2}C_A \lvert\lambda\rvert^{M-2d_a} 2^{-\Delta/2+\Delta d_a+(u-\Delta)M}\, 2^{j(1/2-\alpha-M)}.
\end{multline*}
Since $1/2-\alpha-M<0$, we obtain that the right-hand side goes to 0 when $j_0$ goes to infinity. By continuity, $\abs{\hat \phi(2^{-j_1}\lambda)}$ tends to $\abs{\hat{\phi}(0)}=1$ when $j_0$ goes to infinity. We conclude~\eqref{eqn1ii} by Assumption \ref{ass:Aast}, which supposes that $2^{-j_1 \,d_a}A_{a,b}^\ast(2^{-j_1}\lambda)$ converges when $j_1$ goes to infinity.
\item 
\paragraph[Proof of cross convergence]{Proof of \eqref{eqn1iii}}%~\\
 By equality \eqref{eqn:A=f}, we get
\begin{align*}
\begin{split}
\Bigl\lVert{2^{-j_1}\bmV^{(i,j_0) \ast}(2^{-j_1}\lambda)\overline{\bmV^{(i',j_0)\ast}(2^{-j_1}\lambda)}^T}- 2^{-j_1+j/2+j'/2}\lvert \hat\phi(2^{-j_1}\lambda)\rvert^2\,\overline{\hat\psi(2^{j-j_1}\lambda)}{\hat\psi(2^{j'-j_1}\lambda) }\\
{\bLambda_{j_0}(\bd)^{-1}{\bff(2^{-j_1}\lambda)}\bLambda_{j_0}(\bd)^{-1}\rme^{\rmi(\,2^{u-\Delta}i-\,2^{u'-\Delta}i')\lambda}\Bigr\rVert_\infty}
\end{split}\\
\lefteqn{\leq 2^{-j_1}\Bigl\lVert{\H_{j}(2^{-j_1}\lambda)}\overline{\H_{j'}(2^{-j_1}\lambda)}}\\ 
&\pushright{ -2^{j/2+j'/2}\lvert \hat\phi(2^{-j_1}\lambda)\rvert^2\, \overline{\hat\psi(2^{j-j_1}\lambda)}{\hat\psi(2^{j'-j_1}\lambda) }\,{\bLambda_{j_0}(\bd)^{-1}{\bff(2^{-j_1}\lambda)}\bLambda_{j_0}(\bd)^{-1}}\Bigr\rVert_\infty}
\end{align*}
Inequality \eqref{eqn:Hj2approx} gives that the right-hand side can be bounded by
\[
C_{H3}\, 2^{-j_1}2^{(j+j')(1/2-\alpha)}2^{(u+u'-2\Delta)M}\abs{\lambda}^{2M}\norm{\bLambda_{j_0}(\bd)^{-1}{\bff(2^{-j_1}\lambda)}\bLambda_{j_0}(\bd)^{-1}}_\infty.
\]
With Lemma \ref{lem:f}, the bound becomes
\[
\max_{a,b=1,\dots,p}\,C_{H3}C_f\, 2^{-(2j_0+u+u')(\alpha)}\,2^{(u+u'-2\Delta)(M+1/2)}\,2^{\Delta(d_a+d_b)}\abs{\lambda}^{2M-d_a-d_b}\,.
\]
This term goes to $0$ when $j_0$ goes to infinity uniformly for $\lambda\in(-\pi,\pi)$.

As $\abs{\hat \phi(2^{-j_1}\lambda)}\tend_{j_0\to\infty}\abs{\hat{\phi}(0)}=1$ and as $\bff(2^{-j_1}\lambda)$ satisfies approximation \eqref{eqn:approx}, we obtain convergence~\eqref{eqn1iii}.
\end{proof}

We introduce some useful notations. For $(i,i')\in\{1,\dots,N\}^2$, $t\in\Z$, $\lambda\in\R$, let
\begin{align}
\tilde{\bmV}^{(i,\infty)}(t)&=\frac{1}{\sqrt{2\pi}}\int \bmV^{(i,\infty)\ast}(\lambda)\rme^{\rmi\, \lambda\,t}\rmd\,\lambda, \\
\label{eqn:wtilde}\tilde{\bmW}^{(i,i')\ast}(\lambda)&=\sum_{t\in\Z}\bmW^{(i,i')\ast}(\lambda+2t\pi), 
\end{align}
where $\bmV^{(i,\infty)\ast}(\cdot)$ and $\bmW^{(i,i')\ast}(\cdot)$ have been defined respectively in \eqref{eqn1ii} and in \eqref{eqn:wdef}.
We also define, for $(s,s')\in\{1,\dots,N\}^2$, %$(s,s')\in\N^2$,
 $(a,b,a',b')\in\{1,\dots,p\}^4$,
\begin{align}
\label{eqn:Sigma} \bTheta^{((i,s),(i',s'))}&= \int_{-\infty}^\infty \bmW^{(i,i')\ast}(\lambda)\rme^{-\rmi(s-s')\lambda}\rmd\lambda,\\
\label{eqn:Gamma}\Gamma_{(a,b),(a',b')}^{(i,i')}&=  2\pi\int_{-\pi}^\pi \overline{\tilde{\mW}_{a,a'}^{(i,i')\ast}(\lambda)}\tilde{\mW}_{b,b'}^{(i,i')\ast}(\lambda)\rmd\lambda {+ 2\pi\int_{-\pi}^\pi \overline{\tilde{\mW}_{a,b'}^{(i,i')\ast}(\lambda)}{\tilde{\mW}_{b,a'}^{(i,i')\ast}(\lambda)}\rmd\lambda}.
\end{align} 

We can first state a result on $\bTheta$. This result is useful to show that $\bGamma^{(i,i')}$ is the asymptotic covariance matrix of $(\bmS^{(i,j_0)}, \bmS^{(i',j_0)})$ as $j_0$ goes to infinity. This will be proved in Lemma~\ref{lem:cov}.

\begin{lem}\label{lem:theta}
For all $(i,i')\in\{1,\dots,N\}^2$, 
\begin{gather}
\label{eqnthetai}\forall j_0\geq 0, \forall a=1,\dots, p, ~ \sup_{s\in\Z}\sum_{t\in\Z} \sum_{a'=1}^p  \mV_{a,a'}^{(i,j_0)}(2^{j_1}s-t)^2  < \infty\,, \\
\label{eqn1iv}\forall (s,s')\in\Z^2, ~ \sum_{t\in\Z} \bmV^{(i,j_0)}(2^{j_1}s-t)\bmV^{(i',j_0)}(2^{j_1}s'-t)^T  \tend_{j_0\to\infty}  \bTheta^{((i,s),(i',s'))}\,,
\end{gather} 
with $\bTheta^{((i,s),(i',s'))}$ defined in \eqref{eqn:Sigma}.

Moreover, for all $(a,b)\in\{1,\dots,p\}^2$, for all $i\in\{1,\dots,N\}$,
\begin{equation}
\label{eqnthetaii}
\sup_{t\in\Z} \sum_{s\in\Z} {\mV_{a,b}^{(i,j_0)}(2^{j_1}s-t)}^2 \tend_{j_0\to\infty} 0\,.
\end{equation}
\end{lem}
\begin{proof}
\item 
\paragraph[Proof of dominance]{Proof of \eqref{eqnthetai}}%~\\
By Parseval's identity and a change of variable, for all $(a,a')\in\{1,\dots,p\}^2$,
\begin{equation*}
\sum_{t\in\Z}  \mV_{a,a'}^{(i,j_0)}(2^{j_1}s-t)^2   =  \int_{-\pi}^\pi \lvert \mV_{a,a'}^{(i,j_0)\ast}(\lambda)\rvert^2\,d\lambda \leq {C_{v}}\, \int_{-\infty}^\infty (1+\lvert\lambda\rvert)^{-2\delta_v}\,\rmd\lambda\,,
\end{equation*}
which implies \eqref{eqnthetai}. 
\item 
\paragraph[Proof of cross-convergence]{Proof of \eqref{eqn1iv}}%~\\
Applying Parseval's theorem and the change of variable $\lambda\to 2^{j_1}\lambda$, we get the equality
\begin{multline}
{\sum_{t\in\Z} \bmV^{(i,j_0)}(2^{j_1}s-t){\bmV^{(i',j_0)}(2^{j_1}s'-t)}^T}\\
= \int_{-2^{j_1}\pi}^{2^{j_1}\pi} 2^{-j_1}{\bmV^{(i,j_0)\ast}(2^{-j_1}\lambda)}\overline{\bmV^{(i',j_0)\ast}(2^{-j_1}\lambda)}^T \rme^{\rmi(s-s')\lambda}\,\rmd\lambda.
\end{multline} 
The function under the integral converges to ${\bmW^{(i,i')\ast}(\lambda)}\rme^{\rmi(s-s')\lambda}$ by \eqref{eqn1iii}. Convergence under the integral can be applied thanks to dominated convergence, by Lemma~\ref{lem:v1}.  This gives \eqref{eqn1iv}.
\item 
\paragraph[Proof of convergence]{Proof of \eqref{eqnthetaii}}%~\\
Observe that
\begin{equation*}
\mV^{(i,j_0)}_{a,b}(2^{j_1}s-t) = \frac{1}{\sqrt{2\pi}} \int_{-\pi}^\pi  \overline{\mV^{(i,j_0)\ast}_{a,b}(\lambda)}\rme^{-\rmi (2^{j_1}s+t)\lambda} \rmd\lambda .
\end{equation*}
Since the function $\lambda\to \mV^{(i,j_0)\ast}_{a,b}(\lambda)$ is $2\pi$-periodic,\cite[Lemma 4]{Roueff09central}, leads to
\begin{align*}
\mV^{(i,j_0)}_{a,b}(2^{j_1}s-t) & = \frac{1}{\sqrt{2\pi}}\int_{-\pi}^\pi  2^{-j_1}\sum_{l=0}^{2^j-1}  \mV^{(i,j_0)\ast}_{a,b}2^{-j_1}(\lambda+2\pi l))\rme^{-\rmi (2^{j_1}s-t)(2^{-j_1}(\lambda+2\pi l)} \rmd\lambda\\
&= \frac{1}{\sqrt{2\pi}}\int_{-\pi}^\pi  \bigl[ 2^{-j_1}\sum_{l=0}^{2^{j_1}-1}  \mV^{(i,j_0)\ast}_{a,b}2^{-j_1}(\lambda+2\pi l))\rme^{-\rmi\, t\, 2^{-j_1}(\lambda+2\pi l)}\bigr] \rme^{\rmi s\lambda} \rmd\lambda.
\end{align*}
Parseval's identity entails that
\begin{equation*}
\sum_{s\in\Z} \abs{\mV^{(i,j_0)}_{a,b}(2^{j_1}s-t)}^2 = \int_{-\pi}^\pi  \Bigl\lvert2^{-j_1}\sum_{l=0}^{2^{j_1}-1}  \mV^{(i,j_0)\ast}_{a,b}(2^{-j_1}(\lambda+2\pi l))\rme^{-\rmi\, t \,2^{-j_1}(\lambda+2\pi l)}\Bigr\rvert^2 \rmd\lambda.
\end{equation*}
Hence,
\begin{equation*}
\sum_{s\in\Z} \abs{\mV^{(i,j_0)}_{a,b}(2^{j_1}s-t)}^2 \leq 2^{-j_1} \int_{-\pi}^\pi  \Bigl(\sum_{l=0}^{2^{j_1}-1}  2^{-j_1/2}\abs{\mV^{(i,j_0)\ast}_{a,b}(2^{-j_1}(\lambda+2\pi l))}\Bigr)^2 \rmd\lambda.
\end{equation*}
Lemma \ref{lem:v1} implies that
\begin{equation*}
\sum_{s\in\Z} \abs{\mV^{(i,j_0)}_{a,b}(2^{j_1}s-t)}^2 \leq C_v\,2^{-j_1} \int_{-\pi}^\pi  \Bigl(\sum_{l=0}^{2^{j_1}-1}  (1 + \abs{\lambda+2\pi l})^{-\delta_v}\Bigr)^2 \rmd\lambda.
\end{equation*}
We can deduce the following inequalities
\begin{align*}
\sum_{s\in\Z} \abs{\mV^{(i,j_0)}_{a,b}(2^{j_1}s-t)}^2 &\leq C_v\,2^{-j_1} \int_{-\pi}^\pi  \Bigl(1+\sum_{l=0}^{2^{j_1}-2}  (1 + \abs{2\pi l})^{-\delta_v}\Bigr)^2 \rmd\lambda\\
&\leq C_v\,2^{-j_1}\, 2\pi  \Bigl(1+\int_{0}^{2^{j_1}-2}  (1 + \abs{2\pi \xi})^{-\delta_v}\rmd\xi\Bigr)^2 \\
&\leq C_v\,2^{-j_1}\, 2\pi\Bigl(1+(1+\abs{2\pi\,2^{j_1}})^{1-\delta_v}\Bigr)^2\\
&\leq C_v 4\pi 2^{-j_1}\,\Bigl(1+ (2\pi)^2\,2^{j_1(2-2\delta_v)}\Bigr).
\end{align*}
The right-hand side goes to 0 when $j_1$ goes to infinity since $\delta_v>1/2$.
\end{proof}

\subsection{Preliminary results on the $m$-dependent processes}

We define similar quantities than in \ref{sec:v} in the $m$-dependent setting. That is,
\begin{align}
\bmV^{(i,\infty)(m)}(t)&= H(2^{-j_1} t/m)\,\bmV^{(i,\infty)}(t)\,,\\
\bmV^{(i,\infty)(m)\ast}(\lambda)& = (2\pi)^{-1/2}\sum_{t\in\Z}\bmV^{(i,\infty)(m)}(t)e^{-i\lambda t}\,,\\
\label{eqn:wm}\bmW^{(i,i')(m)\ast}(\lambda)&= {\bmV^{(i,\infty)(m)\ast}(\lambda)}\overline{\bmV^{(i',\infty)(m)\ast}(\lambda)}^T,\\
\label{eqn:wmtilde} \tilde{\bmW}^{(i,i')(m)\ast}(\lambda)&= \sum_{t\in\Z}\bmW^{(i,i')\ast}(\lambda+2t\pi).
\end{align}
We also denote
\begin{equation}
\label{eqn:Sigma_m} \bTheta^{((i,s),(i',s'))(m)} = \int_{-\infty}^\infty \bmW^{(i,i')(m)\ast}(\lambda)\,\rme^{-\rmi(s-s')\lambda}\,\rmd\lambda,
\end{equation}
\begin{equation}
\label{eqn:Gamma_m}\Gamma_{(a,b),(a',b')}^{(i,i')(m)} =  2\pi\int_{-\pi}^\pi \overline{\tilde{\mW}_{a,a'}^{(i,i')(m)\ast}(\lambda)}\tilde{\mW}_{b,b'}^{(i,i')(m)\ast}(\lambda)\rmd\lambda {+ 2\pi\int_{-\pi}^\pi \overline{\tilde{\mW}_{a,b'}^{(i,i')(m)\ast}(\lambda)}{\tilde{\mW}_{b,a'}^{(i,i')(m)\ast}(\lambda)}\rmd\lambda}.
\end{equation}
We will prove in \ref{sec:proof:thm:gauss1:final} that $\bGamma^{(i,i')(m)}$ is the asymptotic covariance matrice of $(\bmS^{(i,j_0)(m)}, \bmS^{(i',j_0)(m)})$ as $j_0$ goes to infinity. 

We now provide some general results on the behavior of $(\bmV^{(i,j_0)(m)})$ and $(\bmV^{(i,j_0)(m)\ast})$, in much the same way as in Lemma~\ref{lem:v1}, Lemma \ref{lem:v} and Lemma \ref{lem:theta} for $(\bmV^{(i,j_0)})$ and $(\bmV^{(i,j_0)\ast})$.
\begin{lem}
\label{lem:vm1}
Suppose assumptions of Theorem~\ref{thm:gauss1} hold. Suppose $\Delta <\infty$. Then there exists $\delta_v>1/2$ such that for all $j\geq j_0$, $j-j_0\leq \Delta$,  we have
\begin{equation}
\label{eqn:vm_ast}
\norm{\bmV^{(i,j_0)(m)\ast}(\lambda)}_\infty\leq C_{vm}\, 2^{j/2}\,(1+2^{j}\lvert\lambda\rvert)^{-\delta_v}\,.
\end{equation}
with $j=j_0+\Delta-\lfloor\log_2(i)\rfloor$ and $C_{vm} < \infty$, depending on $m$, $L$, $\beta$, $\bOmega$, $\bd$, $\Delta$, $\phi(\cdot)$ and $\psi(\cdot)$.
\end{lem}
\begin{proof}
The lemma follows from Lemma \ref{lem:v1} and \cite[Lemma 5]{Roueff09central}.% entail that there exists $C_{vm}>0$ depending on $C_v$, $\delta_v$ and $c_H$ such that for all $\lambda\in(0,\pi)$
\end{proof}

\begin{lem}\label{lem:vm}
Suppose assumptions of Theorem~\ref{thm:gauss1} hold. 
For all $i=1,\dots, N$, for all $m\geq 1$, sequences $\{\mV_{a,b}^{(i,j_0)(m)},\, a,b=1,\dots,p,\, j_0\geq 0\}$ verify the following properties:
\begin{equation}
\label{eqn2i}\forall j_0\geq 0, \forall a=1,\dots, p, \, \sum_{t\in\Z} \sum_{a'=1}^p  \mV_{a,a'}^{(i,j_0)(m)}(2^{j_1}s-t)^2  < \infty\,, \\
\end{equation}
for all $(a,b)\in\{1,\dots,p\}^2$, for all $i\in\{1,\dots,N\}$,
\begin{equation}
\sup_{t\in\Z} \sum_{s\in\Z} {\mV_{a,b}^{(i,j_0)(m)}(2^{j_1}s-t)}^2 \tend_{j_0\to\infty} 0\,.
\end{equation}

Moreover for all $m\geq 1$, for all $(a,b)\in\{1,\dots,p\}^2$, for all $(i,i')\in\{1,\dots,N\}^2$ and $(s,s')\in\{0, \dots, n_{j_1}-1\}^2$, 
\begin{equation}
\label{eqn2iii} \sum_{t\in\Z} \bmV^{(i,j_0)(m)}(2^{j_1}s-t)\bmV^{(i',j_0)(m)}(2^{j_1}s'-t)^T  \tend_{j_0\to\infty}  \bTheta^{((i,s),(i',s'))(m)}\,,
\end{equation}
with $\bTheta^{((i,s),(i',s'))(m)}$ defined in \eqref{eqn:Sigma_m}.
\end{lem}
The proof is similar to that of Lemma~\ref{lem:theta} and it is thus omitted.

\subsection{Asymptotic variance of $(\bmS^{(i,j_0)},~i=1,\dots,N)_{j_0\geq 0}$}

 Lemma~\ref{lem:Sigma_m} below studies the behavior of $\{\bTheta^{((i,s),(i',s'))}, i,i'=1,\dots,N, s,s'\geq 0\}$ and $\{\bTheta^{((i,s),(i',s'))(m)}, i,i'=1,\dots,N, s,s'\geq 0\}$ when summing over the parameters $(s,s')$. It is used next to prove Lemma~\ref{lem:cov} which establishes that the asymptotic covariances of $(\bmS^{(i,j_0)},~i=1,\dots,N)_{j_0\geq 0}$ are equal to $(\bGamma^{(i,i')},~i,i'=1,\dots,N)$ when $j_0$ goes to infinity.

\begin{lem}\label{lem:Sigma_m}
Suppose conditions of Theorem \ref{thm:gauss} hold. For all $(a,b,a',b')\in\{1,\dots,p\}^4$, for all $(i, i')\in\{1,\dots,N\}$,
\begin{align}
\label{eqn:cvgceSigma}
\lim_{\ell\to\infty} \ell^{-1}\sum_{s,s'=0,\dots,\ell-1}\Theta_{a,b}^{((i,s),(i',s'))}\Theta_{a',b'}^{((i,s),(i',s'))} & = 2\pi\int_{-\pi}^\pi \overline{\tilde{\mW}_{a,b}^{(i,i')\ast}(\lambda)}\tilde{\mW}_{a',b'}^{(i,i')\ast}(\lambda)\rmd\lambda,\\
\lim_{\ell\to\infty} \ell^{-1}\sum_{s,s'=0,\dots,\ell-1}\Theta^{((i,s),(i',s'))(m)}_{a,b}\Theta^{((i,s),(i',s'))(m)}_{a',b'}
& = 2\pi\int_{-\pi}^\pi \overline{\tilde{\mW}_{a,b}^{(i,i')(m)\ast}(\lambda)}\tilde{\mW}_{a',b'}^{(i,i')(m)\ast}(\lambda)\rmd\lambda,\\
\label{eqn:cvgceSigma_m}
\end{align}
where $\tilde{\mW}_{a',b'}^{(i,i')\ast}(\cdot)$ and $\tilde{\mW}_{a',b'}^{(i,i')(m)\ast}(\cdot)$ are defined respectively in \eqref{eqn:wtilde} and in \eqref{eqn:wmtilde}.
\end{lem}
\begin{proof} We only prove \eqref{eqn:cvgceSigma}, since the proof of \eqref{eqn:cvgceSigma_m} is similar.
Quantity $\bTheta^{((i,s),(i',s'))}$ can be written as
\begin{align*}
\bTheta^{((i,s),(i',s'))}&= \int_{-\pi}^\pi \tilde \bmW^{(i,i)\ast}(\lambda)\,\rme^{-\rmi\,(s-s')\lambda}\,\rmd\lambda,
\end{align*}
Hence, setting $v=s-s'$,
\begin{multline*}
\ell^{-1}\sum_{s,s'=0,\dots,\ell-1}\Theta_{a,b}^{((i,s),(i',s'))}\Theta_{a',b'}^{((i,s),(i',s'))}\\
=\sum_{v\in\Z}\ell^{-1}(\ell-\lvert v\rvert)_+ \left(\int_{-\pi}^\pi \tilde{\mW}_{a,b}^{(i,i')\ast}(\lambda)\,\rme^{-\rmi\,v\,\lambda}\,\rmd\lambda\right)\left(\int_{-\pi}^\pi \tilde{\mW}_{a',b'}^{(i,i')\ast}(\lambda)\,\rme^{-\rmi\,v\,\lambda}\,\rmd\lambda\right)\,,
\end{multline*}
with $(\ell-\lvert v\rvert)_+=\ell-\lvert v\rvert$ if $\ell-\lvert v\rvert\geq 0$ and $0$ otherwise. Lemma \ref{lem:conv_sum} entails that, when $\ell$ goes to infinity, the above term converges to  
\[
2\pi\int_{-\pi}^\pi \overline{\tilde{\mW}_{a,b}^{(i,i')\ast}(\lambda)}\tilde{\mW}_{a',b'}^{(i,i')\ast}(\lambda)\rmd\lambda\,.
\]
This is precisely the assertion of the lemma.
\end{proof}

We can deduce from Lemma~\ref{lem:Sigma_m} that for all $i,i'=1,\dots, N$, $\bGamma^{(i,i')}$ is the asymptotic covariance between $\bmS^{(i,j_0)}$ and $\bmS^{(i',j_0)}$.

\begin{lem}\label{lem:cov}
For all $m\geq 1$, for all $(a,b,a',b')\in\{1,\dots,p\}^4$, for all $(i,i')\in\{1,\dots,N\}^2$, 
\begin{equation}
\lim_{j_0\to\infty}\cov(\mS^{(i,j_0)}_{a,b},\mS^{(i',j_0)}_{a',b'})=\Gamma^{(i,i')}_{(a,b),(a',b')}\,,
\end{equation}
with $\bGamma^{(i,i')}$ defined in \eqref{eqn:Gamma}.
\end{lem}
\begin{proof} We first decompose $\cov(\mS^{(i,j_0)}_{a,b},\mS^{(i',j_0)}_{a',b'})$ in two terms and next study separately the two terms.
\item 
\paragraph[Step 1]{Step 1. Decomposition of $\cov(\mS^{(i,j_0)}_{a,b},\mS^{(i',j_0)}_{a',b'})$}
Easy calculation shows that
\begin{equation*}
{\cov(\bmY^{({i,s,j_0})},\bmY^{({i',s',j_0})})}
= T^{((i,s),(i',s'),j_0)}_{(a,a'),(b,b')}+T^{((i,s),(i',s'),j_0)}_{(a,b'),(b,a')} - R^{((i,s),(i',s'),j_0)}_{(a,b),(a',b')},
\end{equation*}
with
\begin{multline*}
{T^{((i,s),(i',s'),j_0)}_{(a,a'),(b,b')}}
 =\left(\sum_{t_1\in\Z} \bigl(\bmV(2^{j_1}s-t_1)\bmV^{(i',j_0)}(2^{j_1}s'-t_1)^T\bigr)_{a,a'}\right) \\ \left(\sum_{t_2\in\Z}\bigl(\bmV^{(i,j_0)}(2^{j_1}s-t_2)\bmV^{(i',j_0)}(2^{j_1}s'-t_2)^T\bigr)_{b,b'}\right),
\end{multline*}
a similar expression for $T^{(i,i',s,s',j_0)}_{(a,b'),(b,a')}$, 
and
\begin{align*}
\lefteqn{R^{((i,s),(i',s'),j_0)}_{(a,b),(a',b')}}\\
 & =\sum_{a_1,a_2,a_3,a_4=1,\dots,p} \mu_{a_1,a_2,a_3,a_4}\,\sum_{t\in\Z} \mV_{a,a_1}^{(i,j_0)}(2^{j_1}s-t)\mV_{b,a_2}^{(i,j_0)}(2^{j_1}s-t)\mV_{a',a_3}^{(i',j_0)}(2^{j_1}s'-t)\\
&\pushright{\,\mV^{(i',j_0)}_{b',a_4}(2^{j_1}s'-t)}\\
&- \sum_{a_1,a_2=1,\dots,p} \sum_{t\in\Z} \mV^{(i,j_0)}_{a,a_1}(2^{j_1}s-t)\mV^{(i,j_0)}_{b,a_1}(2^{j_1}s-t){\mV^{(i',j_0)}_{a',a_2}(2^{j_1}s'-t)\mV^{(i',j_0)}_{b',a_2}(2^{j_1}s'-t)}\\
& - \sum_{a_1,a_2=1,\dots,p} \sum_{t\in\Z} \mV^{(i,j_0)}_{a,a_1}(2^{j_1}s-t)\mV^{(i,j_0)}_{b,a_2}(2^{j_1}s-t){\mV^{(i',j_0)}_{a',a_1}(2^{j_1}s'-t)\mV^{(i',j_0)}_{b',a_2}(2^{j_1}s'-t)}\\
& - \sum_{a_1,a_2=1,\dots,p} \sum_{t\in\Z} \mV^{(i,j_0)}_{a,a_1}(2^{j_1}s-t)\mV^{(i,j_0)}_{b,a_2}(2^{j_1}s-t){\mV^{(i',j_0)}_{a',a_2}(2^{j_1}s'-t)\mV^{(i',j_0)}_{b',a_1}(2^{j_1}s'-t)\,.}
\end{align*}
Hence,
\begin{multline}
\label{eqn:covS}
{\cov(\bmS^{({i,j_0})},\bmS^{({i',j_0})})}
= \frac{1}{n_{j_1}}\sum_{s=0}^{n_{j_1}-1}\sum_{s'=0}^{n_{j_1}-1}T^{((i,s),(i',s'),j_0)}_{(a,a'),(b,b')} +\frac{1}{n_{j_1}}\sum_{s=0}^{n_{j_1}-1}\sum_{s'=0}^{n_{j_1}-1}T^{((i,s),(i',s'),j_0)}_{(a,b'),(b,a')}\\
- \frac{1}{n_{j_1}}\sum_{s=0}^{n_{j_1}-1}\sum_{s'=0}^{n_{j_1}-1} R^{((i,s),(i',s'),j_0)}_{(a,b),(a',b')}.
\end{multline}
We shall now study separately the terms in the right-hand side.
\item 
\paragraph[Step 2]{Step 2. Study of $R^{((i,s),(i',s'),j_0)}_{(a,b)}$}
Let us study first $R^{((i,s),(i',s'),j_0)}_{(a,b),(a',b')}$.
Cauchy-Schwarz's inequality yields to:
\begin{multline*}
{\frac{1}{\ell}\sum_{s=0}^{\ell-1}\sum_{s'=0}^{\ell-1}\lvert{R^{((i,s),(i',s'),j_0)}_{(a,b),(a',b')}}\rvert}\\
\leq (\mu_{\infty}+3)\,\sum_{a_1,a_2,a_3,a_4=1,\dots,p} \sup_{t\in\Z} \left(\sum_{s\in\Z} \mV^{(i,j_0)}_{a,a_1}(2^{j_1}s-t)^2\right)^{1/2} \sup_{t\in\Z} \left(\sum_{s\in\Z} \mV^{(i,j_0)}_{b,a_2}(2^{j_1}s-t)\right)^{1/2}\\{\sup_{s'\in\Z}\left( \sum_{t\in\Z} \mV^{(i',j_0)}_{a',a_3}(2^{j_1}s'-t)^2\right)^{1/2}\sup_{s'\in\Z}\left( \sum_{t\in\Z} \mV^{(i',j_0)}_{b',a_4}(2^{j_1}s'-t)^2\right)^{1/2}}.
\end{multline*}
The right-hand side does not depend on $\ell$. Results \eqref{eqnthetai} and \eqref{eqnthetaii} in Lemma \ref{lem:theta} imply that it converges to 0 when $j_0$ goes to infinity.  
Hence, \begin{equation}
\label{eqn:sumR}
\frac{1}{n_{j_1}}\sum_{s=0}^{n_{j_1}-1}\sum_{s'=0}^{n_{j_1}-1}\abs{R^{((i,s),(i',s'),j_0)}_{(a,b),(a',b')}}\tend_{j_0\to\infty} 0.
\end{equation}
\item 
\paragraph[Step 3]{Step 3. Study of $T^{((i,s),(i',s'),j_0)}_{(a,b)}$}
First observe that
\begin{align*}
\lefteqn{\sum_{t_1\in\Z}\bmV(2^{j_1}s-t_1)\bmV^{(i,j_0)}(2^{j_1}s'-t_1)^T}\\
& = \int_{-\pi}^{\pi} 2^{-j_1}{\bmV^{(i,j_0)\ast}(\lambda)}\overline{\bmV^{(i',j_0)\ast}(\lambda)}^T \rme^{\rmi 2^{j_1}(s-s')\lambda}\,\rmd\lambda\\
& = \frac{1}{\sqrt{2\pi}} \int_{-\pi}^{\pi} (2\pi)^{1/2}2^{-j_1}\sum_{q=0}^{2^{j_1}-1}{\bmV^{(i,j_0)\ast}(2^{-j_1}(\lambda+2\pi q))}\overline{\bmV^{(i',j_0)\ast}(2^{-j_1}(\lambda+2\pi q))}^T \rme^{\rmi 2^{j_1}(s-s')\lambda}\rmd\lambda.
\end{align*}
Last equality was obtained by \cite[Lemma 4]{Roueff09central}, since the function $\lambda\to {\bmV^{(i,j_0)\ast}(2^{-j_1}(\lambda+2\pi\,q))}\overline{\bmV^{(i',j_0)\ast}(2^{-j_1}(\lambda+2\pi\,q))}^T\rme^{\rmi 2^{j_1}(s-s')\lambda}$ is $2\pi$-periodic.

For all functions $g_1, g_2$ in $L^2(-\pi,\pi)$, $Q\in\N$, set
\[
L_Q(g_1,g_2)=\sum_{q\in\Z}\bigl(1-\frac{\lvert q\rvert}{Q}\bigr)_+ \left(\int_{-\pi}^\pi g_1(\lambda)\rmd\lambda\right)\left(\int_{-\pi}^\pi g_1(\lambda)\rmd\lambda\right).
\]
For all $\ell\in\N$,
\[
\frac{1}{\ell}\sum_{s=0}^{\ell-1}\sum_{s'=0}^{\ell-1}T^{((i,s),(i',s'),j_0)}_{(a,b),(a',b')}
= L_\ell(g_{a,a'}^{(j_0)},g_{b,b'}^{(j_0)}),
\]
where
\[
g_{a,a'}^{(j_0)}(\lambda)=(2\pi)^{1/2}2^{-j_1}\sum_{q=0}^{2^{j_1}-1}\bigl({\mV^{(i,j_0)\ast}(2^{-j_1}(\lambda+2\pi\,q))}\overline{\mV^{(i',j_0)\ast}(2^{-j_1}(\lambda+2\pi\,q))}^T\bigr)_{a,a'},
\]
and a similar definition of $g_{b,b'}^{(j_0)}(\lambda)$.
We omit the dependence on $i,i'$ temporally to simplify notations.

Introduce 
\[
g_{a,a'}^{(\infty)}(\lambda)=(2\pi)^{1/2} \tilde \mW_{a,a'}(\lambda)= (2\pi)^{1/2}\sum_{q\in\Z}\mW_{a,a'}(\lambda+2\pi\,q).
\]
We have 
\begin{align*}
\lefteqn{\abs{L_\ell(g_{a,a'}^{(j_0)},g_{b,b'}^{(j_0)})-L_\ell(g_{a,a'}^{(\infty)},g_{b,b'}^{(\infty)})}}\\
&= \abs{L_\ell(g_{a,a'}^{(j_0)}-g_{a,a'}^{(\infty)},g_{b,b'}^{(j_0)}-g_{b,b'}^{(\infty)}) + L_\ell(g_{a,a'}^{(\infty)},g_{b,b'}^{(j_0)}-g_{b,b'}^{(\infty)})+ L_\ell(g_{a,a'}^{(j_0)}-g_{a,a'}^{(\infty)},g_{b,b'}^{(\infty)})}\\
&\leq M_\ell(g_{a,a'}^{(j_0)}-g_{a,a'}^{(\infty)})M_\ell(g_{b,b'}^{(j_0)}-g_{b,b'}^{(\infty)}) + M_\ell(g_{a,a'}^{(\infty)})M_\ell(g_{b,b'}^{(j_0)}-g_{b,b'}^{(\infty)})\\
&\qquad{+ M_\ell(g_{a,a'}^{(j_0)}-g_{a,a'}^{(\infty)})M_\ell(g_{b,b'}^{(\infty)}),}
\end{align*}
with $M_Q(g_1)$ defined in Lemma~\ref{lem:diff_sum}. Last inequality results from Cauchy-Schwarz's inequality, which entails that $L_Q(g_1,g_2)\leq M_Q(g_1)M_Q(g_2)$.  

Applying Lemma~\ref{lem:diff_sum}, 
\begin{align}
\label{eqn:Ml1}M_\ell(g_{a,a'}^{(\infty)})&\leq \left(\int_{-\pi}^\pi \abs{g_{a,a'}^{(\infty)}(\lambda)}^2\rmd\lambda\right)^{1/2},\\
\label{eqn:Ml2} M_\ell(g_{a,a'}^{(j_0)}-g_{a,a'}^{(\infty)})&\leq \left(\int_{-\pi}^\pi \abs{g_{a,a'}^{(j_0)}(\lambda)-g_{a,a'}^{(\infty)}(\lambda)}^2\rmd\lambda\right)^{1/2}.
\end{align}
The two bounds in \eqref{eqn:Ml1} and \eqref{eqn:Ml2} do not depend on $\ell$.
The right-hand side of \eqref{eqn:Ml1} is finite since by \eqref{eqn1iii} and Lemma~\ref{lem:v1}, $\abs{g_{a,a'}(\lambda)}=(2\pi)^{1/2}\abs{\tilde \mW_{a,a'}(\lambda)}\leq (2\pi)^{1/2}\,C_v (1+\lvert\lambda\rvert)^{-2\delta_v}$, with $\delta_v>1/2$.
Next, notice that
\begin{multline*}
\int_{-\pi}^\pi \abs{g_{a,a'}^{(j_0)}(\lambda)-g_{a,a'}^{(\infty)}(\lambda)}^2\rmd\lambda\\ 
=2\pi\int_{-\infty}^\infty \abs{2^{-j_1}\bigl(\bmV^{(i,j_0)\ast}(2^{-j_1}\lambda)\overline{\bmV^{(i,j_0)\ast}(2^{-j_1}\lambda)}^T\bigr)_{a,a'}-\mW_{a,a'}^{(i,i')\ast}(\lambda)}^2\rmd\lambda.
\end{multline*}
Convergence \eqref{eqn1iii} and inequality \eqref{eqn:Ml2} ensure that the integral goes to 0 when $j_0$ goes to infinity by dominated convergence. 
Therefore, the right-hand side of \eqref{eqn:Ml2} goes to 0 when $j_0$ goes to infinity.
It results that $\abs{L_\ell(g_{a,a'}^{(j_0)},g_{b,b'}^{(j_0)})-L_\ell(g_{a,a'}^{(\infty)},g_{b,b'}^{(\infty)})}$ can be bounded by a quantity which is independent of $\ell$ and which goes to 0 when $j_0$ goes to infinity. Consequently, $\abs{L_{n_{j_1}}(g_{a,a'}^{(j_0)},g_{b,b'}^{(j_0)})-L_{n_{j_1}}(g_{a,a'}^{(\infty)},g_{b,b'}^{(\infty)})}\tend_{j_0\to \infty} 0$.

Observe that\begin{align*}
L_{n_{j_1}}(g_{a,a'}^{(j_0)},g_{b,b'}^{(j_0)})&= n_{j_1}^{-1}\sum_{s=0}^{n_{j_1}-1}\sum_{s'=0}^{n_{j_1}-1}T^{((i,s),(i',s'),j_0)}_{(a,a'),(b,b')}\\
L_{n_{j_1}}(g_{a,a'}^{(\infty)},g_{b,b'}^{(\infty)}) & = n_{j_1}^{-1}\sum_{s=0}^{n_{j_1}-1}\sum_{s=0}^{n_{j_1}-1}\Theta^{((i,s),(i',s'))}_{a,a'}\Theta^{((i,s),(i',s'))}_{b,b'}\\
 & \tend_{j_0\to\infty} 2\pi\int_{-\pi}^\pi \overline{\tilde{\mW}_{a,a'}^{(i,i')\ast}(\lambda)}\tilde{\mW}_{b,b'}^{(i,i')\ast}(\lambda)\rmd\lambda,
\end{align*}
where the last convergence is given by Lemma~\ref{lem:Sigma_m}. Hence,
\begin{equation}
\label{eqn:sumT1}
n_{j_1}^{-1}\sum_{s=0}^{n_{j_1}-1}\sum_{s'=0}^{n_{j_1}-1}T^{((i,s),(i',s'),j_0)}_{(a,a'),(b,b')} \tend_{j_0\to\infty} 2\pi\int_{-\pi}^\pi \overline{\tilde{\mW}_{a,a'}^{(i,i')\ast}(\lambda)}\tilde{\mW}_{b,b'}^{(i,i')\ast}(\lambda)\rmd\lambda.
\end{equation}
Similarly,
\begin{equation}
\label{eqn:sumT2}
n_{j_1}^{-1}\sum_{s=0}^{n_{j_1}-1}\sum_{s'=0}^{n_{j_1}-1}T^{((i,s),(i',s'),j_0)}_{(a,b'),(b,a')} \tend_{j_0\to\infty} 2\pi\int_{-\pi}^\pi \overline{\tilde{\mW}_{a,b'}^{(i,i')\ast}(\lambda)}\tilde{\mW}_{b,a'}^{(i,i')\ast}(\lambda)\rmd\lambda.
\end{equation}
\item 
\paragraph[Step 4]{Step 4. End of the proof}
Lemma~\ref{lem:cov} follows from \eqref{eqn:covS}, \eqref{eqn:sumR}, \eqref{eqn:sumT1} and \eqref{eqn:sumT2}.
\end{proof}

\subsection{Convergence of $\bGamma^{(m)}$ to $\bGamma$}

We proceed to show that $\bGamma^{(i,i')(m)}$ goes to $\bGamma^{(i,i')}$ when $m$ goes to infinity, for all $i,i'=1,\dots,N$. That is, the asymptotic variance of $\bmS^{(\cdot,j_0)(m)}$ when $j_0$ goes to infinity converges to the asymptotic variance of $\bmS^{(\cdot,j_0)}$.

\begin{lem}\label{lem:Gamma2}
For all $i,i'=1,\dots, N$,
\begin{equation}
\lim_{m\to\infty}\bGamma^{(i,i')(m)} = \bGamma^{(i,i')}.
\end{equation} 
\end{lem}
\begin{proof}
To study the limit of $\bGamma^{(i,i')(m)}_{a,b}$ when $m$ goes to infinity, we will first prove that $\bGamma^{(i,i')}$ satisfies $\bGamma^{(i,i')}=\lim_{j_0\to\infty}\lim_{m\to\infty} \hat\bGamma^{(i,i',j_0)(m)}$,  where $\bigl(\hat\Gamma_{(a,b),(a',b')}^{(i,i',j_0)(m)}\bigr)_{a,b,a',b'}$ are defined by
\begin{multline}
\label{eqn:Gamma_m_hat}\hat\Gamma_{(a,b),(a',b')}^{(i,i',j_0)(m)}=  2\pi\int_{-\pi}^\pi \overline{\tilde{\hat \mW}_{a,a'}^{(i,i',j_0)(m)\ast}(\lambda)}\tilde{\mW}_{b,b'}^{(i,i',j_0)(m)\ast}(\lambda)\rmd\lambda \\+ 2\pi\int_{-\pi}^\pi \overline{\tilde{\hat \mW}_{a,b'}^{(i,i',j_0)(m)\ast}(\lambda)}{\tilde{\hat \mW}_{b,a'}^{(i,i',j_0)(m)\ast}(\lambda)}\rmd\lambda,
\end{multline}
with \begin{align*}
\tilde{\hat \bmW}^{(i,i',j_0)(m)\ast}(\lambda)&=\sum_{t\in\Z}{\hat \bmW}^{(i,i',j_0)(m)\ast}(\lambda+2t\pi),\\
\hat \bmW^{(i,i',j_0)(m)\ast}(\lambda)&= 2^{-j_1}{\bmV^{(i,j_0)(m)\ast}(2^{-j_1}\lambda)}\overline{\bmV^{(i',j_0)(m)\ast}(2^{-j_1}\lambda)}^T.
\end{align*}
Notice that \[\int_{-\pi}^\pi \overline{\tilde{\hat \mW}_{a,a'}^{(i,i',j_0)(m)\ast}(\lambda)}\tilde{\mW}_{b,b'}^{(i,i',j_0)(m)\ast}(\lambda)\rmd\lambda=\int_{-\infty}^\infty \overline{{\hat \mW}_{a,a'}^{(i,i',j_0)(m)\ast}(\lambda)}{\mW}_{b,b'}^{(i,i',j_0)(m)\ast}(\lambda)\rmd\lambda.\] When $j_0$ goes to infinity, $\hat \bmW^{(i,i',j_0)(m)\ast}(\lambda)$ converges to $\bmW^{(i,i')(m)\ast}(\lambda)$. The convergence under the integral is obtained by dominated convergence thanks to Lemma~\ref{lem:v1}. It results that  
 $\lim_{j_0\to\infty} \hat \bGamma^{(i,i',j_0)(m)}=\bGamma^{(i,i')(m)}$.

Let us now study the convergence of $(\hat \bGamma^{(i,i',j_0)(m)})$ with respect to $m$. We introduce
\begin{align}
\hat \bmW^{(i,i',j_0)\ast}(\lambda)&= 2^{-j_1}{\bmV^{(i,j_0)\ast}(2^{-j_1}\lambda)}\overline{\bmV^{(i',j_0)\ast}(2^{-j_1}\lambda)}^T,\\
\tilde{\hat \bmW}^{(i,i',j_0)\ast}(\lambda)&=\sum_{t\in\Z}{\hat \bmW}^{(i,i',j_0)\ast}(\lambda+2t\pi),
\end{align}
and
\begin{equation}
\label{eqn:Gamma_hat}
\hat \Gamma_{(a,b),(a',b')}^{(i,i',j_0)} =  2\pi\int_{-\pi}^\pi \overline{\tilde{\hat \mW}_{a,a'}^{(i,i',j_0)\ast}(\lambda)}\tilde{\hat \mW}_{b,b'}^{(i,i',j_0)\ast}(\lambda)\rmd\lambda{ + 2\pi\int_{-\pi}^\pi \overline{\tilde{\hat \mW}_{a,b'}^{(i,i',j_0)\ast}(\lambda)}{\tilde{\hat \mW}_{b,a'}^{(i,i',j_0)\ast}(\lambda)}\rmd\lambda.}
\end{equation}

Since $\bmV^{(i,s,j_0)\ast}(\cdot)$ is continuous, we can apply a convergence under the integral. Hence, for all $\lambda\in\R$, 
\begin{align*}
\bmV^{(i,j_0)(m)\ast}(\lambda)&= \frac{1}{2\pi}\int_{-\pi}^\pi \hat H(u)\bmV^{(i,j_0)\ast}(\lambda-u/m)\,\rmd u\\ &\tend_{m\to\infty} \bmV^{(i,j_0)\ast}(\lambda)\,\frac{1}{2\pi}\int_{-\pi}^\pi \hat H(u)\rmd u = \bmV^{(i,j_0)\ast}(\lambda)\,.
\end{align*}
Additionally Lemma \ref{lem:vm1} entails that
\[
\forall j_0\geq 0,\;  \sup_{m\geq 1}\sup_{\abs{\lambda}<\pi} 2^{-j/2}\norm{\bmV^{(i,s,j_0)(m)\ast}(\lambda)}_\infty(1+2^j\abs{\lambda})^{\delta_v}  < \infty\,.
\]
Consequently, $\hat \bGamma^{(i,i',j_0)(m)}$ converges uniformly in $m$ to $\hat \bGamma^{(i,i',j_0)}$.
Moreover, $\hat\bGamma^{(i,i',j_0)}$ converges to $\bGamma^{(i,i')}$ when $j_0$ goes to infinity, the convergence under the integral being obtained by continuity.

It results that \[
\lim_{m\to\infty}\lim_{j_0\to\infty}\hat \bGamma^{(i,i',j_0)(m)}=\lim_{j_0\to\infty}\lim_{m\to\infty}\hat \bGamma^{(i,i',j_0)(m)}=\bGamma^{(i,i')}.\]
 Convergence of Lemma \ref{lem:Gamma2} follows.
\end{proof}

\subsection{Sums of $\bGamma^{(i,i')}$}

Due to decomposition \eqref{eqn:defSR}, we will have to manipulate sums of covariances of $\{\bmS^{(i,j_0)},\,i=1,\dots,N\}$. By Lemma~\ref{lem:cov}, the covariances are equal to $\{\bGamma^{(i,i')},\,i,i'=1,\dots,N\}$. The objective of this section is to give some results on these sums. Rather than using expressions \eqref{eqn:Gamma}, we use convergence \eqref{eqn:cvgceSigma}. We first need the following Lemma on quantities $\{{\Theta_{a,b}^{((i,s),(i',s'))}},\,i,i'=1,\dots,N,\, s,s'\geq 0\}$.

\begin{lem}
\label{lem:Sigma}
Suppose conditions of Theorem \ref{thm:gauss} hold.
Let $(a,b)\in\{1,\dots,p\}^2$, $(i,i')\in\{1,\dots,N\}^2$, $(s,s')\in\N^2$. Introduce 
\[
{\Xi_{a,b}^{((u,s),(u',s'))}}=2^{(\Delta-u)(1-d_a-d_b)}\int_{-\infty}^\infty g_{u'-u}(\lambda; d_a+d_b)\,\rme^{-\rmi\,(k-2^{u'-u}k')\lambda}\rmd\lambda,
\]
with $g_{u'-u}(\xi;\delta)=\overline{\hat\psi(\lambda)}\,{\hat\psi(2^{u'-u}\lambda)}\lvert\lambda\rvert^{-\delta}$, and $k=i+2^{\Delta-u}(s-1)$, $k'=i'+2^{\Delta-u'}(s'-1)$.
Then, under assumptions of Theorem~\ref{thm:gauss1},
\[{\Theta_{a,b}^{((i,s),(i',s'))}}
 = \Omega_{a,b}\, \cos\bigl({\pi(d_a-d_b)/2}\bigr)\,2^{(u-\Delta)/2+(u'-\Delta)/2+\Delta\,(d_a+d_b)} \Xi_{a,b}^{((u,s),(u',s'))}.
\]
\end{lem}
\begin{proof}
Recall that $\{\bTheta_{a,b}^{((i,s),(i',s'))},\,a,b=1,\dots,p,\,i=1,\dots,N,\,s=0,\dots,n_{j_1}-1\}$ are defined in \eqref{eqn:Sigma}. For all $(a,b)\in\{1,\dots,p\}^2$, \[\Theta_{a,b}^{((i,s),(i',s'))}=\overline{\Theta_{b,a}^{((i',s'),(i,s))}} = \frac{1}{2}\bigl(\Theta_{a,b}^{((i,s),(i',s'))}+\overline{\Theta_{b,a}^{((i,s),(i',s'))}}\bigr).\]
Hence,
\[\bTheta^{((i,s),(i',s'))}=\int_{-\infty}^\infty \frac{1}{2}\bigl(\bmW^{(i,i')\ast}(\lambda)+\overline{\bmW^{(i',i)\ast}(\lambda)}^T\bigr)\rme^{-\rmi(s-s')\lambda}\rmd\lambda.
\]
Replacing $\bmW^{(i,i')\ast}(\lambda)$ and $\bmW^{(i',i)\ast}(\lambda)$ by their expression \eqref{eqn:wdef}, we get
\[{\Theta_{a,b}^{((i,s),(i',s'))}}
 = \Omega_{a,b}\, \cos\bigl({\pi(d_a-d_b)/2}\bigr)\,2^{(u-\Delta)/2+(u'-\Delta)/2+\Delta\,(d_a+d_b)} \Xi_{a,b}^{((u,s),(u',s'))}
\]
with 
\begin{align*}
\lefteqn{\Xi_{a,b}^{((u,s),(u',s'))}}\\*
&=\int_{-\infty}^\infty\overline{\hat\psi(2^{u-\Delta}\lambda)}\,{\hat\psi(2^{u'-\Delta}\lambda)}\lvert\lambda\rvert^{-d_a-d_b}\,\rme^{-\rmi(\,2^{u-\Delta}(i-1)-\,2^{u'-\Delta}(i'-1))\lambda}\,\rme^{-\rmi(s-s')\lambda}\rmd\lambda,\\
&=2^{(\Delta-u)(1-d_a-d_b)}\int_{-\infty}^\infty\overline{\hat\psi(\lambda)}\,{\hat\psi(2^{u'-u}\lambda)}\lvert\lambda\rvert^{-d_a-d_b}\,\rme^{-\rmi\,(i+2^{\Delta-u}s)\lambda+\rmi\,(2^{\Delta-u'}s'+i')2^{u'-u}\lambda}\,\rmd\lambda.
\end{align*}
\end{proof}

We are now in a position to give a useful result on $\bGamma$. Namely, we consider the sum of $\{\bGamma^{(i,i')},\, i=2^{\Delta-u},\dots 2^{\Delta-u+1}-1,\,i'=2^{\Delta-u'},\dots, 2^{\Delta-u'+1}-1\}$, which corresponds to the contribution of the scales $(j,j')=(j_0+u,j_0+u')$ to the asymptotic variance $\bV$ of the sample wavelet covariance.

\begin{lem}
\label{lem:Gamma_final}
Suppose conditions of Theorem \ref{thm:gauss} hold. For all $(a,b,a',b')\in\{1,\dots,p\}^4$, for all $\Delta\in\N$, $(u,u')\in\{0,\dots,\Delta\}^2$,
\[
\sum_{i=2^{\Delta-u}}^{2^{\Delta-u+1}-1}\sum_{i'=2^{\Delta-u'}}^{2^{\Delta-u'+1}-1}\Gamma_{(a,b),(a',b')}^{(i,i')}=2^{\Delta-u}\,\bigl(\GIG_{(a,a'),(b,b')}(u)+\GIG_{(a,b'),(a',b)}(u)\bigr),
\]
where $\GIG(u)$ is defined in \eqref{eqn:Itilde_mat}.
\end{lem}
\begin{proof}
Quantities $\Gamma^{(i,i')}_{(a,b),(a',b')}$ can be expressed as:
\begin{multline}
\label{eqn:V2}
\Gamma^{(i,i')}_{(a,b),(a',b')}=\lim_{\ell\to\infty} \ell^{-1}\sum_{s=0}^\ell \sum_{s'=0}^\ell \Bigl(\Theta_{a,a'}^{((i,s),(i',s'))}\Theta_{b,b'}^{((i,s),(i',s'))}+\Theta_{a,b'}^{((i,s),(i',s'))}\Theta_{b,a'}^{((i,s),(i',s'))}\Bigr),
\end{multline}
where $\{\bTheta_{a,b}^{((i,s),(i',s'))},\,a,b=1,\dots,p,\,i=1,\dots,N,\,s=0,\dots,n_{j_1}-1\}$ are noted in \eqref{eqn:Sigma}. 
Lemma~\ref{lem:Sigma} yields
\begin{equation}
\label{eqn:V3}{\Theta_{a,b}^{((i,s),(i',s'))}}
 = \Omega_{a,b}\, \cos\bigl({\pi(d_a-d_b)/2}\bigr)\,2^{(u-\Delta)/2+(u'-\Delta)/2+\Delta\,(d_a+d_b)} \Xi_{a,b}^{((u,s),(u',s'))},
\end{equation}
with 
\[
{\Xi_{a,b}^{((u,s),(u',s'))}}=2^{(\Delta-u)(1-d_a-d_b)}\int_{-\infty}^\infty g_{u'-u}(\lambda; d_a+d_b)\,\rme^{-\rmi\,(k-2^{u'-u}k')\lambda}\rmd\lambda,
\]
and $g_{u'-u}(\xi;\delta)=\overline{\hat\psi(\lambda)}\,{\hat\psi(2^{u'-u}\lambda)}\lvert\lambda\rvert^{-\delta}$, $k=i+2^{\Delta-u}(s-1)$, $k'=i'+2^{\Delta-u'}(s'-1)$.

To get all values in $\Z$ from $k-2^{u'-u}k'$, we introduce $\tau\in\{0,\dots, 2^{-(u'-u)}-1\}$. Then, when $i$, $i'$, $s$ and $s'$ vary respectively in $\{2^{\Delta-u},\dots{2^{\Delta-u+1}-1}\},$ $\{2^{\Delta-u'},\dots {2^{\Delta-u'+1}-1}\}$, $\{0,\dots \ell-1\}$ ad $\{0,\dots \ell-1\}$, quantity $q=k-2^{u'-u}k'+2^{u'-u}\tau$ takes all relative integers values in $\{-Q,\dots, Q\}$, with $Q=2^{\Delta-u}(\ell-1)$.

We have 
\begin{align*}
\lefteqn{\sum_{i=2^{\Delta-u}}^{2^{\Delta-u+1}-1}\sum_{i'=2^{\Delta-u'}}^{2^{\Delta-u'+1}-1}\ell^{-1}\sum_{s=0}^{\ell-1} \sum_{s'=0}^{\ell-1} \Xi_{a,b}^{((u,s),(u',s'))}\Xi_{a',b'}^{((u,s),(u',s'))}}\\
&= \frac{Q}{\ell}\!\sum_{-Q\leq q\leq Q}\! \bigl(1-\frac{q}{Q}\bigr)_+ 2^{(\Delta-u)(2-d_a-d_b-d_{a'}-d_{b'})}\Bigl(\int_{-\infty}^\infty \sum_{\tau=0}^{2^{-(u'-u)}-1}\!\!g(\lambda; d_a+d_b)\rme^{-\rmi\,(q-2^{u'-u}\tau)\lambda)}\rmd\lambda\Bigr)\\
&\pushright{\Bigl(\int_{-\infty}^\infty \sum_{\tau=0}^{2^{-(u'-u)}-1} g(\lambda; d_{a'}+d_{b'})\,\rme^{-\rmi\,(q-2^{u'-u}\tau)\,\lambda}\rmd\lambda\Bigr)}\\
&= 2^{\Delta-u}\!\sum_{-Q\leq q\leq Q}\!\bigl(1-\frac{q}{Q}\bigr)_+ 2^{(\Delta-u)(2-d_a-d_b-d_{a'}-d_{b'})}\Bigl(\int_{-\pi}^\pi 2^{-(u'-u)/2} \tilde D_{u'-u,\infty}(\lambda; d_a+d_b)\rme^{\rmi\,q\lambda}\rmd\lambda\Bigr)\\
&\pushright{\Bigl(\int_{-\pi}^\pi 2^{-(u'-u)/2}\tilde  D_{u'-u,\infty}(\lambda; d_a+d_b)\,\rme^{\rmi\,q\lambda}\rmd\lambda\Bigr)}
\end{align*}
since $\tilde D_{u'-u;\infty}(\lambda; d_a+d_b)=\sum_{v=0}^{2^{u'-u}-1}2^{(u'-u)/2}\sum_{t\in\Z} g_{u'-u}(\lambda+2t\pi)\rme^{\rmi \,2^{u'-u}\tau(\lambda+2t\pi)}$.
Applying Lemma~\ref{lem:conv_sum}, we obtain 
\begin{align}
\lefteqn{\lim_{\ell\to\infty}\sum_{i=2^{\Delta-u}}^{2^{\Delta-u+1}-1}\sum_{i'=2^{\Delta-u'}}^{2^{\Delta-u'+1}-1}\ell^{-1}\sum_{s=0}^\ell \sum_{s'=0}^\ell \Xi_{a,b}^{((u,s),(u',s'))}\Xi_{a',b'}^{((u,s),(u',s'))}}\\
&= 2^{\Delta-u}\,2^{(\Delta-u)(2-d_a-d_b-d_{a'}-d_{b'})+u-u'}
(2\pi)\,\int_{-\pi}^\pi \overline{\tilde D_{u'-u;\infty}(\lambda; d_a+d_b)}\tilde D_{u'-u;\infty}(\lambda; d_{a'}+d_{b'})\,\rmd\lambda\\
\label{eqn:V4}&= 2^{\Delta-u'+(\Delta-u)(2-d_a-d_b-d_{a'}-d_{b'})}\,
\tilde{\mathcal{I}}_{u'-u}(d_a+d_b,d_{a'}+d_{b'})K(d_a+d_b)K(d_{a'}+d_{b'}).
\end{align}
Lemma~\ref{lem:Gamma_final} results from \eqref{eqn:V2}, \eqref{eqn:V3} and \eqref{eqn:V4}.
\end{proof}

\section{Proof of Theorem \ref{thm:gauss}}
\label{proof:thm:gauss}

The asymptotic normality is given by Theorem \ref{thm:gauss1} below. 
Proposition~\ref{prop:approx} enables to approximate $\{2^{-j(d_a+d_b)}\sigma_{a,b}(j),\, a,b=1,\dots,p,\, j\geq 0\}$  by $\bG$ and hence entails Theorem~\ref{thm:gauss}.

\begin{thm}
\label{thm:gauss1}
Suppose Assumptions \ref{ass:LRD}--\ref{ass:Aast} and \ref{ass:Wcompact} to \ref{ass:Wparameters} hold. Let $2^{-j_0\beta}\to 0$ and $N_X^{-1}2^{j_0}\to 0$.
Then 
\begin{multline}
\left\{\sqrt{n_{j_0+u}}\,\vect{\bLambda_{j_0}(\bd)^{-1}\bigl(\hat \bSigma(j_0+u)- \bSigma(j_0+u)\bigr)\bLambda_{j_0}(\bd)^{-1}}, ~ u=0,\dots,\Delta\right\}\\
\tend^{\mathcal L}_{j_0\to\infty} \{\bQ(u),\,u=0,\dots,\Delta\}
\end{multline}
where $\bQ(\cdot)$ is the centered Gaussian process defined in Theorem \ref{thm:gauss}.
\end{thm}

The construction of the proof is adapted from that of \cite[Theorem 2]{Roueff09asymptotic}. The proof has been structured as follows. \ref{proof:coeff} proposes a writing of wavelet coefficients as decimated linear processes, and provides an approximation by a $m$-dependent decimated linear processes, $m\geq 0$. Notations and technical results on the decimated decompositions are stated in \ref{proof:v}. They are useful for applying the propositions of \cite{Roueff09central}, which lead to the asymptotic normality described in Theorem~\ref{thm:gauss1}. The current section deals with this step.

Let us use the notations introduced in \ref{sec:indexing}. The sample wavelet covariances satisfy \eqref{eqn:defSR}. It results that the vector of empirical covariances at different scales can be written as
\begin{equation}
\label{eqn:matriciel}
2^{-j_0\,(d_a+d_b)}\begin{pmatrix}
\hat \sigma_{a,b}(j_0)\\ \vdots 
\\ \hat \sigma_{a,b}(j_1)\end{pmatrix} = \sqrt{n_{j_1}}\, \bB_{j_0} \, \begin{pmatrix}
\mS_{a,b}^{({1,j_0})}\\ \vdots \\ \mS_{a,b}^{({N,j_0})}\end{pmatrix} + \begin{pmatrix}
\mR_{a,b}({j_0})\\ \vdots \\ \mR_{a,b}({j_1})\end{pmatrix}
\end{equation}
with 
\[
\bB_{j_0}=\begin{pmatrix}
 0\dots & \dots & \dots~ 0 &\overbrace{n_{j_0}^{-1}\dots n_{j_0}^{-1}}^{2^\Delta\text{ times}} \\
0 \dots & \dots~ 0 & \overbrace{n_{j_0+1}^{-1}\dots n_{j_0+1}^{-1}}^{2^{\Delta-1}\text{ times}}& 0\dots0   \\
\vdots & &\vdots & \vdots\\
n_{j_1}^{-1}~0 \dots  & \dots & \dots & \dots 0
\end{pmatrix}.
\]
The objective is to show that $\vect{\hat\sigma_{a,b}(j),~{a,b=1,\dots,p}}$ is asymptotically Gaussian when $j$ goes to infinity. The proof is divided into the following steps:
\begin{itemize}[topsep=-10pt]
%\item \ref{sec:v} gives some preliminary results necessary for the next steps;
\item \ref{proof:thm:S} establishes that the vector $\vect{\mS_{a,b}({i,j_0}),~{a,b=1,\dots,p}}_{i=1,\dots,N}$ is asymptotically Gaussian when $j_0$ goes to infinity;  the proof is base on the approximation by $m$-dependent processes introduced in \ref{sec:mdep}.
 \item \ref{sec:R} proves that the terms $\vect{\mR_{a,b}(j),~{a,b=1,\dots,p}}_{j_0\leq j\leq j_1}$ are negligible.
 \end{itemize}
\ref{sec:proof:thm:gauss1:final} compiles all elements above to prove Theorem \ref{thm:gauss1}.

\subsection{Asymptotic normality of $\bmS^{({i,j})}$}
\label{proof:thm:S}

The asymptotic normality of $\vect{\mS^{({i,j_0})}}$ is given by the following proposition.

\begin{prop}\label{prop:Sij}
Under conditions of Theorem~\ref{thm:gauss1}, 
\begin{equation}
\left\{\vect{\bmS^{(i,j_0)}}, ~ i=1,\dots,N\right\}\tend^{\mathcal L}_{j_0\to\infty} \{\bQ^S(i),\,i=1,\dots,N\},
\end{equation}
where $\bQ^S(\cdot)$ is a centered Gaussian process with covariance function $\cov{(Q^S_{a,b}(i),\,Q^S_{a',b'}(i'))}=
\Gamma_{(a,b),(a',b')}^{(i,i')}$ defined in~\eqref{eqn:Gamma}. 
\end{prop}

Recall that $\bmS^{(i,j_0)}=\frac{1}{n_{j_1}}\sum_{s=0}^{n_{j_1}-1}\bmY^{(i,s,j_0)}$. The steps of the proof of the asymptotic normality of $\vect{\mS^{({i,j_0})}}$ are the following:
\begin{itemize}[topsep=-10pt]
\item We approximate $\{\mY_{a,b}^{(i,\cdot,j)},\,i=1,\dots N\}$ by the $m$-dependent process $\{\mY_{a,b}^{(i,\cdot,j)(m)},\,i=1,\dots N\}$ defined in \ref{sec:mdep}.
 We establish that $\{\mY^{({i,\cdot,j})(m)}_{a,b},\,i=1,\dots N\}$ is asymptotically normal when $j$ goes to infinity, using  \cite[Proposition~2]{Roueff09central}.
\item We obtain the asymptotic normality of $\sum_{a,b=1,\dots,p} u_{a,b} \mS_{a,b}^{({i,j_0})(m)}=n_{j}^{-1} \sum_{s=0}^{n_{j}-1} \mY^{({i,s,j})(m)}_{a,b}$, thanks to  \cite[Proposition~3]{Roueff09central}.
\item The asymptotic normality for $\sum_{a,b=1,\dots,p} \nu_{a,b} \mS_{a,b}^{({i,j_0})}$ is obtained by letting $m$ goes to infinity, using \cite[Theorem 3.2]{Billingsley}. 
\end{itemize}

\subsubsection{First step, approximation by a $m$-dependent process}

\label{sec:step2}

We study variables $(\bmY^{(i,s,j_0)(m)})_{i=1,\dots,N}$ defined in \eqref{eqn:Ym}, in \ref{sec:mdep}.
The objective of this step is to prove that variables $(\bmY^{(i,s,j_0)(m)})_{i=1,\dots,N}$ are asymptotically Gaussian. They are defined from variables $(\mZ_{a}^{(i,s,j_0)(m)},\, a=1,\dots,p)_{i=1,\dots,N}$ by $\bmY^{(i,s,j_0)(m)}=\bmZ^{(i,s,j_0)(m)}\bmZ^{(i,s,j_0)(m)T}$. We will study first the behavior of  $(\bmZ^{(i,s,j_0)(m)})_{i=1,\dots,N}$ and next deduce that variables $(\bmY^{(i,s,j_0)(m)})_{i=1,\dots,N}$ are asymptotically Gaussian.

For all $a=1,\dots, p$ and $s\in\N$, let
\[
\Z_{a}^{(s,j_0)(m)}=\left(
\mZ_{a}^{(1,s,j_0)(m)}, ~~ {\mZ_{a}^{(2,s,j_0)(m)}},   \dots, ~~ \mZ_{a}^{(N,s,j_0)(m)}\right)^T.
\]
$(\mZ_{a}^{(i,s,j_0)(m)},\, a=1,\dots,p)_{i=1,\dots,N}$ are $m$-dependent decimated linear processes in $\R^N$.
By \cite[Proposition~2]{Roueff09central}, we get that $
 \vect{\Z_{a}^{(s,j_0)(m)}, a=1,\dots,p}$
converges in distribution to $ \vect{\Z_{a}^{(s,\infty)(m)}, a=1,\dots,p}
$, which follows a centered Gaussian distribution with $\Z_{a}^{(s,\infty)(m)}=(\mZ_{a}^{(i,s,\infty)(m)})_{i=1,\dots, N}$ and \[
\cov\bigl(\mZ_{a}^{(i,s,\infty)(m)},\mZ_{b}^{(i',s',\infty)(m)}\bigr)=\Theta_{a,b}^{((i,s),(i',s'))(m)}.
\]
For $s\in\N$, $\mathbf{\nu}\in\R^{p\times p}$, write \begin{align*}
\Y_{a,b}^{(s,j_0)(m)}&=\left(\mY_{a,b}^{(1,s,j_0)(m)}, \dots,\, \mY_{a,b}^{(N,s,j_0)(m)}\right)^T\\
\bY^{(s,j_0)(m)}(\bnu)& =\sum_{a,b=1,\dots,p} \nu_{a,b}\Y_{a,b}^{(s,j_0)(m)}=\left(Y^{(1,s,j_0)(m)}(\bnu), \dots,\, Y^{(N,s,j_0)(m)}(\bnu)\right)^T
\end{align*}
 The continuous mapping theorem implies that, when $j_0$ goes to infinity,  $(\bY^{({s,j_0})(m)}(\bnu))_{s=0,\dots,n_{j_1}-1}$ converges in distribution to $ \left(\bY^{(s,\infty)(m)}(\bnu)\right)_{s=0,\dots,n_{j_1}-1}$ given by \[
 \bY^{(s,\infty)(m)}(\bnu)=\sum_{a,b=1,\dots,p} \nu_{a,b}\Y_{a,b}^{(s,\infty)(m)}=\left(Y^{(1,s,\infty)(m)}(\bnu), \dots,\, Y^{(N,s,\infty)(m)}(\bnu)\right)^T,
 \] with $\Y_{a,b}^{(s,\infty)(m)}=(\mY_{a,b}^{(i,s,\infty)(m)})_{i=1,\dots, N}$ and $\mY_{a,b}^{(i,s,\infty)(m)}=\mZ_{a}^{(i,s,\infty)(m)}\mZ_{b}^{(i,s,\infty)(m)}$.

\subsubsection{Second step: asymptotic normality of $\bmS^{(i,s,j_0)(m)}$}

We first prove that conditions of \cite[Proposition 3]{Roueff09central} are satisfied by $\{\bY^{(s,j_0)(m)},\,s\in\N,\, j_0\geq 0\}$.

\begin{lem}\label{lem:Ym}
For all $m\geq 1$, for all $\bnu\in\R^{p\times p}$,
\begin{gather}
\label{eqn3i}\sup_{i=1,\dots,N}\sup_{s\geq 0}\sup_{j_0\geq 0}\E[Y^{({i,s,j_0})(m)}(\bnu)]  <\infty,\\
\forall s,s'\geq 0,\; \lim_{j_0\to\infty}\cov(\bY^{({s,j_0})(m)}(\bnu),\bY^{({s',j_0})(m)}(\bnu))  = \cov( \bY^{(s,\infty)(m)}(\bnu), \bY^{(s',\infty)(m)}(\bnu)),\\
\label{eqn3ii}\\
\label{eqn3iii}\lim_{\ell\to\infty}\lim_{j_0\to\infty}\cov(\ell^{-1/2}\sum_{s=0}^{\ell-1}\bY^{({s,j_0})(m)}(\bnu))  = \bGamma^{(m)}(\bnu),
\end{gather}
with $\bGamma^{(m)}(\bnu)=\sum_{a,b=1,\dots,p} \sum_{a',b'=1,\dots,p} \nu_{a,b}\nu_{a',b'}\bigl(\Gamma^{(i,i')(m)}_{(a,b),(a',b')}\bigr)_{i,i'=1,\dots,N}$ and $\bGamma^{(m)}$ defined in \eqref{eqn:Gamma_m}.
\end{lem}
\begin{proof} \item 
\paragraph[Proof of dominance]{Proof of \eqref{eqn3i}}%~\\
Assertion \eqref{eqn3i} follows from the fact that $\E[\Y^{({s,\infty})(m)}(\bnu)]=\sum_{a,b=1,\dots,p} \nu_{a,b}\bigl(\Theta^{((i,s),(i,s))(m)}_{a,b}\bigr)_{i=1,\dots,N}$.
\item 
\paragraph[Proof of convergence of covariance]{Proof of \eqref{eqn3ii}}%~\\
Vector $\begin{pmatrix}
\Z_{a}^{(s,\infty)(m)},&  \Z_{b}^{(s,\infty)(m)}, & \tilde \Z_{a'}^{(s',\infty)(m)},&  \Z_{b'}^{(s',\infty)(m)}
\end{pmatrix}^T$ follows a centered Gaussian distribution. We can therefore use Isserlis's theorem. We get 
\begin{align*}
\lefteqn{\E(Y^{({i,s,\infty})(m)}(\bnu) Y^{({i,s',\infty})(m)}(\bnu))}\\
&=\sum_{a,b,a',b'=1,\dots,p}\nu_{a,b}\nu_{a',b'}{\E(\mZ_{a}^{(i,s,\infty)(m)}\mZ_{b}^{(i,s,\infty)(m)}\mZ_{a'}^{(i,s',\infty)(m)}\mZ_{b'}^{(i,s',\infty)(m)})}\\
&=\sum_{a,b,a',b'=1,\dots,p}\nu_{a,b}\nu_{a',b'}\left[\Theta^{((i,s),(i,s))(m)}_{a,b}\Theta^{((i',s'),(i',s'))(m)}_{a',b'}\right.\\
&\qquad\qquad\qquad\qquad{\left.+\Theta^{((i,s),(i',s'))(m)}_{a,a'}\Theta^{((i,s),(i',s'))(m)}_{b,b'}+\Theta^{((i,s),(i',s'))(m)}_{a,b'}\Theta^{((i,s),(i',s'))(m)}_{b,a'}\right]\,.}
\end{align*}
It results that:
\begin{multline}
{\cov(Y^{({i,s,\infty})(m)}(\bnu),Y^{({i',s',\infty})(m)}(\bnu))}\\
=\sum_{a,b,a',b'=1,\dots,p}\nu_{a,b}\nu_{a',b'}\left[\Theta^{((i,s),(i',s'))(m)}_{a,a'}\Theta^{((i,s),(i',s'))(m)}_{b,b'}+\Theta^{((i,s),(i',s'))(m)}_{a,b'}\Theta^{((i,s),(i',s'))(m)}_{b,a'}\right].
\end{multline}
We deduce that it is sufficient to prove that, when $j_0$ goes to infinity,
$
\cov(Y^{({i,s,j_0})(m)}(\bnu),Y^{({i',s',j_0})(m)}(\bnu))$ converges to \begin{multline*}\sum_{a,b,a',b'=1,\dots,p}\nu_{a,b}\nu_{a',b'}\left[\Theta^{((i,s),(i',s'))(m)}_{a,a'}\Theta^{((i,s),(i',s'))(m)}_{b,b'}+\Theta^{((i,s),(i',s'))(m)}_{a,b'}\Theta^{((i,s),(i',s'))(m)}_{b,a'}\right]
\end{multline*} 
to obtain equality \eqref{eqn3ii}.

Following the proof of Lemma~\ref{lem:cov}, we can write $\cov(Y^{({i,s,j_0})(m)}(\bnu),Y^{({i',s',j_0})(m)}(\bnu))$ as
\begin{multline} 
\label{eqn:covm}
\cov(Y^{({i,s,j_0})(m)}(\bnu),Y^{({i',s',j_0})(m)}(\bnu))\\ =\sum_{a,b,a',b'=1,\dots,p}\nu_{a,b}\nu_{a',b'}
\Bigl[T^{(i,i',s,s',j_0)(m)}_{(a,a'),(b,b')} +T^{(i,i',s,s',j_0)(m)}_{(a,b'),(b,a')}
- \sum_{s'=0}^{n_{j_1}-1} \nu_{a,b}\nu_{a',b'}R^{(i,i',s,s',j_0)(m)}_{(a,b),(a',b')}\Bigr],
\end{multline}
with \begin{multline*}
{T^{(i,i',s,s',j_0)(m)}_{(a,a'),(b,b')}(\bnu)}
 =\left(\sum_{t_1\in\Z} \bigl(\bmV^{(i,j_0)(m)}(2^{j_1}s-t_1)\bmV^{(i',j_0)(m)}(2^{j_1}s'-t_1)^T\bigr)_{a,a'}\right)\\ {\left(\sum_{t_2\in\Z}\bigl(\bmV^{(i,j_0)(m)}(2^{j_1}s-t_2)\bmV^{(i',j_0)(m)}(2^{j_1}s'-t_2)^T\bigr)_{b,b'}\right),}
 \end{multline*}
and
$
\lim_{j_0\to\infty} R^{(i,i',s,s',j_0)(m)}_{(a,b),(a',b')}=0.
$
 Conclusion follows from \eqref{eqn2iii}.
\item \paragraph[Proof of convergence of the covariance of the mean]{Proof of \eqref{eqn3iii}}%~\\
The proof is based on decomposition \eqref{eqn:covm} of $\cov(Y^{({i,s,j_0})(m)},Y^{({i',s',j_0})(m)})$. Following the step 2 and the step 3 of the proof of Lemma~\ref{lem:cov}, we can establish that
\begin{align*}
n_{j_1}^{-1}\sum_{s=0}^{n_{j_1}-1}\sum_{s'=0}^{n_{j_1}-1}R^{(i,i',s,s',j_0)(m)}_{(a,a'),(b,b')} &\tend_{j_0\to\infty}0,\\
n_{j_1}^{-1}\sum_{s=0}^{n_{j_1}-1}\sum_{s'=0}^{n_{j_1}-1}T^{(i,i',s,s',j_0)(m)}_{(a,b),(a',b')} &\tend_{j_0\to\infty} 2\pi\int_{-\pi}^\pi \overline{\tilde{\mW}_{a,a'}^{(i,i')(m)\ast}(\lambda)}\tilde{\mW}_{b,b'}^{(i,i')(m)\ast}(\lambda)\rmd\lambda,\\
n_{j_1}^{-1}\sum_{s=0}^{n_{j_1}-1}\sum_{s'=0}^{n_{j_1}-1}T^{(i,i',s,s',j_0)(m)}_{(a,b'),(b,a')} &\tend_{j_0\to\infty} 2\pi\int_{-\pi}^\pi \overline{\tilde{\mW}_{a,b'}^{(i,i')(m)\ast}(\lambda)}\tilde{\mW}_{b,a'}^{(i,i')(m)\ast}(\lambda)\rmd\lambda.
\end{align*}
The proof is very similar and it is not detailed here for the sake of concision. It relies on Lemma~\ref{lem:diff_sum}, Lemma~\ref{lem:vm1}, Lemma~\ref{lem:vm} and Lemma~\ref{lem:Sigma_m}.
\end{proof}

We are now in a position to give the asymptotic normality of variables $\bmS^{(i,s,j_0)(m)}$, defined in \eqref{eqn:Sm}.

\begin{prop} \label{prop:Sm}
Under conditions of Theorem~\ref{thm:gauss1}, 
\begin{equation}
\left\{\vect{\bmS^{(i,j_0)(m)}}, ~ i=1,\dots,N\right\}\tend^{\mathcal L}_{j_0\to\infty} \{\bQ^{(m)}(i),\,i=1,\dots,N\},
\end{equation}
where $\bQ^{(m)}(\cdot)$ is a centered Gaussian process with covariance function $\cov{(Q^{(m)}_{a,b}(i),\,Q^{(m)}_{a',b'}(i'))}=
\Gamma^{(i,i')(m)}_{(a,b),(a',b')}$ defined in~\eqref{eqn:Gamma_m}.
\end{prop}
\begin{proof}
With results of Lemma~\ref{lem:Ym}, we can apply \cite[Proposition~3]{Roueff09central}, which gives the proposition. 
\end{proof}

\subsubsection{Third step: proof of Proposition \ref{prop:Sij}}

For $i=1,\dots, N$, when $j_0$ goes to infinity, $\vect{\bmS^{(i,j_0)(m)}}=\vect{n_j^{-1/2}\sum_{s=0}^{n_j-1}\bmY^{(i,s,j_0)(m)}}$ converges to a $\mathcal N_{p^2}(0,\bGamma^{(m)})$ distribution by Proposition \ref{prop:Sm}.  We want to deduce a similar result for variables $(\bmS^{(i,j_0)})$. 
Lemma \ref{lem:Gamma2} establishes that $\lim_{m\to\infty}\bGamma^{(m)}=\bGamma$. Hence Lemma \ref{lem:Gamma} below entails that we can apply \cite[Theorem 3.2]{Billingsley}. Proposition~\ref{prop:Sij} follows.

It remains to prove the following lemma.

\begin{lem}\label{lem:Gamma}
For all $i,i'=1,\dots, N$,
\begin{equation}
\label{eqn4ii}\lim_{m\to\infty}\lim_{j_0\to\infty} \var\left(\bmS^{({i,j_0})(m)} - \bmS^{({i,j_0})}\right)=0.
\end{equation} 
\end{lem}
\begin{proof}
Lemma \ref{lem:cov} and Proposition \ref{prop:Sm} state respectively that for all $(i,i')\in\{1,\dots,N\}^2$, for all $(a,b,a',b')\in\{1,\dots,p\}^4$, we have $\lim_{j_0\to\infty} \cov(\mS_{a,b}^{({i,j_0})},\mS_{a',b'}^{({i',j_0})})=\Gamma_{(a,b),(a',b')}^{(i,i')}$ and $\lim_{j_0\to\infty} \cov(\mS_{a,b}^{({i,j_0})(m)}\mS_{a',b'}^{({i',j_0})(m)})=\Gamma_{(a,b),(a',b')}^{(i,i')(m)}$. Additionnally, by
Lemma \ref{lem:Gamma2}, $\lim_{m\to\infty} \bGamma^{(i,i')(m)}=\bGamma^{(i,i')}$. Consequently,
\[
\lim_{m\to\infty}\lim_{j_0\to\infty}\cov(\mS_{a,b} ^{(i,j_0)(m)},\mS_{a',b'}^{(i,j_0)(m)})=\lim_{j_0\to\infty}\cov(\mS_{a,b}^{(i,j_0)},\mS_{a',b'}^{(i,j_0)})=\Gamma_{(a,b),(a',b')}^{(i,i')}.
\]

Hence it is sufficient to prove that $\lim_{m\to\infty}\lim_{j_0\to\infty}\cov(\mS_{a,b} ^{(i,j_0)(m)},\mS_{a',b'}^{(i',j_0)})= \Gamma_{(a,b),(a',b')}^{(i,i')}.$ To this aim,  we will prove that limits can be inverted, that is, $\lim_{m\to\infty}\lim_{j_0\to\infty}\cov(\mS_{a,b} ^{(i,j_0)(m)},\mS_{a',b'}^{(i',j_0)})=\lim_{j_0\to\infty}\lim_{m\to\infty}\cov(\mS_{a,b}^{(i,j_0)(m)},\mS_{a',b'}^{(i',j_0)})$.

We can establish that $\cov(\mS_{a,b}^{(i,j_0)(m)},\mS_{a',b'}^{(i',j_0)})$ converges as $j_0$ goes to infinity to \begin{equation}
\label{eqn:cov_Gamma}
2\pi\!\int_{-\pi}^\pi \overline{\tilde{\mW}_{a,a'}^{(i,i',j_0)(m,\infty)\ast}\!(\lambda)}\tilde{\mW}_{b,b'}^{(i,i',j_0)(m,\infty)\ast}\!(\lambda)\rmd\lambda + 2\pi\!\int_{-\pi}^\pi \overline{\tilde{\mW}_{a,b'}^{(i,i',j_0)(m,\infty)\ast}\!(\lambda)}{\tilde{\mW}_{b,a'}^{(i,i',j_0)(m,\infty)\ast}\!(\lambda)}\rmd\lambda,
\end{equation}
with \begin{align*}
\tilde{\mW}_{a,b'}^{(i,i',j_0)(m,\infty)\ast}(\lambda)&=\sum_{t\in\Z}{\mW}_{a,b'}^{(i,i',j_0)(m,\infty)\ast}(\lambda+2t\pi)\\ 
\text{and~~} {\mW}_{a,b'}^{(i,i',j_0)(m,\infty)\ast}(\lambda)&= 2^{-j_1}\sum_{a'=1}^p \mV_{a,a'}^{(i,j_0)(m)\ast}(2^{-j_1}\lambda)\mV_{b,a'}^{(i',j_0)\ast}(2^{-j_1}\lambda).
\end{align*}
 The proof is very similar to the one carried out for proving \eqref{eqn3iii} and it is thus omitted. Additionally \eqref{eqn:cov_Gamma} converges to $\Gamma_{(a,b),(a',b')}^{(i,i')}$ when $j_0$ goes to infinity.

Next we can study the limit of $\cov(\mS_{a,b}^{(i,j_0)(m)},\mS_{a',b'}^{(i',j_0)})$ as $m$ goes to infinity and state that it converges uniformly to $\hat \Gamma_{(a,b),(a',b')}^{(i,i',j_0)}$ defined in \eqref{eqn:Gamma_hat}. In the proof of Lemma~\ref{lem:Gamma2}, we have proved that $\lim_{j_0\to\infty}\hat \Gamma_{(a,b),(a',b')}^{(i,i',j_0)}=\hat \Gamma_{(a,b),(a',b')}^{(i,i',j_0)}$. Hence
\[
\lim_{m\to\infty}\lim_{j_0\to\infty} \cov(\mS_{a,b}^{(i,j_0)(m)},\mS_{a',b'}^{(i',j_0)}) = \lim_{j_0\to\infty} \lim_{m\to\infty} \cov(\mS_{a,b}^{(i,j_0)(m)},\mS_{a',b'}^{(i',j_0)})= \Gamma_{(a,b),(a',b')}^{(i,i')}.
\]
This concludes the proof.
\end{proof}

\subsection{Study of $\mR(j)$}
\label{sec:R}

The following lemma gives the convergence of $\{\mR_{a,b}(j),\, a,b=1,\dots,p,\, j\geq 0\}$ to zero when $j$ goes to infinity.
\begin{lem}
\label{prop:Rij}
$n_{j}^{1/2} \mR_{a,b}(j)=n_{j}^{-1/2} \sum_{i=2^{\Delta-u}}^{T_\psi(2^{\Delta-u}-1)} \mY_{a,b}^{(i,n_{j_1},j_0)}$ goes to zero in probability as $j_0$ goes to infinity.
\end{lem}
\begin{proof}
H\"older's inequality gives 
\[
\E[\lvert \mY_{a,b}^{(i,n_{j_1},j_0)}\rvert]\leq\sum_{a'=1}^p\left(\sum_{t\in\Z} \mV_{a,a'}^{(i,j_0)}(2^{j_1}n_{j_1}-t)^2\right)^{1/2}\left(\sum_{t\in\Z} \mV_{b,a'}^{(i,j_0)}(2^{j_1}n_{j_1}-t)^2\right)^{1/2}\,.
\]
Using Parseval's equality and Lemma \ref{lem:v1}, we get
\begin{align*}
\E[\lvert \mY_{a,b}^{(i,n_{j_1},j_0)}\rvert]& \leq \sum_{a'=1}^p\left(\int_{-\pi}^\pi  \abs{\mV_{a,a'}^{(i,j_0)\ast}(\lambda)}^2\rmd\,\lambda\right)^{1/2}\left(\int_{-\pi}^\pi  \abs{\mV_{b,a'}^{(i,j_0)\ast}(\lambda)}^2\rmd\,\lambda\right)^{1/2}\\ 
&\leq C_v\,\int_\R (1+\lvert\xi\rvert)^{-2\delta_v}\rmd\xi\,.
\end{align*}
Since $\delta_v>1/2$, $\E[\lvert \mY_{a,b}^{({i,n_{j_1},j_0})}\rvert]$ if finite. Thus $\mR_{a,b}(j) = \BigO_\P(n_{j})$ using Markov's inequality.
\end{proof}

\subsection{Proof of Theorem \ref{thm:gauss1}}

\label{sec:proof:thm:gauss1:final}

Equality \eqref{eqn:matriciel} states that for all $a,b=1,\dots,p$,
\begin{multline}
2^{-j_0(d_a+d_b)}\begin{pmatrix}
\sqrt{n_{j_0}}\hat \sigma_{a,b}({j_0})\\ \vdots \\ \sqrt{n_{j_0+\Delta}}\hat \sigma_{a,b}({j_0+\Delta})\end{pmatrix} \\
= \sqrt{n_{j_1}}\,\begin{pmatrix}
\sqrt{n_{j_0}}\\ \vdots \\ \sqrt{n_{j_0+\Delta}}\end{pmatrix}\, \bB_{j_0}\, \begin{pmatrix}
\mS_{a,b}^{({1,j_0})}\\ \vdots \\ \mS_{a,b}^{({N,j_0})}\end{pmatrix} + \begin{pmatrix}
\sqrt{n_{j_0}}\mR_{a,b}({j_0})\\ \vdots \\ \sqrt{n_{j_0+\Delta}}\mR_{a,b}({j_0+\Delta})\end{pmatrix}.
\end{multline}
By Proposition \ref{prop:Rij}, $n_{j}^{1/2}\mR_{a,b}(j)$ goes to 0 in probability when $j$ goes to infinity.

Proposition~\ref{prop:Sij} entails that the first term is asymptotically Gaussian and centered. We now explicit the asymptotic variance.

For all $0\leq u'\leq u \leq \Delta$,  $V_{(a,b),(a',b')}(u,u')=\lim_{j_0\to\infty}\cov\bigl(n_{j_0+u}^{1/2}2^{-j_0(d_a+d_b)}\hat \sigma_{a,b}(j_0+u),n_{j_0+u'}^{1/2}2^{-j_0(d_{a'}+d_{b'})}\hat \sigma_{a',b'}(j_0+u')\bigr)$ satisfies
\[ 
V_{(a,b),(a',b')}(u,u') =\lim_{j_0\to\infty}
n_{j_1} n_{j_0+u}^{1/2}n_{j_0+u'}^{1/2}\bigl(\bB_{j_0} \cov\bigl(\mS_{a,b}^{(i,j_0)},\mS_{a',b'}^{(i,j_0)}\bigr)_{i,i'=1,\dots,N}\bB_{j_0}^T\bigr)_{u,u'}.
\]
Replacing $\bB_{j_0}$ by its expression, the equation above can be reformulated as
\begin{equation}
V_{(a,b),(a',b')}(u,u')
= \lim_{j_0\to\infty} n_{j_1}n_{j_0+u}^{-1/2}n_{j_0+u'}^{-1/2}\sum_{i=2^{\Delta-u}}^{2^{\Delta-u+1}-1}\sum_{i'=2^{\Delta-u'}}^{2^{\Delta-u'+1}-1}\cov\bigl(\mS_{a,b}^{(i,j_0)},\mS_{a',b'}^{(i,j_0)}\bigr)_{i,i'=1,\dots,N}.
\end{equation}
Using the fact that $n_{j_1} n_{j_0+u}^{-1/2}n_{j_0+u'}^{-1/2}\sim 2^{-\Delta+u/2+u'/2}$, and using Proposition~\ref{prop:Sij}, it results that 
\begin{equation}
V_{(a,b),(a',b')}(u,u')
=2^{-\Delta+u/2+u'/2}\sum_{i=2^{\Delta-u}}^{2^{\Delta-u+1}-1}\sum_{i'=2^{\Delta-u'}}^{2^{\Delta-u'+1}-1}\Gamma_{i,i'}((a,b),(a',b')),
\end{equation}
with $\tilde\bGamma((a,b),(a',b'))=\bigl(\Gamma_{(a,b),(a',b')}(i,i')\bigr)_{i,i'=1,\dots, N}$ defined in \eqref{eqn:Gamma}.
We deduce from Lemma~\ref{lem:Gamma_final} that
\[
V_{(a,b),(a',b')}(u,u')=2^{-(u-u')/2}\,\bigl(\GIG_{(a,a'),(b,b')}(u)+\GIG_{(a,b'),(a',b)}(u)\bigr).
\]

\section{Proof of Corollary \ref{cor:cor}}
\label{proof:cor:gauss}
 
Let $(a,b)\in\{1,\dots,p\}^2$ and $j\geq 0$.
We can write the correlation $\hat\rho_{a,b}(j)$ as $\hat \rho_{a,b}(j)=g(2^{-j(d_{a}+d_{b})}\hat\sigma_{a,b}(j),2^{-j2d_{a}}\hat\sigma_{a,a}(j),2^{-j2d_{b}}\hat \sigma_{b,b}(j))$ with  $g(x_1,x_2,x_3)=x_1/\sqrt{x_2x_3}$. The vector $(\hat\sigma_{a,a}(j), \hat\sigma_{b,b}(j), \hat\sigma_{a,b}(j))$ is asymptotically Gaussian by Theorem \ref{thm:gauss1}.
By Delta method \citep[Theorem 4.2.3.]{Anderson}, %\cite[Section 4.1]{Gut},
 we deduce that 
$\sqrt{n_j}\,\left(\hat \rho_{a,b}(j)-r_{a,b}\right)$ converges to a centered Gaussian distribution when $j$ goes to infinity and that its asymptotic covariance satisfies:
\begin{align*}
\lefteqn{\lim_{j_0\to\infty}\cov(\hat\rho_{a,b}(j_0+u),\hat\rho_{a',b'}(j_0+u'))}\\ & = \frac{1}{\sqrt{G_{a,a}G_{b,b}}}\Bigl(\frac{1}{\sqrt{G_{a',a'}G_{b',b'}}} V_{(a,b),(a',b')}(u,u')-\frac{r_{a',b'}}{2G_{a',a'}} V_{(a,b),(a',a')}(u,u')\\
&\qquad\qquad\qquad\qquad\qquad\qquad\qquad\qquad\qquad\qquad\qquad\qquad{-\frac{r_{a',b'}}{2G_{b',b'}} V_{(a,b),(b',b')}(u,u') \Bigr)}\\*
& \quad- \frac{r_{a,b}}{2G_{a,a}}\Bigl(\frac{1}{\sqrt{G_{a',a'}G_{b',b'}}} V_{(a',a'),(a',b')}(u,u')-\frac{r_{a',b'}}{2G_{a',a'}} V_{(a,a),(a',a')}(u,u')\\
&\pushright{-\frac{r_{a',b'}}{2G_{b',b'}} V_{(a,a),(b',b')}(u,u') \Bigr)}\\
& \quad- \frac{r_{a,b}}{2G_{b,b}}\Bigl(\frac{1}{\sqrt{G_{a',a'}G_{b',b'}}} V_{(b,b),(a',b')}(u,u')-\frac{r_{a',b'}}{2G_{a',a'}} V_{(b,b),(a',a')}(u,u')\\
&\pushright{-\frac{r_{a',b'}}{2G_{b',b'}} V_{(b,b),(b',b')}(u,u') \Bigr).}
\end{align*}

We deduce first that 
\begin{align*}
\lefteqn{\lim_{j\to\infty}\var(\hat\rho_{a,b}(j))}\\ & = \frac{1}{G_{a,a}G_{b,b}}\left( V_{(a,b),(a,b)}(0,0)-\frac{G_{a,b}}{2G_{a,a}} V_{(a,b),(a,a)}(0,0)-\frac{G_{a,b}}{2G_{b,b}} V_{(a,b),(b,b)}(0,0) \right)\\
& \quad - \frac{G_{ab}}{2G_{a,a}^2G_{b,b}}\left( V_{(a,b),(a,a)}(0,0)-\frac{G_{a,b}}{2G_{a,a}} V_{(a,a),(a,a)}(0,0)-\frac{G_{a,b}}{2G_{b,b}} V_{(a,a),(b,b)}(0,0) \right)\\
& \quad- \frac{G_{ab}}{2G_{a,a}G_{b,b}^2}\left(V_{(a,b),(b,b)}(0,0)-\frac{G_{a,b}}{2G_{a,a}}V_{(a,a),(b,b)}(0,0)-\frac{G_{a,b}}{2G_{b,b}} V_{(b,b),(b,b)}(0,0)\right),
\end{align*}
where we have used that $\bV(u,u)=\bV(0,0)$ for all $u\in\Z$. Replacing also $V_{(a_1,a_2),(a_3,a_4)}(0,0)$ for all $a_1,a_2,a_3,a_4=1,\dots,p$ by its expression given in \eqref{eqn:variance}, we obtain the asymptotic distribution of Corrollary \ref{cor:cor}.

Second, suppose that the off-diagonal entries of $\bG$ are equal to zero. Then,
\[
{\lim_{j_0\to\infty}\cov(\hat\rho_{a,b}(j_0+u),\hat\rho_{a',b'}(j_0+u'))} = \frac{1}{\sqrt{G_{a,a}G_{b,b}G_{a',a'}G_{b',b'}}} V_{(a,b),(a',b')}(u,u').
\]
Replacing $V_{(a,b),(a',b')}(u,u')$ by its expression, it results that the right-hand side is equal to 0 for all $(a,b)\notin\{(a',b'),(b',a')\}$, and is equal to 
$2^{-|u-u'|/2}2^{-j(2d_a+2d_b)}\tilde I_{|u-u'|}(2d_a,2d_b)$ else.

\section{Additional results on the sample wavelet covariance}
\label{sec:cvgceS}

The objective of this section is to prove that some linear combinations of sample wavelet covariances may be asymptotically Gaussian. Some conditions are given in the following  proposition. It corresponds to \cite[Theorem 3]{Roueff09asymptotic}. The arguments and the scheme of proof are the same.  They are recalled here since the setting and the notations are slightly different.

\begin{prop} \label{prop:cvgceS}
Suppose assumptions of Theorem~\ref{thm:gauss} hold.
Let $\Delta\in\N\cup\{\infty\}$. Let $\{\bomega(u,j_0), u\in\N, j_0\in\N\}$ be a sequence of $\R^{p\times p}$ such that for all $u\in\N$,  $\bomega(u,j_0)\tend_{j_0\to\infty} \tilde \bomega(u)\in\R$ and 
\begin{equation}\label{eqn:wu}
\sum_{u=0}^\Delta \sup_{j_0\geq 0} \norm{\bomega(u,j_0)}_\infty<\infty.
\end{equation}
Define
\[
\bS(\Delta,j_0) =  \sum_{u=0}^{\Delta} \sqrt{n_{j_0+u}}\bomega(u,j_0) \bLambda_{j_0+u}(\bd)^{-1}(\hat\bSigma(j_0+u)-\bSigma(j_0+u))\bLambda_{j_0+u}(\bd)^{-1}\,.
\]
Then $\vect{\bS(\Delta,j_0)}$ converges in distribution to $\mathcal N_p(0,\bV^{(S)}(\tilde\bomega,\Delta))$ when $j_0$ goes to infinity, with \[
\bV^{(S)}_{(a,b),(a',b')}(\tilde\bomega,\Delta)=\sum_{u,u'=0,\dots,\Delta} 2^{-u(d_a^0+d_b^0+d_{a'}^0+d_{b'}^0)}\tilde \omega_{a,b}(u) V_{(a,b),(a',b')}(u,u') \tilde \omega_{a',b'}(u').
\]
\end{prop}
\begin{proof}
For all $\ell\geq 0$, introduce 
\[
\tilde S_{a,b}(\ell,j_0) =  \sum_{u=0}^{\ell}   \tilde \omega_{a,b}(u) \sqrt{n_{j_0+u}}2^{-(j_0+u)(d_a+d_b)}(\hat\sigma_{a,b}(j_0+u)-\sigma_{a,b}(j_0+u))\,.
\]
By Theorem~\ref{thm:gauss}, for all $0\leq \ell<\infty$, $\tilde S_{a,b}(\ell,j_0)$ is asymptotically Gaussian, with distribution $\mathcal N_p(0,\bV^{(S)}(\tilde\bomega,\ell))$.

We will first establish the result when $\Delta=j_1-j_0$ is finite and next when it is infinite.

\begin{itemize}
\item {$\Delta$ finite.}~\\
Since $\tilde S_{a,b}(\Delta,j_0)$ is asymptotically Gaussian, it is sufficient to prove that $E\abs{S_{a,b}(\Delta,j_0)-\tilde S_{a,b}(\Delta,j_0)}$ goes to 0 when $j_0$ goes to infinity. Using Lemma~\ref{lem:sigma}, we have 
\begin{equation}\label{eqn:Stemp}
E\abs{S_{a,b}(\Delta,j_0)-\tilde S_{a,b}(\Delta,j_0)} \leq C_\sigma \sum_{u=0}^{\Delta}   \abs{\omega_{a,b}(u,j_0) -\tilde \omega_{a,b}(u)}.
\end{equation}
We conclude with \eqref{eqn:wu}.

\item {$\Delta$ infinite.}~\\
The convergence of $S_{a,b}(\ell,j_0)$ has been established when $\ell$ is finite and \eqref{eqn:wu}-\eqref{eqn:Stemp} imply that 
\[
\lim_{\ell\to\infty}\lim_{j_0\to\infty}\E\abs{S_{a,b}(\ell,j_0)-S_{a,b}(\ell,j_0)}= 0. 
\]
Hence, it is sufficient to prove that $\lim_{\ell\to\infty}\lim_{j_0\to\infty}\E\abs{S_{a,b}(\infty,j_0)-S_{a,b}(\ell,j_0)}= 0.$
Lemma~\ref{lem:sigma} gives
\[
E\abs{S_{a,b}(\infty,j_0)- S_{a,b}(\ell,j_0)} \\
\leq \sum_{u=\ell+1}^{\infty}   \abs{\omega_{a,b}(u,j_0)}\,.
\]
The convergence is obtained using \eqref{eqn:wu}.
\end{itemize}
\end{proof}

The proof of Proposition \ref{prop:cvgceS} above is based on the following lemma.

\begin{lem}
\label{lem:sigma}
Suppose conditions of Theorem~\ref{thm:gauss1} hold. There exists $C_\sigma$ depending on $\bOmega$, $\bd$, $\phi(\cdot)$, $\psi(\cdot)$, $L$ and $\beta$ such that for all $(a,b)\in\{1,\dots,p\}$,
\[
\E\abs{\hat\sigma_{a,b}(j)-\sigma_{a,b}(j)}\leq C_\sigma \bigl(2^{j(d_a+d_b)}n_j^{-1/2}\bigr)\,.
\]
\end{lem}
\begin{proof}
Since we only consider one scale, suppose momentarily that $j_0=j=j_1$. Based on notations of \ref{sec:indexing}, equation \eqref{eqn:defSR}, and Proposition \ref{prop:Rij},
\[
2^{-j(d_a+d_b)}\bigl({\hat\sigma_{a,b}(j)-\sigma_{a,b}(j)}\bigr)=\frac{1}{\sqrt{n_{j}}}(\mS_{a,b}(1,j)-\E(\mS_{a,b}(1,j)))+O_\P(n_j).
\]
Lemma \ref{lem:Ym} and Lemma \ref{lem:Gamma} imply that $\var\bigl(\mS_{a,b}(1,j)\bigr)$ is finite. Lemma~\ref{lem:sigma} is then straightforward.
\end{proof}

\section{Proof of Theorem \ref{thm:d}}

\label{proof:thmd}
In this section the true parameters are denoted with an exponent 0.

Observe that conditions of  \cite[Proposition 6]{AchardGannaz} are satisfied since we suppose assumption \ref{ass:linear}. Consequently, under assumptions of Theorem~\ref{thm:d}, conditions of Theorem 6 of \cite{AchardGannaz} hold. It entails that 
\begin{align}
\hat\bd-\bd^0&=\BigO_\P(2^{-j_0\beta}+N_X2^{j_0/2}),\\
\forall(a,b)\in\{1,\ldots,p\}^2,\,\hat G_{a,b}(\hat \bd)-G_{a,b}(\bd^0) &= \BigO_\P(\log(N)(2^{-j_0\beta}+ N^{-1/2} 2^{j_0/2})).
\end{align}

The proof of Theorem \ref{thm:d} is based on a Taylor expansion of the objective function. We first recall some useful results obtained \cite{AchardGannaz} in \ref{sec:deriv}, next we give a normality result on the first derivative of the objective function in 
\ref{sec:Sn}. \ref{sec:proof:thmd} finally gives the proof of Theorem  \ref{thm:d}.

\subsection{Some results about the objective function and its derivative}
\label{sec:deriv}

The objective function $R(\cdot)$ is equal to $R(\bd)=\log\det \left(\bLambda_{\mj}(\bd)\hat \bG(\bd) \bLambda_{\mj}(\bd)\right)-1$. It is straightforward that $\hat \bd = \argmin_{\bdexp} R(\bd)$ satisfies 
\begin{eqnarray}
\label{eqn:R2}\hat \bd = \argmin_{\bdexp} \barR(\bd)  &\text{~with~}& \barR(\bd)=\log\det\bbarG(\bd)\\ 
\label{eqn:barG} &\text{~and~}& \bbarG(\bd)=\bLambda_{\mj}(\bd-\bd^0)\hat \bG(\bd) \bLambda_{\mj}(\bd-\bd^0).
\end{eqnarray}
 
The derivatives of the criterion $\barR(\bd)$ are equal to
\begin{align}
\label{eqn:derivR}\frac{\partial \barR(\bd)}{\partial d_a}&= \trace\left(\bbarG(\bd)^{-1} \frac{\partial \bbarG(\bd)}{\partial d_a} \right)\;,\\
\label{eqn:deriv2R}\frac{\partial^2 \barR(\bd)}{\partial d_a\partial d_b}&= - \trace\left(\bbarG(\bd)^{-1} \frac{\partial \bbarG(\bd)}{\partial d_b}\bbarG(\bd)^{-1} \frac{\partial \bbarG(\bd)}{\partial d_a} \right)+\trace\left(\bbarG(\bd)^{-1} \frac{\partial^2 \bbarG(\bd)}{\partial d_a \partial d_b} \right)\;.
\end{align}
when $\bbarG(\bd)^{-1}$ exists.

For any $a=1,\ldots,p$, let $\bi_a$ be a $p\times p$ matrix whose $a$-{\it th} diagonal element is one and all other elements are
zero. Let $a$ and $b$ be two indexes in $1,\ldots,p$. The first derivative of $\bbarG(\bd)$ with respect to $d_a$, $\frac{\partial \bbarG(\bd)}{\partial d_a}$, is equal to
\begin{multline*}
 - \log(2) \frac{1}{n} \sum_{j=j_0}^{j_1} n_j\,(j-\mj)\bLambda_{\mj}(\bd-\bd^0){\bLambda_{j}(\bd)}^{-1}\\(\bi_a \bSigma(j)+\bSigma(j)\bi_a){\bLambda_{j}(\bd)}^{-1}\bLambda_{\mj}(\bd-\bd^0),
 \end{multline*} with $\hat\bSigma(j)=(\hat\sigma_{a',b'}(j))_{a',b'=1,\dots,p}$.
Thus
\begin{equation}
\label{eqn:dbarG}
\left.\frac{\partial \bbarG }{\partial d_a}\right|_{\bdexp^0} = -\frac{\log(2)}{n}\sum_{j=j_0}^{j_1} n_j\,(j-\mj)\,{\bLambda_{j}(\bd)}^{-1}(\bi_a \bSigma(j)+\bSigma(j)\bi_a){\bLambda_{j}(\bd)}^{-1}\,.
\end{equation}

\subsection{Asymptotic normality of the first derivative}
\label{sec:Sn}

The objective of this section is to prove that $\sqrt{n}\,\left.\frac{\partial \barR}{\partial \bd}\right|_{\bdexp^0}$ is asymptotically Gaussian. 

\begin{prop}\label{prop:cvgcedR}
Under assumptions of Theorem \ref{thm:d},
\[
\sqrt{n}\left.\frac{\partial \barR}{\partial \bd}\right|_{\bdexp^0} \tend^{\mathcal L}_{j_0\to\infty} \mathcal N_{p}\left(0,4\log(2)^4\kappa_\Delta^2\, \bV^{\bdexp(\Delta)}\right).
\]
where $\bV^{\bdexp(\Delta)}$ is defined in equation~\eqref{eqn:W}.
\end{prop}
\begin{proof}
For any vector $\bv=(\upsilon_a)_{a=1,\dots,p}\in\R^p$, we want to prove that $\bv^T\left.\frac{\partial \barR}{\partial \bd}\right|_{\bdexp^0}$ converges to a centered Gaussian distribution with variance $\bv^T(4\log(2)^2\kappa_\Delta^2 \bV)\bv$.

By \eqref{eqn:derivR}, \[
\bv^T\frac{\partial \barR(\bd)}{\partial d_a}= \sum_{a=1}^p \upsilon_a \trace\left(\bbarG(\bd)^{-1} \frac{\partial \bbarG(\bd)}{\partial d_a} \right).
\]
As expressed in \cite[page 36]{AchardGannaz}, $\bbarG(\bd)$ can be written as
\[
\barG_{a,b}(\bd^0)=G^0_{a,b}+\sum_{j=j_0}^{j_1}\sum_{k=0}^{n_j} \frac{1}{n} \left(\frac{W_a({j,k})W_b({j,k})}{ 2^{j(d_a^0+d_b^0)}}-G_{a,b}^0\right).
%=\sum_{j=j_0}^{j_1} \frac{1}{n} \left(\frac{n_j\hat\sigma_{a,b}(j)}{2^{j(d_a^0+d_b^0)}}-n_j G_{a,b}^0\right).
\]
Applying \cite[Proposition~8]{AchardGannaz}, we get
\[
\barG_{a,b}(\bd^0)=G^0_{a,b}+\BigO_\P(2^{-j_0\beta}+N_X^{-1/2}2^{-j_1/2})\;.
\]

We introduce
\begin{equation}\label{eqn:Sn}
\tilde S_{j_0}= \sqrt{n}\sum_{a=1}^p \upsilon_a \trace\left(\bG^{0-1}\left.\frac{\partial \bbarG(\bd) }{\partial d_a}\right\rvert_{\bdexp^0}\right)\,.
\end{equation}
It is easily seen that
\begin{equation}\label{eqn:dR}
\sqrt{n}\bv^T\left.\frac{\partial \barR(\bd)}{\partial \bd}\right|_{\bdexp^0} -\tilde S_{j_0} \tend^\P 0.
\end{equation}
Reformulating \eqref{eqn:dbarG}, we get
\[
\left(\left.\frac{\partial \bbarG(\bd)}{\partial d_a}\right|_{\bdexp}\right)_{a,b}=-2\log(2)\frac 1 n \sum_{j=j_0}^{j_1} (j-\mj) 2^{-j(d_a^0+d_b^0)}\,n_j\hat\sigma_{a,b}(j).
\]
Thus we can write $\tilde S_{j_0}$ as
\begin{align} \label{eqn:Sn2}
\tilde S_{j_0} & = \sum_{a,b=1,\dots,p} S_{a,b}^{(\bdexp)}(\Delta,j_0),\\
\nonumber\text{with~~} S_{a,b}^{(\bdexp)}(\Delta,j_0)&=\sum_{u=0}^{\Delta}   \omega^{(S)}_{a,b}(u,j_0) \sqrt{n_{j_0+u}}\,2^{-(j_0+u)(d_a^0+d_b^0)}\,\hat\sigma_{a,b}(j_0+u)\\
\nonumber\text{and~~} \omega_{a,b}^{(S)}(u,j_0) &= -2\log(2) \sqrt{\frac{n_{j_0+u}}{n}}(j_0+u-\mj)\,\upsilon_a (G^{0-1})_{a,b}.
\end{align}
Lemma 13 of \cite{Moulines08Whittle} states that $\mj-j_0\to\eta_\Delta$ when $N_X\to\infty$. When $j_0\to\infty$, $\omega_{a,b}^{(S)}(u,j_0)$ hence converges to $\tilde \omega_{a,b}^{(S)}(u)=-2\log(2) \sqrt{\frac{2^{-u}}{2-2^{-\Delta}}}(u-\eta_\Delta)\,\upsilon_a (G^{0-1})_{a,b}.$ 
Moreover \eqref{eqn:wu} holds for $\Delta\in\N\cup\{\infty\}$.

Applying Proposition \ref{prop:cvgceS}, $\vect{\bS^{(\bdexp)}(\Delta,j_0)-\E(\bS^{(\bdexp)}(\Delta,j_0))}$ is asymptotically Gaussian, with distribution $\mathcal N_p(0,\bV^{(S)}(\tilde\bomega^{(S)},\Delta))$. Consequently, $\tilde S_{j_0}-\E(\tilde S_{j_0})$ follows asymptotically the Gaussian distribution $\mathcal N_p(0,\sum_{a,b,a',b'=1,\dots,p}\bV^{(S)}_{(a,b),(a',b')}(\tilde\bomega^{(S)},\Delta))$.

The end of the proof is divided into two steps. First we prove that $\E(\tilde S_{j_0})$ goes to 0 when $j_0$ goes to infinity, and next we establish that the asymptotic variance above, that is, $\sum_{a,b,a',b'=1,\dots,p}\bV^{(S)}_{(a,b),(a',b')}(\tilde\bomega^{(S)},\Delta)$, is equal to $4\log(2)^2\kappa_\Delta^2\bv^T \bV^{\bdexp(\Delta)}\bv$.

\item 
\paragraph[Asymptotic expectancy]{Convergence of $\E(\tilde S_{j_0})$ toward 0}~\\
Taking the expectancy of $\tilde S_{j_0}$,% in \eqref{eqn:Sn2}, 
\[ \E(\tilde S_{j_0})= -\frac{2\log(2)}{\sqrt{n}}\sum_{j=j_0}^{j_1} n_j(j-\mj)\sum_{a,b=1,\dots,p} \upsilon_a\,(G^{0-1})_{a,b}2^{-j(d_a^0+d_b^0)}\sigma_{a,b}(j).\] 
%\[ \E(\tilde S_{j_0})= \sum_{a,b=1,\dots,p} \sum_{u=0}^\Delta \omega^{(S)}_{a,b}(u,j_0)\sqrt{n_{j_0+u}}2^{-(j_0+u)(d_a^0+d_b^0)}\sigma_{a,b}(j_0+u).\] 
Since $\sum_{j=j_0}^{j_1} n_j(j-\mj)=0$.
\[
\E(\tilde S_{j_0})= -\frac{2\log(2)}{\sqrt{n}}\sum_{j=j_0}^{j_1} n_j(j-\mj)\sum_{a,b=1,\dots,p} \upsilon_a\,(G^{0-1})_{a,b}\bigl(2^{-j(d_a^0+d_b^0)}\sigma_{a,b}(j)-G^0_{a,b}\bigr).
\]
Using Proposition \ref{prop:approx}, \[
\E\lvert \tilde S_{j_0}\rvert =O\bigl(\sqrt{n} j_0 2^{-j_0\,\beta} \bigr). 
\]
Thus $\E(\tilde S_{j_0})$ converges to zero when $j_0^2\,N_X2^{-j_0(1+2\beta)}\to 0$.

\item 
\paragraph[Asymptotic variance]{Expression of $\sum_{a,b,a',b'=1,\dots,p}\bV^{(S)}_{(a,b),(a',b')}(\tilde\bomega^{(S)},\Delta)$}~\\
It remains to prove that $\sum_{a,b,a',b'=1,\dots,p}\bV^{(S)}_{(a,b),(a',b')}(\tilde\bomega^{(S)},\Delta)=4\log(2)^2\kappa_\Delta^2\bv^T \bV^{\bdexp(\Delta)}\bv$. By expanding the expression, 
\begin{align*}
\lefteqn{\sum_{a,b,a',b'=1,\dots,p}\bV^{(S)}_{(a,b),(a',b')}(\tilde\bomega^{(S)},\Delta))}\\
& = \frac{4\log(2)^2}{2-2^{-\Delta}}\sum_{a,b,a',b'=1,\dots,p} \upsilon_a \upsilon_{a'} (G^{0-1})_{a,b}(G^{0-1})_{a',b'} \sum_{u=0}^{\Delta}\sum_{u'=0}^{\Delta} 2^{-u/2-u'/2} (u-\eta_\Delta)(u'-\eta_\Delta)\\
 & \pushright{2^{-u(d_a^0+d_b^0)-u'(d_{a'}^0+d_{b'}^0)}V_{(a,b),(a',b')}(u,u')}\\
& = 4\,{\log(2)^2}\sum_{a,b,a',b'=1,\dots,p} \upsilon_a \upsilon_{a'} (G^{0-1})_{a,b}(G^{0-1})_{a',b'} \left(G_{aa'}^0G_{bb'}^0\mathcal I^{(S)}(d_a^0+d_{a'}^0,d_b^0+d_{b'}^0) \right.\\&\pushright{\left.+ G_{ab'}^0G_{a'b}^0\mathcal I^{(S)}(d_a^0+d_{b'}^0,d_{a'}^0+d_b^0)\right),}
\end{align*}
where
\begin{multline*}
\mathcal I^{(S)}(\delta_1,\delta_2)=\frac{1}{2-2^{-\Delta}}\sum_{u=0}^{\Delta}\sum_{u'=0}^{\Delta} 2^{-u/2-u'/2}(u-\eta_\Delta)(u'-\eta_\Delta)\\
2^{(\delta_1+\delta_2)u\vee u'-|u-u'|/2}2^{-u\delta_1-u'\delta_2}\tilde I_{\lvert u-u'\rvert}(\delta_1,\delta_2) \,.
\end{multline*}
We can formulate this expression to recover a similar form to that of  \cite[Theorem 5]{Moulines08Whittle} and \cite[Theorem 5]{Roueff09asymptotic}. The arguments are the same than those used in the proof of \cite[Proposition 10]{Moulines08Whittle}, but are recalled here to explicit the form of the variance.  We can express $\mathcal I^{(S)}(\delta_1,\delta_2)$ as:
\begin{multline*}
{\mathcal I^{(S)}(\delta_1,\delta_2)=\frac{1}{2-2^{\Delta}}\sum_{\gamma'=0}^\Delta (\gamma'-\eta_\Delta)^2\,2^{-\gamma'}\tilde I_{0}(\delta_1,\delta_2)}\\
 +  \frac{1}{2-2^{\Delta}}\sum_{\gamma=1}^{\Delta}\sum_{\gamma'=0}^{\Delta-\gamma}(\gamma+\gamma'-\eta_\Delta)(\gamma'-\eta_\Delta)2^{-\gamma-\gamma'}{(2^{\gamma\delta_1}+2^{\gamma\delta_2})\tilde I_{\gamma}(\delta_1,\delta_2) \,,}
\end{multline*}
where we set $\gamma=\abs{u- u'}$ and $\gamma'=u\wedge u'$.
We can use the equalities $\sum_{\gamma'=0}^{\Delta-\gamma}(\gamma'-\eta_\Delta)^22^{-\gamma'}=(2-2^{-\Delta+\gamma})(\kappa_{\Delta-\gamma}+(\eta_{\Delta-\gamma}-\eta_\Delta)^2)$ and $\sum_{\gamma'=0}^{\Delta-\gamma}(\gamma'-\eta_\Delta)2^{-\gamma'}=(2-2^{-\Delta+\gamma})(\eta_{\Delta-\gamma}-\eta_\Delta)$.  We obtain
\begin{multline*}
\mathcal I^{(S)}(\delta_1,\delta_2)
=\kappa_\Delta \tilde I_{0}(\delta_1,\delta_2)\\
+  \frac{1}{2-2^{\Delta}}\sum_{\gamma=1}^{\Delta}(2-2^{-\Delta+\gamma})((\gamma+\eta_{\Delta-\gamma}-\eta_\Delta)(\eta_{\Delta-\gamma}-\eta_\Delta)+\kappa_{\Delta-\gamma})\,2^{-\gamma}{(2^{\gamma\delta_1}+2^{\gamma\delta_2})}\tilde I_{\gamma}(\delta_1,\delta_2) \,.
\end{multline*}
We see that when $\Delta < \infty$, $\mathcal I^{(S)}(\delta_1,\delta_2)= \frac{\kappa_\Delta^2}{2}\mathcal I_\Delta(\delta_1,\delta_2)$, with $\mathcal I_\Delta(\delta_1,\delta_2)$ defined in \eqref{eqn:Idelta}. Hence, $\sum_{a,b,a',b'=1,\dots,p}\bV^{(S)}_{(a,b),(a',b')}(\tilde\bomega^{(S)},\Delta)=4\log(2)^2\kappa_\Delta^2\bv^T \bV^{\bdexp(\Delta)}\bv$.

When $\Delta$ goes to infinity, the sequence $\kappa_\Delta$ converges to $2$ and the sequence $\eta_\Delta$ converges to $1$.
%Additionally, for all $u\geq 0$,  \begin{equation}\label{eqn:kappa_inf}
%\frac{1}{\kappa_\Delta}\sum_{i=0}^{\Delta-u}\frac{2^{-i}}{2-2^{-\Delta}}(i-\eta_\Delta)(i+u-\eta_\Delta) \to 1 \text{~when~} \Delta\to\infty\,.
%\end{equation}
 We deduce the asymptotic form \eqref{eqn:Iinf} when $\Delta\to\infty$ by dominated convergence.
\end{proof}

\subsection{Proof of Theorem~\ref{thm:d}}
\label{sec:proof:thmd}

The Taylor expansion of $\frac{\partial \barR(\bd)}{\partial \bd}$ at $\hat \bd$ at the neighborhood of $\bd^0$ gives 
\begin{equation}
\label{eqn:DLd}
\sqrt{n}(\hat\bd -\bd^0)=\left(\left.\frac{\partial^2 \barR}{\partial \bd\partial \bd^T}\right|_{\bar \bd}\right)^{-1} \sqrt{n}\,\left.\frac{\partial \barR}{\partial \bd}\right|_{\bdexp^0},
\end{equation}
 where $\bar{\bd}$ is such that $\|\bar{\bd}-\bd^0\|\leq\|\hat \bd-\bd^0\|$. 

It has already been established in \cite[Equation (E10)]{AchardGannaz} that, under assumptions of Theorem~\ref{thm:d},
\begin{equation}
\label{eqn:dR2} \left.\frac{\partial^2 \barR(\bd)}{\partial \bd \partial \bd^T}\right\lvert_{\bar{\bdexp}}\tend^\P  
\kappa_{j_1-j_0}\log(2)^2 2(\bG^{0-1}\circ \bG^0+\bI_p)\,.
\end{equation}
This justifies also that the matrix $\left( \left.\frac{\partial^2 R}{\partial \bdexp\partial \bdexp^T}\right|_{\bar \bdexp}\right)$ in \eqref{eqn:DLd} is indeed invertible for sufficiently high $N_X$ when $2^{-j_0 \beta}+N_X^{-1/2}2^{j_0/2}\to 0$.

Next Proposition~\ref{prop:cvgcedR} establishes that $\left.\frac{\partial \barR}{\partial \bdexp}\right|_{\bdexp^0}$ converges to a centered Gaussian distribution with variance $(2\log(2)^2\kappa_{j_1-j_0}^2\, \bV^{\bdexp(\Delta)})$. Theorem \ref{thm:d} then follows with \eqref{eqn:DLd} and \eqref{eqn:dR2}.

\section{Proof of Theorem \ref{thm:G}}

\label{proof:thmG}
In this section the true parameters are denoted with an exponent 0.

We have $\hat G_{a,b}(\bd)=2^{\mj(d_a-d_a^0+d_b-d_b^0)}\tilde  G_{a,b}(\bd)$ and $\hat G_{a,b}(\bd^0)=\bar G_{a,b}(\bd^0)$ with $\bbarG(\bd)$ defined in~\eqref{eqn:barG}. As $2^{\mj\, u}-1=j_0\,u\,\log(2)(1+o(1))$ when $u\to 0$, we deduce that \begin{equation} \label{eqn:diffGbar}
\hat G_{a,b}(\hat\bd) - \tilde  G_{a,b}(\hat\bd) =j_0\,(\hat d_a-d_a^0+\hat d_b-d_b^0)\,\log(2)(1+o(1))\bar G_{a,b}(\hat\bd)\,.
\end{equation}
Since $j_0\,(\hat d_a-d_a^0+\hat d_b-d_b^0)=o_{\P}(1)$, it is sufficient to establish the asymptotic distribution of $\barG_{a,b}(\hat\bd)$. More precisely, we want to prove that $\sqrt{n}\,\vect{\bar \bG(\hat\bd)-\bG^0} $ converges in distribution to a centered Gaussian distribution. We decompose $\sqrt{n}\vect{\tilde \bG(\hat\bd)-\bG^0}$ as \begin{equation}\label{eqn:Gdecompose}
\sqrt{n}(\tilde \bG(\hat\bd)-\bG^0)=\sqrt{n}\left(\bar\bG(\bd^0)-\bG^0\right)+\sqrt{n}\left(\tilde \bG(\hat\bd)-\tilde \bG(\bd^0)\right).
\end{equation}
The first term converges to the desired distribution as established in Lemma \ref{lem:Gbarnorm}, while the second one is negligible by Lemma \ref{lem:Gbar0}.

Lemma \ref{lem:Gbarnorm} and Lemma \ref{lem:Gbar0} are given hereafter.

\subsection{Study of $\sqrt{n}\left(\bbarG(\bd^0)-\bG^0\right)$}

\begin{lem}
\label{lem:Gbarnorm}
Under assumptions of Theorem \ref{thm:G},
$\sqrt{n}\vect{\bar\bG(\bd^0)-\bG^0}$ converges as $j_0$ goes to infinity to a centered Gaussian distribution, with variance $\bTheta^{\bGexp(\Delta)}$ defined in \eqref{eqn:WG}.
\end{lem}
\begin{proof}
Consider $\bT_0(j_0)= \vect{\sqrt{n}\left(\bar\bG(\bd^0)-\E(\bar\bG(\bd^0))\right)}$.
Recall that 
\[
\barG_{a,b}(\bd^0)=\frac{1}{n}\sum_{j=j_0}^{j_1} n_j 2^{-j(d_a^0+d_b^0)}\hat\sigma_{a,b}(j)\,.
\]
Using inequality \eqref{eqn:theta},
\[
\E(\barG_{a,b}(\bd^0))=G_{a,b}^0(j)+\BigO\left(2^{j_0\beta}\right)\,,
\]
and consequently $\sqrt{n}\left(\E(\barG_{a,b}(\bd^0))-G_{a,b}^0(j)\right) = o(1)$.

We can write $\sqrt{n}\vect{\bar\bG(\bd^0)-\bG^0}$ as
\begin{multline}\label{eqn:Gw}
\sqrt{n}\left(\barG_{a,b}(\bd^0)-\E(\barG_{a,b}(\bd^0))\right)\\
= \sum_{u=0}^{j_1-j_0} \omega^{(\bGexp)}({u},j_0) \sqrt{n_{j_0+u}}\,2^{-(j_0+u)(d_a^0+d_b^0)}(\hat\sigma_{a,b}(j_0+u)-\sigma_{a,b}(j_0+u)),
\end{multline}
with $\omega^{(\bGexp)}(u,j_0)=\sqrt{n_{j_0+u}/n}$. The sequence $\omega^{(\bGexp)}(u)$ converges to $\tilde \omega^{(\bGexp)}(u)=2^{-u/2}/\sqrt{2-2^{-\Delta}}$ when $j_0$ goes to infinity. Applying Proposition \ref{prop:cvgceS}, we obtain that $\bT_0(j_0)-\E(\bT_0(j_0))$ converges as $n$ goes to infinity to a centered Gaussian distribution, with variance
\begin{align*}
\lefteqn{\lim_{j_0\to\infty}\cov\left({T_{0\,a,b}(j_0),T_{0\,a',b'}(j_0)}\right)}\\&=\sum_{u=0}^\Delta \sum_{u'=0}^\Delta 2^{-u(d_a^0+d_b^0)-u'(d_{a'}^0+d_{b'}^0)}\frac{2^{-u/2-u'/2}}{2-2^{-\Delta}}V_{(a,b),(a',b')}(u,u')\\
&= 2\pi\left(G_{a,a'}^0G_{b,b'}^0{\mathcal I_\Delta^G(d_a+d_{a'},d_b+d_{b'})}+G_{a,b'}^0G_{b,a'}^0\,{\mathcal I_\Delta^G(d_a+d_{b'},d_b+d_{a'})}\right)\,,
\end{align*}
where
\[
\mathcal I_\Delta^G(\delta_1,\delta_2) = 
\sum_{u=0}^\Delta \sum_{u'=0}^\Delta 2^{(u\vee u'-u)\delta_1+(u\vee u'-u')\delta_2}\,\frac{2^{-u/2-u'/2-|u-u'|/2}}{2-2^{-\Delta}}\tilde I_{\abs{u'-u}}(\delta_1,\delta_2).
\]
Quantity $\mathcal I^G(\delta_1,\delta_2)$ can be simplified as:
\begin{align*}
\mathcal I_\Delta^G(\delta_1,\delta_2) & =  \sum_{u=0}^{\Delta} \frac{2^{-u}}{2-2^{-\Delta}}\tilde I_{0}(\delta_1,\delta_2)+
\sum_{u=1}^\Delta (2^{u\delta_1}+2^{u\delta_2})2^{-u}\,\tilde I_{u}(\delta_1,\delta_2) \sum_{v=0}^{\Delta-u} \frac{2^{-v}}{2-2^{-\Delta}}\\
&=\tilde I_{0}(\delta_1,\delta_2)+\sum_{u=1}^\Delta (2^{u\delta_1}+2^{u\delta_2})2^{-u}\,\frac{2-2^{-\Delta+u}}{2-2^{-\Delta}}\,\tilde I_{u}(\delta_1,\delta_2) .
\end{align*}
When $\Delta$ goes to infinity,
\[
\mathcal I_\infty^G(\delta_1,\delta_2) =\tilde I_{0}(\delta_1,\delta_2)+\sum_{u=1}^\infty (2^{u\delta_1}+2^{u\delta_2})2^{-u}\,\tilde I_{u}(\delta_1,\delta_2) .
\]
\end{proof}

\subsection{Study of $\sqrt{n}\left(\bbarG(\hat\bd)-\bbarG(\bd^0)\right)$}

\begin{lem}
\label{lem:Gbar0}
Under assumptions of Theorem \ref{thm:G},
$
\sqrt{n}\left(\bbarG(\hat\bd)-\bbarG(\bd^0)\right)$ tends to 0 in probability when $j_0$ goes to infinity.
\end{lem}
\begin{proof}
A Taylor expansion at order one at $\bd^0$ gives
\begin{equation*}
\sqrt{n}\left(\bbarG(\hat\bd)-\bbarG(\bd^0)\right)=\sqrt{n}\left(\left.\frac{\partial \bbarG }{\partial \bd}\right|_{\bdexp^0}\right)(\hat\bd-\bd^0)+\sqrt{n}(\hat\bd-\bd^0)^T\left(\left.\frac{\partial^2 \bbarG }{\partial \bd\partial \bd^T}\right|_{\bar{\bd}}\right)(\hat\bd-\bd^0),
\end{equation*}
with $\norm{\bar\bd-\bd^0}\leq\norm{\hat\bd-\bd^0}$. Achard and Gannaz in \cite[Section E.2.4.]{AchardGannaz} state that, when $2^{-j_0\beta}+N_X^{-1/2}2^{j_0/2}\to 0$, \[
\left.\frac{\partial^2 \bbarG }{\partial \bd\partial \bd^T}\right|_{\bdexp^0}=\BigO_{\P}(1)\,.
\]
 Thus \[
\sqrt{n}(\hat\bd-\bd^0)^T\left(\left.\frac{\partial^2 \bbarG }{\partial \bd\partial \bd^T}\right|_{\bar{\bd}}\right)(\hat\bd-\bd^0)=o_{\P}(1).
\]

Next, using the derivative of $\bar\bG(\bd)$ given in \eqref{eqn:dbarG}, we have:
\begin{multline*}
\sqrt{n}\left(\left.\frac{\partial \bbarG }{\partial \bd}\right|_{\bdexp^0}(\hat\bd-\bd^0)\right)_{a,b}\\ = -\frac{\log(2)}{\sqrt{n}}\sum_{j=j_0}^{j_1} n_j\,(j-\mj)\,2^{-j(d_a^0+d_b^0)}\,\hat\sigma_{a,b}(j)(\hat d_a-d_a^0+\hat d_b-d_b^0)\,.
\end{multline*}
Similarly to what was done in \ref{proof:thmd}, we can establish that 
$\frac{\log(2)}{\sqrt{n}}\sum_{j=j_0}^{j_1} n_j\,(j-\mj)\,2^{-j(d_a^0+d_b^0)}\,\hat\sigma_{a,b}(j)$ is asymptotically Gaussian. Since $(\hat d_a-d_a^0+\hat d_b-d_b^0)\to 0$ we conclude that this term goes to 0, which concludes the proof.
\end{proof}

\section{Proof of Proposition~\ref{prop:dG}}
\label{proof:thmdG}

In this section the true parameters are denoted with an exponent 0.

From the proof of Theorem~\ref{thm:d} (using equations \eqref{eqn:DLd}, \eqref{eqn:dR2}, \eqref{eqn:dR} and \eqref{eqn:Sn2}) and Theorem~\ref{thm:G} (using equations \eqref{eqn:diffGbar}, \eqref{eqn:Gdecompose}, \eqref{eqn:Gw} and Lemma~\ref{lem:Gbar0}), we can extract the following results
\begin{align*} 
&\pushleft{\left((\bG^{0-1}\circ \bG^0+\bI_p)\sqrt{n}(\hat \bd -\bd^0)\right)_{a'}} \\
&\quad= \sum_{b'=1}^p
{\sum_{u=0}^{j_1-j_0}  \omega_{a',b'}^{(\bdexp)}(u,j_0) \sqrt{n_{j_0+u}}\,2^{-(j_0+u)(d_{a'}^0+d_{b'}^0)}(\hat\sigma_{a',b'}(j)-\sigma_{a',b'}(j))(1+o_{\P}(1))}{+o_{\P}(1),}\\
&\pushleft{\sqrt{n}\left(\hat G_{a,b}(\hat \bd)-G_{a,b}^0\right)}\\ 
&\quad=\sum_{u=0}^{j_1-j_0} \omega_{a,b}^{(\bGexp)}(u,j_0) \sqrt{n_{j_0+u}}\,2^{-(j_0+u)(d_{a'}^0+d_{b'}^0)}(\hat\sigma_{a,b}(j)-\sigma_{a,b}(j))(1+o_{\P}(1))+o_{\P}(1),
\end{align*}
with
\begin{align*}
\omega_{(a',b')}^{(\bdexp)}(u,j_0) &= -\left(\kappa_{j_1-j_0}\log(2)\right)^{-1}\sqrt{n_{j_0+u}/n}\,(j_0+u-\mj)\,(G^{0-1})_{a',b'}\,,\\
\omega_{(a,b)}^{(\bGexp)}(u,j_0) &= \sqrt{n_{j_0+u}/n}.
\end{align*}
Hence, linear combinations of $\sqrt{n}(\hat d_{a'} -d_{a'})$ and $\sqrt{n}\left(\hat G_{a,b}(\hat \bd)-G_{a,b}^0))\right)$ can be written as $\sum_{a'',b''=1,\dots,p} \sum_{u=0}^{j_1-j_0}   \omega_{a'',b''}(u,j_0)\sqrt{n_{j_0+u}}\,2^{-(j_0+u)(d_{a'}^0+d_{b'}^0)} (\hat\sigma_{\ell,m}(j)-\sigma_{\ell,m}(j))(1+o_{\P}(1))+o_{\P}(1)$ with $\omega_{a'',b''}(u,j_0)$ linear combination of $(\omega_{a''',b'''}^{(\bdexp)}(u,j_0))_{a''',b'''=1,\dots,p}$ and  $(\mW_{a''',b'''}^{(\bGexp)}(u,j_0))_{a''',b'''=1,\dots,p}$. Proposition \ref{prop:cvgceS} gives the joint convergence to a Gaussian distribution.

It remains to explicit the asymptotic covariance between $\hat \bd$ and $\hat \bG(\hat\bd)$. Asymptotically, the asymptotic covariance between $\sqrt{n}(\hat d_{a'}-d_{a'}^0)$ and $\sqrt{n}(\hat G_{a,b}(\hat \bd)-G_{a,b}^0)$ is 
\begin{align}
\label{eqn:covdG}
V_{a',(a,b)}^{d,G (\Delta)} =& \left(\kappa_{\Delta}2\log(2)\right)^{-1}(2-2^{-\Delta})^{-1}\sum_{u=0}^\Delta \sum_{u'=0}^\Delta 2^{-u-u'-\abs{u'-u}/2}(j_0+u-\mj)\\ 
\nonumber& 2^{(u\wedge u'-u')(d_a^0+d_b^0)}
\sum_{b'=1}^p \left((\bG^{0-1}\circ \bG^0+\bI_p)^{-1}\right)_{a',b'}2^{(u\wedge u'-u)(d_{a'}^0+d_{b'}^0)}\\
\nonumber&  \left(G_{a',a}^0G_{b',b}^0{\tilde I_{|u-u'|}(d_{a'}+d_{a},d_{b'}+d_{b})}+G_{a',b}^0G_{a,b'}^0\,{\tilde I_{|u-u'|}(d_{a'}+d_{b},d_{b'}+d_{a})}\right)\\
\label{eqn:covdG2}=& \left(4\log(2)\right)^{-1}\sum_{b'=1}^p((\bG^{0-1}\circ \bG^0+\bI_p)^{-1})_{a',b'}\\
\nonumber&  {\left(G_{a',a}^0G_{b',b}^0{\mathcal  I^{d,G}_{\Delta}(d_{a'}+d_{a},d_{b'}+d_{b})}+G_{a',b}^0G_{a,b'}^0\,{\mathcal I^{d,G}_{\Delta}(d_{a'}+d_{b},d_{b'}+d_{a})}\right),}
\end{align}
with 
\begin{align}
\nonumber\mathcal I^{d,G}_{\Delta}(\delta_1,\delta_2)=& \frac{2}{\kappa_{\Delta}}\sum_{u=0}^\Delta \frac{2^{-u}}{2-2^{-\Delta}}(u-\eta_\Delta)\tilde I_{|u-u'|}(\delta_1,\delta_2)\\ 
\nonumber& + \frac{1}{\kappa_{\Delta}}\sum_{u=1}^\Delta 2^{-u}\tilde I_{u}(\delta_1,\delta_2)\sum_{v=0}^{\Delta-u}\frac{2^{-v}}{2-2^{-\Delta}}\left(2^{u\delta_1}(u-\eta_\Delta)+2^{u\delta_2}(u+v-\eta_\Delta)\right)\\
=&\frac{1}{\kappa_{\Delta}}\sum_{u=1}^\Delta 2^{-u}\tilde I_{u}(\delta_1,\delta_2)\left((2^{u\delta_1}+2^{u\delta_2})\frac{2-2^{-\Delta-u}}{2-2^{-\Delta}}\,(u-\eta_\Delta)\,+2^{u\delta_2}\,\eta_{\Delta-u}\right)\,.\\
\label{eqn:covdG3}
\end{align}

\section{Technical lemmas}

\label{proof:technical}

We first recall some inequalities on wavelet filters given in \cite[Proposition 3]{Moulines07SpectralDensity}.
\begin{prop}
Under {\ref{ass:Wcompact}--\ref{ass:Wmoments}}, there exist positive constants $C_{H1}$, $C_{H2}$ and $C_{H3}$ only depending on $\phi$ and $\psi$, such that, for all $j, j'\geq 0$ and $\lambda\in(-\pi,\pi)$,
\begin{align}
\label{eqn:Hj}
\abs{\H_{j}(\lambda)}&\leq C_{H1} 2^{j/2}\abs{2^{j}\lambda}^M(1+2^{j}\abs{\lambda})^{-\alpha-M},\\
\label{eqn:Hjapprox}
\abs{\H_{j}(\lambda)-2^{j/2}\hat\phi(\lambda) \overline{\hat\psi(2^j\lambda)}}&\leq C_{H2} 2^{j/2-j\alpha}\abs{\lambda}^M,
\end{align}
\begin{equation}
\label{eqn:Hj2approx} {\abs{{\H_{j}(\lambda)}\overline{\H_{j'}(\lambda)}-2^{j/2+j'/2}\lvert\hat\phi(\lambda)\rvert^2\, \overline{\hat\psi(2^j\lambda)} {\hat\psi(2^{j'}\lambda)}} \leq C_{H3} 2^{(j+j')(M-\alpha+1/2)}\abs{\lambda}^{2M}\,.}
\end{equation}
\end{prop}
\begin{proof}
Inequalities \eqref{eqn:Hj} and \eqref{eqn:Hjapprox} are proved in \cite[Proposition 3]{Moulines07SpectralDensity}. 
Next,
\begin{align*}
\lefteqn{\abs{{\H_{j}(\lambda)}\overline{\H_{j'}(\lambda)}-2^{j/2+j'/2}\lvert\hat\phi(\lambda)\rvert^2\, \overline{\hat\psi(2^j\lambda)} {\hat\psi(2^{j'}\lambda)}}} \\
& \leq  \abs{{\H_{j}(\lambda)}}\abs{\H_{j'}(\lambda)-2^{j'/2}\hat\phi(\lambda)\, \overline{\hat\psi(2^{j'}\lambda)}}+ \abs{\H_{j'}(\lambda)}\abs{\H_{j}(\lambda)-2^{j/2}\hat\phi(\lambda)\, \overline{\hat\psi(2^{j}\lambda)}}\\
&\qquad+\abs{\H_{j'}(\lambda)-{2^{j'/2}\hat\phi(\lambda)\, \overline{\hat\psi(2^{j'}\lambda)}}}\abs{\H_{j}(\lambda)-2^{j/2}\hat\phi(\lambda)\, \overline{\hat\psi(2^{j}\lambda)}}.
\end{align*}
Applying inequalities \eqref{eqn:Hj} and \eqref{eqn:Hjapprox} to the right-hand side gives \eqref{eqn:Hj2approx}.
\end{proof}

The following lemma is \cite[Lemma 1]{Moulines07SpectralDensity}. It is used in the proofs of Lemma~\ref{lem:cov} and of Lemma \ref{lem:Ym}.
\begin{lem}[\cite{Moulines07SpectralDensity}] \label{lem:diff_sum}
Let $Q\in\N$. For all function $g\in L^2(-\pi,\pi)$, write
\[M_Q(g)=\left(\sum_{q\in\Z}(1-\lvert q\rvert/Q)_+ \left(\int_{-\pi}^\pi g(\lambda)\rme^{-\rmi\,q\,\lambda}\rmd\lambda\right)^2\right)^{1/2}.
\]
Suppose $g_1$ and $g_2$ are $\C$-valued functions of $L^2((-\pi,\pi))$. 
Then, 
\[
\abs{M_Q(g_1) -M_Q(g_2)}^2\leq 2\pi\int_{-\pi}^\pi \abs{g_1(\lambda)-g_2(\lambda)}^2\rmd\lambda.
\]
\end{lem}

Next, the following lemma states the convergence of a series of bivariate Fourier coefficients. It is used in the proofs of Lemma~\ref{lem:Sigma_m} and of Lemma~\ref{lem:Gamma_final}.

\begin{lem} \label{lem:conv_sum}
Suppose $\{w_1^\ast(\lambda),\,\lambda\in(-\pi,\pi)\}$ and $\{w_2^\ast(\lambda),\,\lambda\in(-\pi,\pi)\}$ are $\C$-valued functions of $L^2((-\pi,\pi))$. 
Then
\[\sum_{q\in\Z}(1-\lvert q\rvert/Q)_+ \left(\int_{-\pi}^\pi w_1^\ast(\lambda)\rme^{-\rmi\,q\,\lambda}\rmd\lambda\right) \left(\int_{-\pi}^\pi w_2^\ast(\lambda)\rme^{-\rmi\,q\,\lambda}\rmd\lambda\right) \tend_{Q\to\infty} 2\pi\int_{-\pi}^\pi \overline{w_1^\ast(\lambda)}w_2^\ast(\lambda)\rmd\lambda.\]
\end{lem}
\begin{proof}
Note that \[
\sum_{q\in\Z}(1-\lvert q\rvert/Q)_+ \left(\int_{-\pi}^\pi w_1^\ast(\lambda)\rme^{-\rmi\,q\,\lambda}\rmd\lambda\right) \left(\int_{-\pi}^\pi w_2^\ast(\lambda)\rme^{-\rmi\,q\,\lambda}\rmd\lambda\right)= \sum_{q\in\Z}(1-\lvert q\rvert/Q)_+ c_{1q}\,c_{2q},
\]
with $c_{1q}$ and $c_{2q}$ the $q^\text{th}$ Fourier coefficient respectively of functions $w_1^\ast$ and $w_2^\ast$. 

The sequence $(1-\lvert q\rvert/Q)_+ c_{1q}c_{2q}$ converges to $ c_{1q}c_{2q}$ when $Q$ goes to infinity. By Cauchy-Schwarz's inequality and Parseval's equality, 
\[
\sum_{|q|\leq Q}\lvert c_{1q}\,c_{2q}\rvert\leq 
 2\pi\left(\int_{-\pi}^\pi \abs{w_i^\ast(\lambda)}^2\rmd\lambda\right)^{1/2}\left(\int_{-\pi}^\pi \abs{w_i^\ast(\lambda)}^2\rmd\lambda\right)^{1/2}<\infty.
 \]
Thus dominated convergence entails that the series $\sum_{q\in\Z}(1-\lvert q\rvert/Q)_+ c_{1q}c_{2q}$ converges to $\sum_{q\in\Z}c_{1q}\,c_{2q}$.
By Parseval's theorem, \[
\sum_{q\in\Z}c_{1q}\,c_{2q}= 2\pi \int_{-\pi}^\pi \overline{w_1^\ast(\lambda)}w_2^\ast(\lambda)d\lambda,
\]
which concludes the proof.
\end{proof}

%{Ce qui sert autre: \cite[Lemma 5]{Roueff09central}, \cite[Proposition 2]{Roueff09central}, \cite[Proposition 3]{Roueff09central}, \cite[Theorem 3.2]{Billingsley}}

\section*{Acknowledgments}
The author gratefully acknowledge to the two anonymous referees for their comments that led to substantial improvements in the article.
The author also wishes to express her thanks to Sophie Achard for fruitful exchanges.  Emmanuel Barbier and Guillaume Becq are also acknowledged for providing the data of the resting state fMRI on the rats. 

\section*{References}

\bibliographystyle{elsarticle-num}

\bibliography{biblio}

\end{document}